\documentclass[leqno,10.5pt]{article} 

\usepackage[latin1,utf8]{inputenc}  
\usepackage{lmodern}
\usepackage{verbatim} 
\usepackage{amsmath, amssymb}
\usepackage{amsthm} 
\usepackage{amscd}
\usepackage{color}
\usepackage{mathrsfs}
\usepackage{float}
\usepackage{ulem} 
\usepackage{amsmath,amssymb,amsthm}
\usepackage{amscd}
\usepackage{hyperref} 
\usepackage{enumitem}
\setenumerate[0]{label=(\roman*)} 
\usepackage{soul} 
\usepackage[mathscr]{euscript} 

\theoremstyle{plain} 
\newtheorem{theorem}{\sc Theorem}[section]
\newtheorem{lemma}[theorem]{\sc Lemma}
\newtheorem{corollary}[theorem]{\sc Corollary}
\newtheorem{proposition}[theorem]{\sc Proposition}

\theoremstyle{definition} 
\newtheorem{definition}[theorem]{\sc Definition}
\newtheorem{remark}[theorem]{\sc Remark}

\newcommand\GrO{\mathcal{O}}
\newcommand{\C}{\mathbb{C}}
\newcommand{\R}{\mathbb{R}}
\newcommand{\Q}{\mathbb{Q}}
\newcommand{\Z}{\mathbb{Z}}
\newcommand{\N}{\mathbb{N}}
\newcommand{\Shift}{\tau}
\newcommand{\btau}{\boldsymbol{\tau}}

\newcommand{\AND}{\quad \textrm{and} \quad}
\newcommand{\ba}{\mathbf{a}}

\newcommand{\bI}{\boldsymbol{I}}
\newcommand{\blambda}{\boldsymbol{\lambda}}

\newcommand{\ee}{\varepsilon}
\newcommand{\bx}{\mathbf{x}}
\newcommand{\by}{\mathbf{y}}

\newcommand{\bDelta}{\boldsymbol{\Delta}}
\newcommand{\Eval}{\mathrm{Eval}}

\renewcommand\epsilon\varespilon 
\renewcommand\phi\varphi 
\renewcommand \subset \subseteq

\newcommand{\Span}[2][]{\left\langle #2\right\rangle_{#1}} 

\newcommand\la{\langle}
\newcommand\ra{\rangle}

\newcommand{\bp}{\boldsymbol{p}}
\newcommand{\bq}{\boldsymbol{q}}
\newcommand{\br}{\mathbf{r}}
\newcommand{\bD}{\mathbf{D}}
\newcommand{\bP}{\mathbf{P}}
\newcommand{\bR}{\mathbf{R}}
\newcommand{\cA}{\mathcal{A}}

\newcommand{\cV}{\mathcal{V}}

\newcommand{\cM}{\mathcal{M}}

\newcommand{\fM}{\mathfrak{M}}
\newcommand{\fD}{\mathfrak{D}}
\newcommand{\fR}{\mathfrak{R}}

\newcommand{\M}{\mathfrak{M}^{\rm{Inv}}}

\newcommand{\tA}{\widetilde{A}}

\newcommand{\den}{\mathrm{den}}
\newcommand{\End}{\mathrm{End}}
\newcommand{\Hom}{\mathrm{Hom}}
\newcommand{\lcm}{\mathrm{lcm}}
\newcommand{\ord}{\mathrm{ord}_\infty}

\newcommand{\hMellinInv}{\widehat{\fM}^{\mathrm{Inv}}}
\newcommand{\hp}{\widehat{p}}
\newcommand{\hq}{\widehat{q}}

\newcommand{\Vp}{V_{1}}
\newcommand{\Vm}{V_{-1}}
\newcommand{\Vpm}{V_{\pm 1}}
\newcommand{\cVp}{\cV_{1}}

\newcommand{\Wvar}[1]{W_{#1}}
\newcommand{\Wp}{\Wvar{1}}
\newcommand{\Wm}{\Wvar{-1}}
\newcommand{\Wpm}{\Wvar{\pm 1}}

\newcommand{\bpvar}[1]{\bp_{#1}}
\newcommand{\bpp}{\bpvar{1}}
\newcommand{\bpm}{\bpvar{-1}}
\newcommand{\bppm}{\bpvar{\pm 1}}

\newcommand{\bRvar}[1]{\bR_{#1}}
\newcommand{\bRp}{\bRvar{1}}
\newcommand{\bRm}{\bRvar{-1}}

\title{A linear independence criterion for $p$-adic Hurwitz \\ zeta values at odd integers
}
\author{\textsc{Makoto Kawashima} and \textsc{Anthony Po\"{e}ls}}
\date{ }

\begin{document}

\maketitle

\begin{abstract}
In this article, a new criterion is established to obtain the linear independence of some values of the $p$-adic Hurwitz zeta function at odd integers $s\geq 3$. 
We generalize a theorem of F.~Beukers~\cite[Theorem~11.2]{B}, who gave, among others, an alternative proof---using Pad\'{e} approximation theory---of a result originally due to F.~Calegari~\cite{Cal} on the irrationality of the $p$-adic Riemann zeta value \( \zeta_p(3) \) for \( p = 2, 3 \).
To prove our main theorem, we work in a general setting, where we construct and study new Padé approximants of Laurent series satisfying the functional equation $f(z) = f(1-z)$.
This complements our previous work~\cite{KP2} on the linear independence of $p$-adic polygamma values.
\end{abstract}

\textit{Key words and phrases}: {$p$-adic Hurwitz zeta function, irrationality, linear independence, Pad\'{e} approximation.}

\medskip

\textit{MSC 2020}: {11J72 (primary), 11J61 (secondary).}

\section{Introduction}

The study of arithmetic properties of special values of $L$-functions at positive integers lies at the heart of modern number theory.
In the complex setting, remarkable progress has been achieved over the last two decades.
A landmark result is due to Ball and Rivoal~\cite{B-R}, who proved that infinitely many odd values of the Riemann zeta function are irrational, together with quantitative lower bounds for their linear independence.
Their method, based on carefully constructed Pad\'e approximations combined with symmetry phenomena, has proved extraordinarily fruitful and has inspired a substantial body of subsequent work, including contributions by Zudilin~\cite{Zudilin}, Fischler, Sprang and Zudilin~\cite{FSZ}, Li and Yu~\cite{LaiYuZeta}, and Fischler~\cite{F}.
Very recently, a major breakthrough in a different direction was achieved by Calegari, Dimitrov and Tang~\cite{CDT1}, who established the linear independence of $1$, $\zeta(2)$ and $L(2,\chi_{-3})$ using the framework of arithmetic holonomicity.

By contrast, the $p$-adic setting remains far less understood. Although $p$-adic $L$-functions share many formal analogies with their complex counterparts, proving irrationality or linear independence of their values at positive integers is considerably more delicate.
A first $p$-adic analogue of the Ball--Rivoal phenomenon was obtained by Sprang~\cite{SprangL}, who proved linear independence results for families of $p$-adic $L$-values. This work was subsequently refined and extended in several directions, notably by Li~\cite{Li}, Li, Lupu and Sprang~\cite{Li-Lupu-Sprang}, and Li and Sprang~\cite{Li-Sprang}, leading to increasingly precise $p$-adic linear independence statements.

A major conceptual advance in the study of irrationality of $p$-adic values was achieved by Calegari~\cite{Cal}, who proved the irrationality of certain $p$-adic zeta and $L$-values using $p$-adic modular forms. More recently, Calegari, Dimitrov and Tang~\cite{CDT2} established the irrationality of $\zeta_2(5)$ by means of arithmetic holonomicity, while Lai, Sprang and Zudilin~\cite{LSZ} gave an independent proof of this result by constructing explicit rational approximations to $\zeta_2(5)$.
Shortly thereafter, Beukers~\cite{B} introduced a strikingly different and more elementary approach, relying on Pad\'e approximation techniques and continued fractions in the spirit of Stieltjes. In particular, Beukers proved the irrationality of certain values of the $p$-adic Hurwitz zeta function $\zeta_p(s,x)$ (see Definition~\ref{def: p-adic Hurwitz zeta function} in Section~\ref{section: notation}) for some $x$ and for $s=2,3$.
In~\cite{B1}, P.~Bel obtains similar properties for some $p$-adic functions; the Padé approximants he used are adapted from the article by Rivoal \cite{R3}. This was later generalized by M.~Hirose, N.~Sato and the first author~\cite{HKS}. In~\cite{B2}, Bel also proved the irrationality of {$\zeta_p(4,1/p)$ for any prime $p\ge 19$}. In~\cite{KP2} we established, among others, a linear independence criterion for values of the $p$-adic Hurwitz zeta function at distinct points (with distinct shifts $x$), which improves the results of~\cite{B1}. One corollary of our main result is the following. For any positive integers $m,r$, there exists an explicit positive constant $C_{m,r}$ such that, for any prime number $p\geq C_{m,r}$, the $p$-adic numbers
\begin{align}\label{eq intro: old linear indepdce result}
    1, \zeta_p(2,p^{-r}), \zeta_p(3,p^{-r}), \ldots, \zeta_p(m,p^{-r}) \textrm{ are linearly independent over $\Q$.}
\end{align}
In the special case $r=1$, the first (rough) values of the constants $C_m = C_{m,1}$ are given in the table below.
\begin{table}[H]
   \centering
    \begin{align*}
        \begin{array}{|c|c|c|c|c|c|c|c|c|}
        \hline
        m & 2 & 3 & 4 & 5 & 6 & 7 & 8 & 9 \\
        \hline
        C_m & 5 & 144 & 7\cdot 10^6 & 10^{9} & 3\cdot 10^{11} & 7\cdot 10^{13} & 2\cdot 10^{16} & 8\cdot 10^{18} \\
        \hline
        \end{array}
    \end{align*}
  \caption{If $p\geq C_m$, then \eqref{eq intro: old linear indepdce result} holds with $r=1$.\label{table: previsous article}}
\end{table}

\noindent{\bf Our motivation.} In~\cite{B}, Beukers --whose work was inspired by Calegari's result $\zeta_2(3)$, $\zeta_3(3) \notin \Q$, see ~\cite{Cal}-- shows the irrationality of a large family of $p$-adic numbers, which contains values of $p$-adic $L$-series at $s=2,3$. In his approach, he constructs Padé approximants, which were also described, in a different way, in a paper of Prévost \cite{P} (where an alternative proof of Apéry's result $\zeta(2),\zeta(3)\notin\Q$ is provided), as well as in a paper of Rivoal \cite{R2}, where the author attempted to prove the irrationality of Catalan's constant (in the real setting). The first goal of the present work is to provide a systematic generalization of Beukers' Padé approximants in a more general setting: rather than focusing on isolated examples, we develop a general framework for constructing Pad\'e approximations adapted to several families of Laurent series 
satisfying a functional symmetry of the form
\[
    f(z) =  f(1-z).
\]
Such symmetries arise naturally in the study of the $p$-adic Hurwitz zeta function at positive odd integers, {see for example the Laurent series $T(x)$ studied by Beukers \cite[Corollary 4.5]{B}.} The above functional equation plays a central role in controlling both the algebraic and arithmetic properties of the new Pad\'e approximants that we construct in this paper. As a consequence, we will obtain the linear independence over $\Q$ of several families of $p$-adic numbers involving the values of the $p$-adic Hurwitz zeta function $\zeta_p(s,x)$, in particular its values at odd integers $s$. Although we will not investigate this problem here, we expect that our constructions can be used in a real or complex setting (as is the case for Beukers' ones).\\
The objective behind our current research is to deduce linear independence results of the $p$-adic Riemann zeta values. One problem that we had in mind while working at this project is the following one:

\medskip

\noindent{\bf Problem.} Are $1,\zeta_p(3),\zeta_p(5)$ linearly independent over $\Q$ for $p=2,3$?

\medskip

\noindent Since values of the Kubota--Leopoldt $p$-adic $L$-function at positive integers can be expressed as linear combinations of values of the $p$-adic Hurwitz zeta function at rational points, as shown by J.~Diamond (see~\cite[Theorem~3]{Dia2}), it is natural to try to get strong linear independence results on values of $\zeta_p(s,x)$ in order to achieve this goal. It is not yet clear whether or not our construction will allow us to solve the above problem, although the estimates we get seem rather promising (they correspond to the case $m=6$ of Theorem~\ref{main theorem}). We will investigate this in a future work;  we expect to need the beautiful machinery of Calegari-Dimitrov-Tang arithmetic holonomy bounds \cite{CDT2} combined with finer estimates.

\subsection{Main results}

We need the following definitions in order to state our main results. Set $C_0(4) = C_0(6) = 1$ and
\begin{align}\label{eq: def C(m)}
    C_0(m):=  \left(\frac{(m-1)^2}{2}\right)^{(m-1)/2} \qquad (\textrm{resp. }  C_0(m):= \frac{m}{2}\left(\frac{(m-2)m}{2}\right)^{(m-2)/2})
\end{align}
if $m \geq 3$ is odd (resp. if $m \geq 8$ is even).
For any integer $m \geq 3$, set
\begin{align}\label{eq: def f(m)}
    f(m) = \log C_0(m) + m-1 + \sum_{j=1}^{M-1} \frac{2M}{j}, \quad \textrm{where } M = \left\lfloor \frac{m-1}{2} \right \rfloor.
\end{align}

\begin{theorem} \label{main theorem}
    Let $p$ be a prime number and $m,r$ be positive integers with $m\ge 3$ and {$r\geq r_p$, where $r_p = 1$ (resp. $r_p = 2$) if $p\geq 3$ (resp. $p=2$).} Suppose that
    \[
        f(m)  < 2\left(r+\dfrac{1}{p-1}\right)\log p.
    \]
    \begin{enumerate}
      \item If $m$ is even, then the $m/2$ following $p$-adic numbers
        \[
            1,\zeta_p(3,p^{-r}), \zeta_p(5,p^{-r}), \ldots,\zeta_p(m-1,p^{-r})\textrm{ are linearly independent over $\Q$}.
        \]
      \item If $m$ is odd, then $1$ together with the $(m-1)/2$ following $p$-adic numbers
        \begin{align*}
            \zeta_p\big(s,\tfrac{1}{p^r}+\tfrac{1}{2}\big)-\zeta_p\big(s,\tfrac{1}{p^r}\big) \qquad (s=2,4,6,\dots,m-1)
        \end{align*}
        are linearly independent over $\Q$
    \end{enumerate}
\end{theorem}

A crude order of magnitude in the case $r=r_p$ is given below.

\begin{figure}[H]
   \centering
        \begin{align*}
        \begin{array}{|c|c|c|c|c|c|c|c|c|c|c|}
        \hline
        m & 4 & 6 & 8 & 10 & 12 & 14 & 16 & 18 & 20 & 22 \\
        \hline
        p\ge & {2} & 89 & 7 \cdot 10^{5} & 5 \cdot 10^{8} & 6 \cdot 10^{11} & 9\cdot 10^{14} & 2\cdot 10^{18} & 6\cdot 10^{21} & 2 \cdot 10^{25} & 1 \cdot 10^{29} \\ 
        \hline
        \end{array}
    \end{align*}
    \begin{align*}
        \begin{array}{|c|c|c|c|c|c|c|c|c|c|c|c|}
        \hline
        m & 3 & 5 & 7 & 9 & 11 & 13 & 15 & 17 & 19 & 21 & 23 \\ 
        \hline
        p\ge & 2 & 431 &  10^{5} & 9 \cdot 10^{7} & 9 \cdot 10^{10} &  10^{14} & 3\cdot 10^{17} & 8\cdot 10^{20} & 3 \cdot 10^{24} & 1 \cdot 10^{28} & 8 \cdot 10^{31} \\
        \hline
        \end{array}
    \end{align*}
\label{table: our bound}
\end{figure}

{
We can deduce from the above the following corollary, which was already proved in \cite{B} (and is originally due to Calegari~\cite{Cal}), see Section~\ref{section: Proof of theorem} for details.
\begin{corollary}\label{cor: irrationality zeta_p(3)}
    We have $\zeta_2(3) \notin \Q$ and $\zeta_3(3) \notin \Q$.
\end{corollary}
}

\begin{remark}
    The Laurent series involved in Theorem~\ref{main theorem} satisfy the functional equation
    \begin{align*}
        f(z) = f(1-z),
    \end{align*}
    which is at the heart of our study. Our constructions could be adapted to deal with Laurent series satisfying
    \begin{align*}
        f(z) = -f(1-z).
    \end{align*}
    There does not seem to be any real difficulty in doing so, some dual statements are even established throughout our article. This would, for example, lead to linear independence results of values $\zeta_p(s,x)$ with $s$ \textrm{even}. We did not include this study in our article in order to keep it of reasonable size, and since we were ultimately more interested in odd $p$-adic Riemann zeta values (which are connected to odd values of the $p$-adic Hurwitz Riemann function), rather than its even values.
\end{remark}

\noindent{\bf Our strategy.} The reader will find the definition of $\zeta_p(s,x)$ in Section~\ref{section: notation}. The Laurent series $R_s(z)$ and $\Theta_s(z)$ for which we will construct Padé approximants, and which generalize the Laurent series $T(z)$ and $\Theta(z)$ studied by Beukers, are presented in Subsection~\ref{subsection: theta_s and R_s}. The general strategy of the classical approach using Padé-approximants always need the following steps, which are essentially done in Sections~\ref{section: Construction of Pade} to~\ref{section: Proof of theorem}:
\begin{itemize}
    \item  New construction of a family of Padé approximants (it is done in Section~\ref{section: Construction of Pade}). It is based on the use of (formal) ``umbral transforms'' and the property of a modified (formal) Mellin transform (see Sections~\ref{section: f transform} and~\ref{section: Modified Mellin transform}).
    \item Linearly independence of the aforementioned family of Padé approximants through the non-vanishing of some determinant. We adapt the new method that we developed in \cite{KP2} in order to treat this kind of questions (see Section~\ref{section: det non zero}).
    \item Asymptotic control of the Padé approximants and Padé approximations (with respect to the archimedean and $p$-adic norms), see Section~\ref{section: estimates}. This is based on the use of explicit formulas established in Section~\ref{section: estimates bis}.
    \item Arithmetic control of the denominators (with respect to the archimedean and $p$-adic norms). We combine several ingredients: the integrality of some coefficients appearing in our explicit formulas (this part is particularly delicate), a formula of Gaudron-Rémond for the least common multiple for multinomial coefficients \cite[Theorem 8.1]{gaudron2013minima} and the prime number theorem, see Section~\ref{section: estimates bis}.
\end{itemize}
If the above steps are correctly solved, the linearly independence criterion follows easily, see Section~\ref{section: Proof of theorem}. As we explained earlier in the introduction, we aim at studying Laurent series $f(z)$ satisfying
\begin{align*}
    f(z) = f(1-z),
\end{align*}
which is completely new compared to \cite{KP2}. Section~\ref{involution} is devoted to the study of such Laurent series. By exploiting this property we get substantial improvements compare to previous works.

\begin{remark}[{\bf Comparaison to our previous work \cite{KP2}}]
    Of course, we may deduce from \cite{KP2} the linear independence of
    \begin{align*}
        1, \zeta_p(3,p^{-1}), \zeta_p(5,p^{-1}), \ldots, \zeta_p(2a+1,p^{-1}),
    \end{align*}
    for large $p$. Nevertheless, since we did not particulary focus of odd values in \cite{KP2}, the condition on $p$ is much worse than the one we establish here, especially for the small values of $a$. For example, for $a=2$, we have the condition $p\geq 10^{9}$ coming from \cite{KP2}, compare to our new condition $p\geq 89$.
    Also note that the constructions established in the present paper are different from the ones of \cite{KP2}. Consequently, each of the steps described as above (linear independence, asymptotic estimates, arithmetic estimates, explicit formulas etc.) need to be done from scratch. Furthermore, all of these steps are in some sense more complicated. For example, to prove the linear independence, we manage to follow the same novel strategy developed in \cite{KP2}, however the study of the kernels of the umbral transforms is more intricate because a non-trivial residual term appears (see Section~\ref{section: det non zero} and Remark~\ref{remark: linear indpce, residual term}). \\
    Also, a striking new difficulty concerns the study of the coefficients of the expansion of the Padé approximants established in Section~\ref{section: estimates bis}. Unlike in~\cite{KP2} where there is no real difficulty in proving the integrality of the corresponding coefficients (this was essentially automatic), this is one of the most delicate points in our new constructions. We solve this issue in Section~\ref{section: estimates bis} by establishing several identities based on Zeilberger's algorithm, and then by proceeding by induction (without the integrality property, the estimates for the arithmetic of the denominator --that we believe to be almost optimal -- would have be much worse).\\
    In the case $m=6$ of Theorem~\ref{main theorem}, in order to get better asymptotic estimates, we prove that our Padé approximants satisfy some (matrix) Poincaré-Perron recurrence, which is more complex than a more usual Poincaré-Perron linear recurrence. Unlike in~\cite{KP2}, we are not dealing with an expression which is the sum of a doubly hypergeometric term, so we cannot use (the very convenient) Zeilberger's algorithm to find the recurrence. We found this matrix recurrence thanks to a handwritten Maple program.
\end{remark}

\section{Notation} \label{section: notation}
In this section, we introduce some notation and we give the definition of the $p$-adic Hurwitz zeta function $\zeta_p(s,x)$, as well as some basic properties they satisfy and that we will use later. See \cite[Section 2]{KP2} for more details.
In subsection~\ref{subsection: Pade theory}, we recall some elements of Pad\'{e} approximation theory.

\subsection{$p$-adic Hurwitz zeta function} \label{subsection: p-adic hurwitz zeta}
The floor (resp. ceiling) function is denoted by $\lfloor \cdot \rfloor$ (resp. $\lceil \cdot \rceil$). For any $a\in \Z$, we denote by $\Z_{\leq a}$ the set of integers $n$ with $n \leq a$. We define similarly $\Z_{<a}$, $\Z_{\geq a}$ and $\Z_{>a}$. The rising factorials are the polynomials $(x)_n = x(x+1)\cdots(x+n-1)$ ($n\in\Z_{\geq 0}$), with the convention that $(x)_0 = 1$. Given a (unitary) ring $R$ and an integer $n\geq 1$, we denote by $R^\times$ the unit group of $R$.

\medskip

In the following, we fix a prime number $p$ and we set
\begin{align}
    \label{eq def q_p}
    q_p = \left\{ \begin{array}{ll} p & \textrm{if $p\geq 3$,} \\ 4 & \textrm{if $p=2$.} \end{array}\right.
\end{align}
As usual, $\Q_p$ is the field of $p$-adic numbers, and $\C_p$ is the $p$-adic completion of an algebraic closure of $\Q_p$. We write  $|\cdot |_p:\C_p\longrightarrow \R_{\ge0}$ for the $p$-adic norm with the normalization $|p|_p=p^{-1}$. We denote by $v_p:\C_p\rightarrow \Q \cup \{\infty\}$ the valuation which extends the usual $p$-adic valuation on $\Z$. With this notation, we have $|x|_p = p^{-v_p(x)}$ for each $x\in\C_p$. The ring of $p$-adic integers is the subset $\Z_p=\{x\in \Q_p\,;\, |x|_p \leq 1 \}$, and the group of units of $\Z_p$ is $\Z_p^\times = \{x\in \Q_p\,;\, |x|_p = 1 \}$.

\medskip

\noindent\textbf{Bernoulli numbers and polynomials.} We define the Bernoulli polynomials $B_n(x)$ by their exponential generating function
\begin{align}
    \label{eq: generating function B_k(x)}
    \dfrac{ze^{xz}}{e^z-1}=\sum_{k=0}^{\infty}B_k(x)\dfrac{z^k}{k!},
\end{align}
and the Bernoulli numbers $B_n$ by $B_n = B_n(0)$.
Staudt-Clausen Theorem (see \cite{Cla}) ensures that, for any prime number $p$ and any integer $n\geq 0$, we have $|B_n|_p \leq p$. 

\medskip

\noindent\textbf{Teichm\"{u}ller character.} Recall that $q_p$ is defined in~\eqref{eq def q_p}. We define the Teichm\"{u}ller character $\omega: \Q_p^\times \rightarrow \Q_p^\times$ as follows. For any $x\in\Z_p^\times = \{x\in\Q_p\,;\, |x|_p =1\}$, we denote by $\omega(x)$ the unique $\phi(q_p)$-th root of unity such that
\begin{align*}
    \langle x \rangle := \frac{x}{\omega(x)} \in 1+q_p\Z_p,
\end{align*}
where $\varphi$ is Euler's totient function. Then, given $y\in\Q_p^\times$, we put
\begin{align*}
    \omega(y) = p^{v_p(y)} \omega\Big( p^{-v_p(y)}y\Big) \AND \langle y \rangle  = \frac{y}{\omega(y)} =  \langle p^{-v_p(y)}y \rangle \in 1+q_p\Z_p.
\end{align*}

\begin{remark}
    If $p\geq 3$ and $x\in \Z_p^\times$, then 
    $\omega(x)$ is the unique $(p-1)$-th root of unity that is congruent to $x$ modulo $p$ (see \cite[Theorem~33.4]{Shi}). We easily deduce from the above that
    \begin{align*}
        \omega(p^{-a}+y) = p^{-a} \omega(1+yp^{a}) = p^{-a} \in \Q
    \end{align*}
    for each integer $a\geq 0$ and each $y\in\Q_p$ with $|y|_p < p^a$. In the special case $p=2$, either $x$ or $-x$ is congruent to $1$ modulo $4 = q_p$, and we set $\omega(x) = \pm 1$, so that $x$ is congruent to $\omega(x)$ modulo $4$. See \cite[Section 2]{KP2} for more details.
\end{remark}

\noindent\textbf{The $p$-adic Hurwitz zeta function.} Recall that $q_p$ is defined in~\eqref{eq def q_p}. We follow \cite[Definition~11.2.5]{C} for the definition of the $p$-adic Hurwitz zeta function $\zeta_p(s,x)$.

\begin{definition}
    \label{def: p-adic Hurwitz zeta function}
    For $x\in\Q_p$ and $s\in \C_p\setminus\{1\}$ with $|x|_p\geq q_p$ and  $|s|_p<q_pp^{-1/(p-1)}$, we define $\zeta_p(s,x)$ by the formula
    \begin{align*}
        \zeta_p(s,x) = 
        \frac{\la x\ra^{1-s}}{s-1} \sum_{k\geq 0} \binom{1-s}{k}B_kx^{-k}.
    \end{align*}
\end{definition}
Note that we can also define $\zeta_p(s,x)$ by using a Volkenborn integral, see \cite[Section 2]{KP2} for more details. For a fixed $x\in \Q_p$ with $|x|_p \geq q_p$, the function $s\mapsto \zeta_p(s,x)$ is the unique $p$-adic meromorphic function on $|s|_p<q_pp^{-1/(p-1)}$ satisfying
\begin{align*}
    \zeta_p(1-n,x)=-\omega(x)^{-n}\dfrac{B_n(x)}{n}
\end{align*}
for each integer $n\geq 1$. In addition, this function is ($p$-adic) analytic, except for a simple pole at $s=1$ with residue $1$ (see \cite[Proposition~11.2.8]{C}). Note that for $p\geq 3$ the condition ``$x\in\Q_p$ with $|x|_p \geq q_p$'' simply means that $x\in \Q_p\setminus\Z_p$.

\begin{lemma} {\rm{ \cite[Proposition~11.5.6.]{C}}} \label{expansion zeta_p}
Let $s\ge2$ be an integer and $x\in \Q_p$ with $|x|_p\geq q_p$. Then one has
\begin{equation}
    \label{eq: link zeta_p and derivative of G_p}
    \omega(x)^{1-s}\zeta_p(s,x) = \dfrac{(-1)^{s+1}}{(s-1)!} \sum_{k=s-2}^{\infty}(k-s+3)_{s-2}B_{k-s+2}\cdot\frac{(-1)^{k+1}}{x^{k+1}}.
\end{equation}
\end{lemma}

\subsection{Pad\'{e} approximation theory}\label{subsection: Pade theory}

Fix a field $K$ of characteristic $0$. For any subset $S$ of a $K$-vector space $V$, we denote by $\Span[K]{S} \subset V$ the $K$-vector space generated by the elements of $S$. Given an integer $n\geq 0$, we denote by $K[z]$ the ring of polynomials in $z$ with coefficients in $K$, and by $K[z]_{\leq n} \subset K[z]$ the subgroup of polynomials of degree at most $n$.

\medskip

Let us recall the definition of Pad\'{e}-type approximants of Laurent series and their basic properties. We denote by $K[[1/z]]$ the ring of formal power series ring of variable $1/z$ with coefficients in $K$, and by $K((1/z))$ its field of fractions. We say that an element of $K((1/z))$, which can be written as
\begin{align*}
     \sum_{k = -n}^{\infty} \dfrac{a_k}{z^{k}},
\end{align*}
with $n\in\Z$ and $a_k\in K$, is a \textsl{formal Laurent series} in $1/z$. We define the \textsl{order} function at $z=\infty$ by
\begin{align*}
    \ord :K((1/z)) \longrightarrow \Z\cup \{\infty\}; \ \sum_{k} \dfrac{a_k}{z^k} \mapsto \min\{k\in\Z\cup \{\infty\} \,;\, a_k\neq 0\},
\end{align*}
with the convention $\min \emptyset = \infty$. In particular, for any $f\in K((1/z))$, we have $\ord \, f=\infty$ if and only if $f=0$. Given $\alpha\in K$, we view
\begin{align*}
    (1+z)^\alpha = \sum_{k\geq 0} \binom{\alpha}{k}z^k
\end{align*}
as an element of $K[[z]]$, where $\displaystyle \binom{\alpha}{k} = \alpha(\alpha-1)\cdots(\alpha-k+1)/k!$. However, most of the time, we will consider $1/(z+\alpha)$ as an element of $(1/z)\cdot K[[1/z]]$ by writing
\begin{align*}
    \frac{1}{z+\alpha} = \sum_{k\geq 0} \frac{(-\alpha)^k}{z^{k+1}}.
\end{align*}

\medskip

Given two $K$-vector spaces $V$ and $W$, we denote by $\Hom_K(V,W)$ the $K$-vector space of $K$-linear homomorphisms $V\rightarrow W$. When $V=W$, we write $\End_{K\text{-lin}}(V) = \Hom_K(V,W)$. Similarly, for any $K$-algebra $\cA$, we define $\End_{K\text{-alg}}(\cA)$ as the $K$-vector space of $K$-algebra endomorphisms of $\cA$.

\medskip

We recall, without proof, the following basic result from Pad\'{e} approximation theory.

\begin{lemma} \label{pade}
    Let $m$ be a positive integer, $f_1,\ldots,f_m\in (1/z)\cdot K[[1/z]]$ and $\boldsymbol{n}=(n_1,\ldots,n_m)\in \N^{m}$.
    Put $N=\sum_{j=1}^mn_j$. For any non-negative integer $M\ge N$, there exists a non-zero vector of polynomials $(P,Q_{1},\ldots,Q_m)\in K[z]^{m+1}$ satisfying the following conditions:
    \begin{enumerate}[label= (\roman*)]
      \item   \label{item: lem pade: item 1} $\deg P\le M$,
      \item  \label{item: lem pade: item 2} $\ord \left(P(z)f_j(z)-Q_j(z)\right)\ge n_j+1$ for $j=1,\dots,m$.
    \end{enumerate}
\end{lemma}

\begin{definition}
    With the notation of Lemma~\ref{pade}, fix a non-zero vector of polynomials $(P,Q_{1},\ldots,Q_m) \in K[z]^{m+1}$ satisfying the properties \ref{item: lem pade: item 1} and \ref{item: lem pade: item 2}.
    \begin{itemize}[label=$\bullet$]
      \item We say that $(P,Q_{1},\ldots,Q_m) \in K[z]^{m+1}$ is a \textsl{weight $\boldsymbol{n}$ and degree $M$ Pad\'{e}-type approximant of $(f_1,\ldots,f_m)$}.
      \item  We call the remainders, namely the formal Laurent series  $(P(z)f_j(z)-Q_{j}(z))_{1\le j \le m}$, \textsl{weight $\boldsymbol{n}$ degree $M$ Pad\'{e}-type approximations of $(f_1,\ldots,f_m)$}.
    \end{itemize}
    In the special case $n_1 = \dots = n_m = n$, we simply say that the weight is $n$.
\end{definition}

\section{Formal $f$-umbral transform} \label{section: f transform}

As in \cite[Section 4]{KP2}, the formal $f$-integration transforms will play a key role in the study of our Padé-approximations.

\begin{definition}
    We associate to any Laurent series $f(z) = \sum_{k=0}^{\infty} f_k/z^{k+1}\in (1/z)\cdot K[[1/z]]$, a $K$-linear form $\varphi_f:K[t]\longrightarrow K$, called the \textsl{formal $f$-umbral transform} (or $f$-integration transform), by setting
    \begin{align*} 
        \varphi_f(t^k) = f_k \qquad (k\geq0).
    \end{align*}
\end{definition}

\begin{remark}
    In the case $f(z) = -\log(1-1/z)$, the map $\varphi_f$ is simply the operator $P(t)\mapsto \int_0^1 P(u)du$, which is the reason why we called it $f$-integration transform in previous work. {The way we use these transforms and their properties also gives a strong flavor of umbral calculus, hence the new terminology ``umbral transform''.} Given $f\in (1/z)\cdot K[[1/z]]$, the map $\varphi_f$ extends naturally in a $K[z]$-linear map $\varphi_f: K[z,t]\rightarrow K[z]$. Given a formal series of the form
    \begin{align*}
         F(z,t) = \sum_{k = -n}^{\infty} \dfrac{P_k(t)}{z^{k}},
    \end{align*}
    where $n\in\Z$ and $P_k(t)\in K[t]$ for each $k\geq -n$, we also set
    \begin{align*}
        \varphi_f(F(z,t)) = \sum_{k\geq -n} \frac{\varphi_f(P_k(t))}{z^k} \in z^n K[[1/z]].
    \end{align*}
    Writing $1/(z-t) = \sum_{k\geq 1}t^k/z^{k+1}$, this yields the identity
    \begin{align*}
        f(z)=\varphi_f \left(\dfrac{1}{z-t}\right),
    \end{align*}
    ({\it confer} \cite[(6.2) p.60 and (5.7) p.52]{N-S}). 
    Furthermore, for any $P(z)\in K[z]$,
    \begin{align*}
         P(z)f(z)-\varphi_f\left(\dfrac{P(z)-P(t)}{z-t}\right)\in (1/z)\cdot K[[1/z]].
    \end{align*}
\end{remark}

\noindent\textbf{Notation.} Throughout this article, we will use the following operators. We use bold symbols to indicate when a map is defined on the ring $K[t]$ (in order to distinguish them from their analogues defined on $K[z]$ in Section~\ref{section: Modified Mellin transform}). Fix an element $\alpha\in K$.
\begin{itemize}[label=$\bullet$]
    \item $[P]\in {\rm{End}}_{K\text{-lin}}(K[t])$ denotes the multiplication by the polynomial $P \in K[t]$. If there is no ambiguity, we will sometimes omit the brackets.
    \item We denote by $\Eval_{t=\alpha}\in {\rm{Hom}}_K(K[t],K)$ the $\alpha$-evaluation linear form, and by $\btau_{\alpha}\in {\rm{End}}_{K\text{-alg}}(K[t])$ (resp. $\bDelta_{\alpha} = \btau_\alpha- [1] \in {\rm{End}}_{K\text{-lin}}(K[t])$) the $\alpha$-shift (resp. $\alpha$-difference) operator. They are given by
        \begin{align*}
            \Eval_{t=\alpha}(P(t)) = P(\alpha), \quad \btau_{\alpha}(P(t)) = P(t+\alpha) \AND \bDelta_\alpha(P(t)) = P(t+\alpha)-P(t),
        \end{align*}
        for each $P(t)\in K[t]$.
    \item For any $f(z)\in K((1/z))$, we denote by $\pi(f(z))$ the unique element of $(1/z)\cdot K[[1/z]]$ such that $f(z)-\pi(f(z)) \in K[z]$. This defines a (surjective) projection map
        \begin{align*}
            \pi:K((1/z))\longrightarrow (1/z)\cdot K[[1/z]].
        \end{align*}
    \item The {\it formal adjoint map} $\iota: K[z,\Shift_{\alpha}]\longrightarrow K[t,\btau_{-\alpha}]$ is defined for each $D = \sum_{i}[a_i(z)]\circ \Shift_{\alpha}\in K[z,\Shift_{\alpha}]$ by
        \begin{align*}
            D^* := \iota\Big(\sum_{i}[a_i(z)]\circ \Shift_{\alpha}^i\Big) =  \sum_{i}\btau_{-\alpha}^i\circ [a_i(t)],
        \end{align*}
        (this is a similar operator as the one in \cite[Exercise III(3)]{A}). For any $D\in K[z,\Shift_{\alpha}]$, and any formal series $f(z,t) = \sum_{n\geq 0} P_k(t)/z^n \in K[t][[1/z]]$, we set
        \[
            D(f(z,t)) =  \sum_{n\geq 0} P_k(t)D\left(\frac{1}{z^n}\right) \AND D^*(f(z,t)) =  \sum_{n\geq 0} \frac{D^*(P_k(t))}{z^n}
        \]
        (the first series converges since $\ord(D(1/z^n)) \geq n$ for each $n$).
    \item We define an involution $\bI :K[t]\longrightarrow  K[t]$ by setting for each $P(t)\in K[t]$
        \begin{align*} 
            \bI(P(t)) =  P(1-t).
        \end{align*}
        We have the decomposition $K[t] = \Wp \oplus \Wm$, with the projections
        \[
            \bppm : K[t] \rightarrow \Wpm := \{ P(t) \in K[t] \,;\, P(1-t) = \pm P(t) \}
        \]
        given by $\bppm(P(t)) = \frac{1}{2}(P(t)\pm P(1-t))$. Note that $\Wm = (2t-1)\Wp$.
\end{itemize}


\noindent{\bf A key property of $\phi_f$.} The following results will be important in our computation (see \cite[Proposition~$4.5$]{KP2} and \cite[Corollary~$4.6$]{KP2}).

\begin{proposition} \label{key prop}
    Let $f(z)\in (1/z)\cdot K[[1/z]]$ and $D\in K[z,\Shift_{\alpha}]$. Then, viewing $\varphi_{\pi(D(f))}$ and $\varphi_f\circ D^{*}$ as elements of ${\rm{Hom}}_{K}(K[t],K)$, one has
    \begin{align*}
        \varphi_{\pi(D(f))}=\varphi_f\circ D^{*}.
    \end{align*}
\end{proposition}

\begin{corollary} \label{inc}
    Let $f(z)\in (1/z)\cdot K[[1/z]]$ and $D\in K[z,\Shift_{\alpha}]\setminus\{0\}$. We have the equivalence:
    \begin{align*}
        D(f(z))\in K[z] \Longleftrightarrow D^{*}(K[t])\subseteq \ker \varphi_f.
    \end{align*}
\end{corollary}


\section{Laurent series $R_s$ and $\Theta_s$ and {Mellin-type} transform} \label{section: Modified Mellin transform}

Let $K$ be a field of characteristic $0$. In this section we recall the definition of the modified inverse Mellin transform $\M$ introduced in \cite[Definition 3.5]{KP2}. This will be a key tool to define and study formally the functions $R_s(s)$ and $\Theta_s(z)$ of Definition~\ref{def: Theta_R}, which are the Laurent series at the heart of this paper. The family of Laurent series $R_s(z)$ were first introduced and studied in \cite[Section 5]{KP2} and are closely related to the $p$-adic Hurwitz zeta function $\zeta_p(s,x)$. We will recall some of their properties and establish new ones.

\subsection{Definition of a Mellin-type transform}

\noindent\textbf{Difference and differential operators of $K[[1/z]]$.}
Given $\alpha\in K$ we denote by  $\Shift_{\alpha}$ the $\alpha$-shift operator, and by $\Delta_\alpha = \Shift_\alpha-1$ the $\alpha$-difference operator of $K((1/z))$. Note that they stabilize $K[[1/z]]$ and $(1/z)\cdot K[[1/z]]$. They are defined for each $f(z)\in K((1/z))$ by
\begin{align*}
    \Shift_\alpha \big(f(z)\big) = f(z+\alpha) \AND \Delta_\alpha \big(f(z)\big) = f(z+\alpha)-f(z).
\end{align*}
For any $f(z)\in K((1/z))$ and any sequence $(a_n)_{n\geq 0}$ of elements in $K$, the series
\begin{align*}
    \sum_{n\geq 0}a_n \Delta_\alpha^n \big(f(z)\big) \AND  \sum_{n\geq 0}a_n \frac{d^n}{dz^n} f(z)
\end{align*}
converge in $K((1/z))$ (see \cite[Section 4]{KP2} for the details). Moreover, a quick computation shows that $(z+1)^n = \exp(d/dz)(z^n)$ for any $n\in \Z$, which yields the identity
\[
    \tau_1 = \exp(d/dz).
\]

\begin{definition} 
    The \textsl{modified inverse Mellin transform} $\M_K:K[[z]]\longrightarrow (1/z) \cdot K\left[\left[1/z\right]\right]$ is defined by
    \begin{align*}
        \M_K(g)(z) = \sum_{k=0}^{\infty}b_k\left(-\dfrac{1}{z}\right)^{k+1},
    \end{align*}
    for each  $g(z)\in K[[z]]$, where $(b_k)_{k\ge0} \in K^{\Z_{\geq 0}}$ is such that $g(e^z-1)=\sum_{k=0}^{\infty}b_kz^k/k!$. When there is no ambiguity, we simply write $\M = \M_K$
\end{definition}

\begin{proposition} $(${\rm{\cite[Proposition~$3.8$]{KP2}}}$)$ \label{key prop 0}
    For any $f(z), g(z) \in K[[z]]$, one has
    \begin{align*}
        \M\big(f(z)g(z)\big) = \hMellinInv(f(z))\big(\M(g(z)) \big),
    \end{align*}
    where $\hMellinInv_K: K[[z]] \longrightarrow K[[d/dz]]$ is the isomorphism of $K$-algebra defined by $z \mapsto \Delta_1 = \exp(d/dz)-1$.
    This yields the following commutative diagram:
    \begin{align*}
        \begin{CD}
        K[[z]]\times K[[z]] @>{\hMellinInv\times \M} >> K\left[\left[{d}/{dz}\right]\right]\times (1/z)\cdot K \left[\left[1/z\right]\right]\\
        @VV{ }V @VV{ }V\\
        K[[z]] @>\M >> (1/z)\cdot K\left[\left[1/z\right]\right].
        \end{CD}
    \end{align*}
\end{proposition}

\noindent{\bf  Important example.} In \cite[Lemma~3.1]{KP2}, we proved that for any $\alpha\in K$, we have $\hMellinInv_K((z+1)^\alpha) = (\Delta_1+1)^\alpha = \tau_\alpha$.  Since $\M(1) = -1/z$, for any $g\in K[[z]]$, by Proposition~\ref{key prop 0} we have $\M(g)(z) = \hMellinInv(g)(-1/z)$. Let us fix an integer $k\geq 0$ and choose $g(z) = z^{2k+1}(z+1)^{-k-1}$. By the above, we have
\[
    \M\left(\frac{z^{2k+1}}{(1+z)^{k+1}} \right) = \tau_{-k-1} \circ \Delta_1^{2k+1} \left(\frac{-1}{z} \right).
\]
A quick induction shows that for any integer $n\geq 0$, we have $\Delta_1^n(1/z) = \frac{(-1)^n n!}{z(z+1)\cdots (z+n)}$. Hence
\begin{align}\label{eq: preimage basis of V^+}
    \M\left(\frac{z^{2k+1}}{(1+z)^{k+1}} \right) = \frac{(-1)^{k+1}(2k+1)!}{(z)_{k+1}(1-z)_{k+1}}.
\end{align}

We will need the above identity in Lemma~\ref{image of M}. Using Proposition~$\ref{key prop 0}$, we obtain the following properties.

\begin{lemma} \label{trivial}
Let $g(z)\in K[[z]]$ and $\alpha\in K^\times$. The following identities hold.
 \begin{enumerate}
      \item \label{lem trivial: item 1} We have
        \[
            \M\left(-g\left(\dfrac{-z}{1+z}\right)\right)=\M(g)(-z).
        \]
        More generally, we have $\M(\tilde{g})(z)=\M(g)(z/\alpha)$, where $\tilde{g}(z)=\alpha g((1+z)^{\alpha}-1)$.
      \item \label{lem trivial: item 2}   We have $\M((1+z)^{\alpha}g)(z)=\M(g)(z+\alpha)$.
      \item \label{lem trivial: item 3}  We have $\M\left(-\dfrac{1}{1+z}g\left(\dfrac{-z}{1+z}\right)\right)=\M(g)(1-z)$.
    \end{enumerate}
\end{lemma}

\begin{proof}
    Assertion~\ref{lem trivial: item 1} is obtained by computing $\M(\tilde{g})$ and noticing that $\tilde{g}(e^{z}-1)=\alpha g(e^{\alpha z}-1)$. The second assertion~\ref{lem trivial: item 2} was proved in \cite[Lemma~$5.3$]{KP2}. To get \ref{lem trivial: item 3}, we apply \ref{lem trivial: item 2} with $\alpha=-1$ and $h(z) = -g(-z/(1+z))$, and we conclude thanks to \ref{lem trivial: item 1}.
\end{proof}

\subsection{Laurent series $R_s$ and $\Theta_s$} \label{subsection: theta_s and R_s}

\begin{definition} \label{def: Theta_R}
    For a positive integer $s\ge 2$, set
    \begin{align*}
        &\Theta_s(z)=\M\left(\dfrac{\log(1+z)^{s-1}}{z+2}\right) \AND  R_s(z)=\M\left(\dfrac{\log(1+z)^{s-1}}{z}\right).
    \end{align*}
    We denote the formal umbral transform associated with $\Theta_s$ (resp. $R_s$) by $\varphi_{\Theta_s}$ (resp. $\varphi_{R_s}$).
\end{definition}

\begin{lemma} \label{lem: series expansion Theta_s and R_s}
    Let $s\geq 2$ be an integer.  We have
    \begin{align*}
        R_2(z) = \sum_{k\geq 0} B_k \left(-\frac{1}{z} \right)^{k+1} \AND \Theta_2(z) = \sum_{k\geq 0} (1-2^{k})B_k \left(-\frac{1}{z} \right)^{k+1}.
    \end{align*}
    More generally, for each integer $s\geq 2$, we have
    \begin{align*}
         R_{s}(z) &= \frac{d^{s-2}}{dz^{s-2}}R_2(z) = \sum_{k=s-2}^{\infty}b_k(s)\left(-\frac{1}{z} \right)^{k+1}, \\
         \Theta_{s}(z) &= \frac{d^{s-2}}{dz^{s-2}}\Theta_2(z) = -\frac{1}{2}\sum_{k=s-2}^{\infty}(2^{k-s+3}-2)b_k(s)\left(-\frac{1}{z} \right)^{k+1},
    \end{align*}
    where $b_k(s) = (k-s+3)_{s-2}B_{k-s+2}$ for each $k\geq s-2$.
\end{lemma}

\begin{proof}
    The identity for $R_2(z)$ is given by \cite[Lemma 5.2]{KP2}. Set $g(z) = \log(1+z)/(z+2)$, so that $\Theta_2(z) = \M(g)(z)$. Then, using the generating function $F(z,x) = ze^{xz}/(e^z-1)$ of the Bernoulli polynomials given in \eqref{eq: generating function B_k(x)} together with the identity $B_k(1/2) = (1/2^{k-1}-1)B_k$ for each $k\geq 1$, we find
    \begin{align*}
        g(e^z-1) = \frac{z}{e^z+1} = \frac{1}{2}\Big(F(2z,1/2) - F(2z,0) \Big)  = \sum_{k\geq 0}(B_k(1/2)-B_k)2^{k-1}\frac{z^{k}}{k!}  =  \sum_{k\geq 0}(1-2^{k})B_k\frac{z^{k}}{k!}.
    \end{align*}
    The identity for $\Theta_2(z)$ follows from the above and the definition of $\M$. Finally, according to \cite[Lemma 5.2]{KP2}, for any $g\in K[[z]]$, we have $\M(\log(1+z)g(z)) = \frac{d}{dz}(\M(g)(z))$. This yields the expected formulas for $\Theta_s(z)$ and $R_s(z)$.
\end{proof}

\begin{remark} \label{Remark T,R}
    In \cite{B}, Beukers defines the Laurent series $\Theta(z)$, $R(z)$ and $T(z)$ by
    \begin{align*}
        &\Theta(z)=\sum_{k=0}^{\infty}(2^{k+1}-2)B_k\left(-\dfrac{1}{z}\right)^{k+1}, \quad R(z) = \sum_{k=0}^{\infty}B_k\left(-\dfrac{1}{z}\right)^{k+1} \AND  T(z)=\sum_{k=0}^{\infty}(k+1)B_k\left(-\dfrac{1}{z}\right)^{k+2}.
    \end{align*}
    Lemma~\ref{lem: series expansion Theta_s and R_s} implies that
    \[
        \Theta(z)=-2\Theta_2(z), \quad R(z) = R_2(z) \AND T(z)=R_3(z).
    \]
\end{remark}

We first record some basic identities satisfied by $\Theta_s(z)$ and $R_s(z)$.

\begin{lemma} \label{identities}
Let $s\ge 2$ be a positive integer. Then the following identities hold.
   \begin{enumerate}
      \item \label{assertion identities 1}  $\Theta_2(1-z)=\Theta_2(z)$.
      \item \label{assertion identities 2}  $R_{2}(1-z)=-R_2(z)$.
      \item \label{assertion identities 5}  $\Theta_s(z)= 2^{1-s}\cdot \left(R_s\big((z+1)/2\big) - R_s(z/2)\right)$.
      \item \label{assertion identities new} $(\tau_{1}-1)(R_s(z)) = (\tau_{1}+1)(\Theta_s(z)) =\M(\log(1+z)^{s-1}) = (-1)^s (s-1)!/z^s$.
    \end{enumerate}
\end{lemma}

\begin{proof}
    We first prove \ref{assertion identities new}. The identities $\Theta_s(z) = \M(\log(1+z)^{s-1}/(z+2))$ and $\tau_1+1 = \hMellinInv(z+2)$ together with Proposition~\ref{key prop 0}, lead to $(\tau_{1}+1)(\Theta_s) = \M(\log(1+z)^{s-1})$. Similarly, since $\tau_1-1 = \hMellinInv(z)$, we obtain $(\tau_{1}-1)(R_s) = \M(\log(1+z)^{s-1})$. Finally, $\M(\log(1+z)^{s-1})= (-1)^s (s-1)!/z^s$ by \cite[Proposition~5.4]{KP}. According to \cite[Proposition~4.4]{B} and  \cite[Corollary~4.5]{B}, we have $R(z) = -R(1-z)$ and $\Theta(z) = R(z/2)-R((z+1)/2)$, from which we deduce $\Theta(z) = \Theta(1-z)$. 
    We obtain the remaining assertions by using Lemma~\ref{lem: series expansion Theta_s and R_s} together with {Remark~\ref{Remark T,R}}.
\end{proof}

\noindent{\bf Study of the kernels.}

\medskip

Combining Corollary $\ref{inc}$ with the last assertion of Lemma~\ref{identities}, we obtain the following information on $\ker \varphi_{\Theta_s}$ and $\ker \varphi_{R_s}$.

\begin{corollary} \label{kernel}
    For each integer $s\ge 2$ we have
    \begin{align*}
        (\btau_{-1}+1)(t^sK[t])\subset {\rm{ker}}\,\varphi_{\Theta_s} \AND (\btau_{-1}-1)(t^sK[t])\subset {\rm{ker}}\,\varphi_{R_s}.
    \end{align*}
\end{corollary}

\begin{lemma} \label{lem: lem 8.4 extended}
    Let $m\geq 3$ be an integer. For any integer $s\geq 2$, we have
    \begin{align*}
        \ker \varphi_{\Theta_s} = (\btau_{-1}+1)\big( V_s\big) \AND \ker \varphi_{R_s} = (\btau_{-1}-1)\big(V_s  \big).
    \end{align*}
    where $V_s = \Span[K]{1,t,\dots,t^{s-2},t^{s},t^{s+1},\dots}$. 
\end{lemma}

\begin{proof}
    Proposition~\ref{key prop} applied with $D=(\tau_{1}+1)$ (yields the identity $\varphi_{\Theta_s}\circ (\btau_{-1}+1) = \phi_{(\tau_{1}+1)(\Theta_s)}$. Similarly, we have $\varphi_{R_s}\circ (\btau_{-1}-1) = \phi_{(\tau_{1}-1)(R_s)}$), and Lemma~\ref{identities}~\ref{assertion identities new} implies
    \begin{align*}
       \varphi_{\Theta_s}\circ (\btau_{-1}+1) = \varphi_{R_s}\circ (\btau_{-1}-1).
    \end{align*}
    The identity $\ker \varphi_{R_s} = (\btau_{-1}-1)\big(V_s \big)$ comes from \cite[Lemma~8.4]{KP2} (with $\alpha=0$, and using $1\in \ker(\btau_{-1}-1)$). Thus $(\btau_{-1}+1)\big(V_s\big) \subset \ker \varphi_{\Theta_s}$. This inclusion is actually an equality since $(\btau_{-1}+1)\big(V_s\big)$ is a hyperplane of $K[t]$ (for $\btau_{-1}+1$ is an automorphism of $K[t]$), and $\varphi_{\Theta_s}$ is a non-zero linear form. 
\end{proof}

Define the projections $\bppm :  P(t) \mapsto \frac{1}{2}(P(t)\pm P(1-t))$ as in Section~\ref{section: f transform}. We also introduce the projections $\blambda_{\pm 1} : P(t) \mapsto \frac{1}{2}(P(t)\pm P(-t))$. A direct and short computation ensures that for $\ee=\pm 1$, we have
\begin{align} \label{eq inter: identity op easy}
    (\btau_{-1}+\ee) \circ \blambda_\ee = \bpp \circ (\btau_{-1}+\ee) \AND (\btau_{-1}+\ee) \circ \blambda_{-\ee} = \bpm \circ (\btau_{-1}+\ee).
\end{align}
Lemma~\ref{lem: lem 8.4 extended} has the following consequence.

\begin{corollary}\label{cor: expression restricted intersection ker phi}
    Let $m\geq 3$ be an integer. Set $\delta_m = \dim \ker (\btau_{-1}+(-1)^{m+1})$. If $m$ is even (resp. odd) we denote by $S_m$ the set of odd (resp. even) indices in $\{2,\dots,m-1\}$, and by $F_s$ the Laurent series $R_s$ (resp. $\Theta_s$) for each $s\in S_m$. Then
    \begin{align*}
        \bigcap_{s\in S_m} \ker \varphi_{F_s} = \bpp\big(K[t]_{\leq m-1-\delta_m}\big) + (\btau_{-1}+(-1)^{m+1})\big(t^m K[t]\big).
    \end{align*}
\end{corollary}

\begin{proof}
    Fix an integer $m\geq 3$. We first suppose that $m$ is odd. Using Lemma~\ref{lem: lem 8.4 extended} and the fact that $\btau_{-1}+1$ is an automorphism of $K[t]$, we obtain
    \begin{align*}
        \bigcap_{s\in S_m} \ker \varphi_{\Theta_s} =  (\btau_{-1}+1)\Big( \bigcap_{s\in S_m} V_s \Big) = (\btau_{-1}+1)\big(V\big) + (\btau_{-1}+1)\big(t^m K[t]\big),
    \end{align*}
    where $V = \Span[K]{1,t^2,\dots,t^{m-1}} = \blambda_1(K[t]_{\leq m-1})$. The conclusion follows from the first identity of \eqref{eq inter: identity op easy} (with $\ee=1$) together with $(\btau_{-1}+1)(K[t]_{\leq m-1}) = K[t]_{\leq m-1}$.\\
    Suppose now that $m\geq 3$ is even. Let us first prove that
    \begin{align} \label{eq inter: computation small inter for Ker phi_R}
        W_m:=  \bigcap_{s\in S_m} (\btau_{-1}-1)(V_s) = (\btau_{-1}-1)\Big(tK[t] \cap \bigcap_{s\in S_m} V_s \Big).
    \end{align}
    The inclusion $\supseteq$ is clear. To prove the direct inclusion, fix $P(t) \in W_m$. Since $(\btau_{-1}-1)(K) = \{0\}$, for each $s\in S_m$, there exists $Q_s\in V_s \cap tK[t]$ such that $P = (\btau_{-1}-1)(Q_s)$. Given $s,s' \in S_m$, we deduce that $Q_s-Q_{s'} \in \ker (\btau_{-1}-1) \cap tK[t] = \{0\}$, and therefore $Q_s = Q_{s'}$. Consequently, there exists $Q(t)\in tK[t]$ such that  $P(t) = (\btau_{-1}-1)(Q(t))$ and $Q(t)\in V_3\cap V_5\cap \cdots \cap V_{m-1}$, hence \eqref{eq inter: computation small inter for Ker phi_R}. Note that $\ker \varphi_{R_3} \cap \ker \varphi_{R_5} \cap \cdots \cap \ker \varphi_{R_{m-1}}$ is equal to $W_m$ by Lemma~\ref{lem: lem 8.4 extended}. Finally, we have
    \begin{align*}
        tK[t] \cap  V_3\cap V_5\cap \cdots \cap V_{m-1} = V + t^m K[t],
    \end{align*}
    where $V = \Span[K]{t,t^3,\dots,t^{m-1}} = \blambda_{-1}(K[t]_{\leq m-1})$. The space $(\btau_{-1}-1)(V)$ is equal to $\bpp \circ (\btau_{-1}-1)(K[t]_{\leq m-1})$ thanks to the first identity of \eqref{eq inter: identity op easy}. We conclude easily by noticing that $(\btau_{-1}-1)(K[t]_{\leq m-1}) = K[t]_{\leq m-2}$.
\end{proof}


\section{Expansion of Laurent series with some symmetries}
\label{involution}


In \cite{KP2}, we constructed Padé approximants of the tuple $(R_2,R_3,\dots,R_m)$. In this article, the fundamental difference is that we restrict ourselves to $R_s$ with $s$ odd, which make the study of the Padé approximants significantly more delicate. On the other hand, we will be able to use the symmetry $R_s(z) = R_s(1-z)$ (valid for odd indices $s$) to construct new Padé approximants and to ultimately get better estimates. Recall that we see $\dfrac{1}{1-z} = -\dfrac{1}{z}\cdot\dfrac{1}{1-1/z}$ as an element of $(1/z)\cdot K[[1/z]]$, which allows us to define the involution $I:K((1/z))\longrightarrow K((1/z))$ by
\[
    I(f(z)) = f(1-z) \qquad \textrm{for each } f(z)\in K((1/z)).
\]
Then $I$ has two eigenvalues $\pm 1$. We denote by $\Vp$ and $\Vm$ the associated eigenspaces
\begin{align*}
    \Vp=\{f(z)\in K((1/z)) \,;\, I(f)=f\} \AND  \Vm=\{f(z)\in K((1/z)) \,;\, I(f)=-f\}.
\end{align*}
The involution $I$ is an order preserving homomorphism, i.e. ${\rm{ord}}_{\infty}(f)={\rm{ord}}_{\infty}(I(f))$. Note that $\Vp$ is closed under product, and that the order ${\rm{ord}}_{\infty}(f)$ of any $f(z)\in \Vp$ is even. Moreover, Lemma~\ref{identities} \ref{assertion identities 1} and \ref{assertion identities 2} imply that
\begin{align} \label{V^{+},V^{-}}
    \Theta_{2s}(z),R_{2s+1}(z)\in \Vp \AND \Theta_{2s+1}(z),R_{2s}(z)\in \Vm \ \ \text{for each $s\in \Z_{\ge1}$.}
\end{align}

The goal of this section is to get the following expansion of $R_s$ (resp. $\Theta_s$) when $s$ is odd (resp. even).
\begin{proposition} \label{key expansion R_s and Theta_s}
    Let $t\geq  0$ be an integer. We have
    \begin{align*}
        R_{3+2t}(z) = (-1)^{t+1}\sum_{k=t}^{\infty}\dfrac{b_k(2k+1)!}{(z)_{k+1}(1-z)_{k+1}}, \quad \textrm{where } b_k = \sum_{\substack{ j_0,\dots,j_t \geq 0 \\ j_0+\cdots + j_t = k-t}} \prod_{i=0}^{t} \frac{j_i!^2}{(j_i+1)(2j_i+1)!}, \\
        2\Theta_{2+2t}(z) = (-1)^{t+1}\sum_{k=t}^{\infty}\dfrac{c_k(2k+1)!}{(z)_{k+1}(1-z)_{k+1}}, \quad \textrm{where } c_k = \sum_{\substack{ j_0,\dots,j_t \geq 0 \\ j_0+\cdots + j_t = k-t}}(j_t+1) \prod_{i=0}^{t} \frac{j_i!^2}{(j_i+1)(2j_i+1)!}.
    \end{align*}
\end{proposition}
This will be a consequence of the following general result, which provides a new expression of Laurent series of the form $f(z) = \M(g(z))$, when $f(z)\in \Vp$. 
\begin{proposition} \label{explicit expansion}
    Let $g(z)\in K[[z]]$ and set $f(z)=\M(g)(z)$. Suppose that $f(z)\in \Vp$. Then there exists (a unique) sequence $(a_k)_{k\geq 0}\in K^{\Z_{\geq 0}}$ such that
    \begin{align*} 
        g(z) = \sum_{k=0}^{\infty}(-1)^{k+1} a_k\dfrac{z^{2k+1}}{(1+z)^{k+1}}  \AND f(z)=\sum_{k=0}^{\infty}\dfrac{a_k(2k+1)!}{(z)_{k+1}(1-z)_{k+1}}.
    \end{align*}
\end{proposition}
The two above propositions are proved in Subsection~\ref{subsection: proof expansion R_s and Theta_s}. In the last Subsection~\ref{subsection: pade approximants Laurent series with symmetries} we establish some general properties about Padé approximants of Laurent series $F(z)$ which satisfy $F(z) = F(1-z)$.

\subsection{Structure of the eigenspaces} \label{subsection: structure eigenspace}

\begin{lemma}  \label{poly eigen1}
    We have the identity $\Vm \cap K[z] = (2z-1)\Vp\cap K[z]$, and $\big((z+\alpha)_k(1+\alpha-z)_k\big)_{k\geq 0}$  is a basis of the space $\Vp\cap K[z]$ for each $\alpha\in K$.
\end{lemma}

\begin{proof}
    Clearly $W := \Span[K]{(z+\alpha)_k(1+\alpha-z)_k\,;\, k\geq 0} \subset \Vp \cap K[z]$. Now, let $P(z)$ be a non-zero element of $\Vp \cap K[z]$. Then the degree of $P$ must be an even integer $2n$. Denote by $a$ its leading coefficient and set $Q(z) = P(z) - a(z+\alpha)_{n} (1 + \alpha - z)_{n}$. Then $Q(z)\in \Vp \cap K[z]$ and $\deg Q(z) < \deg P(z)$. We conclude by an easy induction that $P(z)\in \Vp \cap K[z]$.
\end{proof}


\begin{lemma} \label{ans prob 1.2}
    Define $U_{\pm 1} = \Vpm \cap (1/z)\cdot K[[1/z]]$. Then, $(2z-1)U_{-1} = K + U_1$ and
    \[
        U_1 = \left\{\sum_{k=1}^{\infty}\dfrac{a_k}{(z)_k (1-z)_k} \,;\, \forall k\geq 1, a_k\in K \right\}.
    \]
\end{lemma}

\begin{proof}
    Let $U^+$ denote the right-hand side of the above equality. Clearly, $U^+ \subseteq U_1$. Now, let $f(z) \in U_1$. If $f(z)\neq 0$, then its order must be even. So, there exists $N\geq 1$ and $a_{N}\in K^\times$ such that $f(z) \in a_N/z^{2N} + (1/z^{2N+1})\cdot K[[1/z]]$. This implies that
    \begin{align*}
        r_N(z) := f(z) - \frac{a_N}{(z)_N(1-z)_N} \in \Vp \cap \frac{1}{z^{2N+1}}K[[1/z]] = \Vp \cap \frac{1}{z^{2N+2}}K[[1/z]].
    \end{align*}
    By repeating the process, we construct by induction a sequence $(a_k)_{k\geq N}$ in $K$ such that
    \begin{align*}
        r_M(z) := f(z) - \sum_{k=N}^M\frac{a_k}{(z)_k(1-z)_k} \in \Vp \cap \frac{1}{z^{2M+2}}K[[1/z]]
    \end{align*}
    tends to $0$ as $n$ tends to $\infty$ (with respect to the $(1/z)$-adic topology on $K[[1/z]]$), hence $f(z)\in U^+$.
\end{proof}

\subsection{Proof of Propositions~\ref{key expansion R_s and Theta_s} and~\ref{explicit expansion}} \label{subsection: proof expansion R_s and Theta_s}

\noindent{\bf Proof of Proposition~\ref{explicit expansion}.} Lemma~\ref{trivial}~\ref{lem trivial: item 3} implies that $(\M)^{-1}\left(\Vp\cap (1/z)\cdot K[[1/z]]\right)$ is equal to the space
\[
    \cVp=\biggl\{g(z)\in K[[z]] \,;\, g(z)=-\dfrac{1}{1+z}g\left(-\dfrac{z}{1+z}\right)\biggr\},
\]
in particular the modified inverse Mellin transform $\M$ induces a $K$-isomorphism $\M:\cVp\longrightarrow \Vp\cap (1/z)\cdot K[[1/z]]$. In order to study the structure of $\cVp$, and in view of Lemma~\ref{ans prob 1.2}, we need to compute the preimages of the elements $1/((z)_{k+1}(1-z)_{k+1})$. The next result is equivalent to \eqref{eq: preimage basis of V^+}.

\begin{lemma}\label{image of M}
    For an integer $k\ge0$, we have
    \[
        \M\left((-1)^{k+1}\dfrac{z^{2k+1}}{(1+z)^{k+1}}\right)=\dfrac{(2k+1)!}{(z)_{k+1}(1-z)_{k+1}}.
    \]
\end{lemma}

Let $g(z)\in \cVp$ and set $f(z)=\M(g)(z)$. By combining Lemmas~\ref{ans prob 1.2} and~\ref{image of M}, there exists a sequence $(a_k)_{k\geq 0}$ of elements in $K$, such that
\[
    f(z)=\sum_{k=0}^{\infty}\dfrac{a_k(2k+1)!}{(z)_{k+1}(1-z)_{k+1}} \AND g(z)=\sum_{k=0}^{\infty}(-1)^{k+1}a_k\dfrac{z^{2k+1}}{(1+z)^{k+1}}.
\]

\medskip

\noindent{\bf Proof of Proposition~\ref{key expansion R_s and Theta_s}.} Recall that by definition (see Section~\ref{subsection: theta_s and R_s}), for each integer $t \geq 0$, we have
\begin{align*}
        R_{3+2t}(z)=\M\left(\dfrac{\log(1+z)^{2+2t}}{z}\right) \AND  \Theta_{2+2t}(z)=\M\left(\dfrac{\log(1+z)^{1+2t}}{z+2}\right).
    \end{align*}
Using the following identities,
\begin{align*}
    \log(1+z)^2 = -\sum_{k\geq 0} \frac{a_k}{(k+1)} Z^{k+1} \AND \frac{\log(1+z)}{z+2} = -\frac{1}{2z}\sum_{k\geq 0}a_k  Z^{k+1},
\end{align*}
where $a_k = k!^2/(2k+1)!$ and $Z = -z^2/(1+z)$, it follows that
\begin{align}
    \dfrac{\log(1+z)^{2+2t}}{z} & = (-1)^{t+1}\sum_{k\geq t}(-1)^{k+1} b_k \dfrac{z^{2k+1}}{(1+z)^{k+1}}, \label{eq: expansion g_s m even} \\
    \dfrac{\log(1+z)^{1+2t}}{z+2}  & = \frac{(-1)^{t+1}}{2}\sum_{k\geq t} (-1)^{k+1} c_k \dfrac{z^{2k+1}}{(1+z)^{k+1}}, \label{eq: expansion g_s m odd}
\end{align}
{where the coefficients $b_k$ and $c_k$ are defined in Proposition~\ref{key expansion R_s and Theta_s}. 
We conclude by Lemma~\ref{image of M}.

\subsection{Padé approximants of Laurent series in $\Vp$}\label{subsection: pade approximants Laurent series with symmetries}

In Section~\ref{section: Construction of Pade}, we will construct Padé approximants for tuples of functions as above, which belong to the space $\Vp$ and are of the form $\M(g_s(z))$ (see Proposition~$\ref{linear relations}$). The following result provides a useful description of their Padé approximants. Following the notation of Section~\ref{section: f transform}, we denote by $\Wp$ the subspace
\[
    \Wp = \{P \in K[t] \,;\, P(t) = P(1-t) \}.
\]

\begin{proposition} \label{linear relations}
    Let $m,n$ be positive integers and $(f_1,\ldots,f_m)\in (\Vp\cap (1/z)\cdot K[[1/z]])^m$. Let $P(t)\in \Wp$ with ${\rm{deg}}\,P\ge 2mn$, and define
    \[
        Q_s(z)=\varphi_{f_s}\left(\frac{P(z)-P(t)}{z-t}\right) \quad \textrm{for $1\le s \le m$.}
    \]
    The followings assertions are equivalent.
    \begin{enumerate}
        \item \label{prop: pade cond: item 1} The vector $(P(z),Q_1(z),\ldots,Q_m(z))$ forms a weight $2n$ Pad\'{e} type approximants of
        $(f_1,\ldots,f_m)$.
        \item \label{prop: pade cond: item 2} $\varphi_{f_s}(t^kP(t))=0$ for $1\le s \le m \ \text{and} \ 0\le k \le 2n-1$.
        \item \label{prop: pade cond: item 3} There is a set of $n$ polynomials $\{A_k(t)\}_{k=1,3,\ldots,2n-1}$ with ${\rm{deg}}\,A_k=k$ for each odd $k$, such that
        \[
            \varphi_{f_s}(A_k(t)P(t))=0 \ \text{for} \ 1\le s \le m \ \text{and} \ k=1,3,\ldots,2n-1.
        \]
    \end{enumerate}
\end{proposition}

The proof of Proposition~\ref{linear relations} will be at the end of this subsection.

\begin{lemma} \label{value phif}
    Consider a Laurent series $f(z)\in (1/z)\cdot K[[1/z]]$.
    \begin{enumerate}[label=(\alph*)]
      \item Suppose that $f(z)\in \Vp$ and write $f(z)=\sum_{k\geq 0}\dfrac{a_k}{(z)_{k+1}(1-z)_{k+1}}$. Then
        for each integer $j\geq 0$:
        \begin{enumerate}
            \item \label{lem value phif: item 1} $\varphi_f\left((t)_j(1-t)_j\right)=0$,
            \item \label{lem value phif: item 2} $\varphi_f\left((t)_{j+1}(1-t)_j\right)=-a_j$,
            \item \label{lem value phif: item 3} $\varphi_f\left((t)_j(1-t)_{j+1}\right)=a_j$.
        \end{enumerate}
      \item Suppose that $f(z)\in\Vm$ and write $f(z)=\sum_{k\geq 0}\dfrac{a_k}{(2z-1)(z)_{k}(1-z)_{k}}$. Then for each integer $j\geq 0$:
        \begin{enumerate}
            \setcounter{enumii}{3}
            \item \label{lem value phif: item 4} $\varphi_f\left((2t-1)(t)_j(1-t)_j\right)=0$,
            \item \label{lem value phif: item 5} $\varphi_f\left((t)_{j+1}(1-t)_{j+1}\right)= a_j/2$.
        \end{enumerate}
    \end{enumerate}
\end{lemma}

\begin{proof}
    Recall that $\pi:K((1/z))\longrightarrow (1/z)\cdot K[[1/z]]$ is the projection defined in Section~\ref{section: f transform}. Since all the arguments are similar, we only prove \ref{lem value phif: item 1} and \ref{lem value phif: item 4}. Fix $j\geq 0$.  If $f(z)\in \Vp$ (resp. $\Vm$) we set $P_j(z)= (z)_j(1-z)_j$ (resp. $P_j(z):=(2z-1)(z)_{j}(1-z)_j$). Then, writing $f_j(z) = P_j(z)f(z)$, we have
    \[
        \pi(f_j)=\sum_{k=j}^{\infty}\dfrac{a_k}{(z+j)_{k+1-j}(j+1-z)_{k+1-j}}\in (1/z^2)\cdot K[[1/z]].
    \]
    This implies by Proposition~\ref{key prop} (with $D = [P_j(z)]$) that $\varphi_f\left(P_j(t)\right)=\varphi_{\pi(f_j)}(1)=0$.
\end{proof}

Lemma~$\ref{value phif}$ leads us the following interesting corollary.

\begin{corollary} \label{sub ker}
    Let $f(z)\in (1/z)\cdot K[[1/z]]$. Suppose that $f(z)\in\Vp$ (resp. $f(z)\in\Vm$). Then
    \begin{enumerate}
        \item \label{item: cor: sub ker:1}  $\Wp\subseteq \ker\,\varphi_f$ (resp. $\Wm\subseteq \ker\,\varphi_f$),
        \item \label{item: cor: sub ker:2} $f(z)=-\dfrac{1}{2}{\displaystyle{\sum_{k=0}^{\infty}}}\dfrac{\varphi_f((2t-1)(t)_k(1-t)_{k})}{(z)_{k+1}(1-z)_{k+1}}$ (resp. $f(z)=2{\displaystyle{\sum_{k=0}^{\infty}}}\dfrac{\varphi_f((t)_{k+1}(1-t)_{k+1})}{(2z-1)(z)_{k+1}(1-z)_{k+1}}$).
    \end{enumerate}
\end{corollary}

\begin{proof}
    The first assertion is a direct consequence of Lemmas~\ref{poly eigen1} and \ref{value phif}. According to Lemma~$\ref{ans prob 1.2}$ the exists a sequence $(a_k)_{k\geq 0}$ in $K$ such that
    \[
        f(z)=\sum_{k=0}^{\infty}\dfrac{a_k}{(z)_{k+1}(1-z)_{k+1}} \quad (\textrm{resp. } f(z)=\sum_{k=0}^{\infty}\dfrac{a_k}{(2z-1)(z)_{k}(1-z)_{k}}).
    \]
    Lemma~\ref{value phif}~\ref{lem value phif: item 3} combined with Lemma~\ref{value phif}~\ref{lem value phif: item 1} implies that
    \begin{align*}
        a_k=\varphi_f((t)_k(1-t)_{k+1}) = \varphi_f(-t(t)_k(1-t)_{k})=-\dfrac{1}{2}\varphi_f((2t-1)(t)_k(1-t)_k).
    \end{align*}
    Similarly, if $f(z)\in \Vm$, we conclude by using Lemma~\ref{value phif}~\ref{lem value phif: item 5} instead.
\end{proof}

\begin{proof}[{\bf Proof of Proposition~\ref{linear relations}}]
    The equivalence of \ref{prop: pade cond: item 1} and \ref{prop: pade cond: item 2} was established in \cite[Lemma~$2.3$]{Kaw}, and clearly \ref{prop: pade cond: item 2} $\Rightarrow$ \ref{prop: pade cond: item 3}. It remains to prove that~\ref{prop: pade cond: item 3} $\Rightarrow$ \ref{prop: pade cond: item 2}. First, observe that the set of polynomials
    \[
        \{(t)_k(1-t)_k ; 0\le k\le n-1\}\cup \{A_k(t); k=1,3,\ldots,2n-1\}
    \]
    forms a basis of the $K$-vector space $K[t]_{\le 2n-1}$. Fix $s\in\{1,\dots,m\}$. Since $P\in \Wp$, Corollary~\ref{sub ker} implies that $(t)_k(1-t)_kP(t)$ belongs to $\ker \varphi_{f_s}$ for each $k\geq 0$. Hence, the equivalence
    \begin{align*}
        \varphi_{f_s}(t^kP(t))=0 \ \text {for } 0\le k \le 2n-1 \Longleftrightarrow
        \varphi_{f_s}(A_k(t)P(t))=0 \ \text {for } \ k=1,3,\ldots,2n-1.
    \end{align*}
\end{proof}

\section{Construction of Pad\'{e} approximants} \label{section: Construction of Pade}

We keep the notation and convention established in Section~\ref{section: Modified Mellin transform} and~\ref{involution}. This section is devoted to the explicit construction of Pad\'{e}-type approximants for the Laurent series vectors
\begin{align*}
    \Big(\Theta_2(z), \Theta_4(z),\ldots, \Theta_{m-1}(z)\Big) &\quad \text{for odd integer} \ m \geq 3, \\
    \Big(R_3(z), R_5(z), \ldots, R_{m-1}(z)\Big) &\quad \text{for even integer} \ m \geq 4,
\end{align*}
introduced in Definition~\ref{def: Theta_R}. Let $m,n$ be positive integers. We define the following difference operators:
\[
    D_{m,n} = R_{(-1)^{m+1}}^n \circ \left[ \left(z - \dfrac{n}{2} \right)^m_n \right], \qquad \textrm{where } R_{\ee} = \left( \left[ \dfrac{1}{2z - 1} \right] \circ (\tau_{1/2} + \ee \tau_{-1/2}) \right).
\]
Given any $s\geq 2$, we write
\begin{align*}
    \varphi_s = \varphi_{F_s} \quad \textrm{where } F_s = \begin{cases}
        \Theta_s & \textrm{if $s$ is even},\\
        R_s & \textrm{otherwise.}
    \end{cases}
\end{align*}

\begin{definition} \label{def: construction Pade approximants, choice of A_ell}
    Let $n\in \Z_{\geq 0}$. For $\ell=0,\dots,M:= \lfloor (m-1)/2 \rfloor$ and $s=2,\dots,m-1$, define
    \begin{align*}
        P_{m,n,\ell}(z) &= D_{m,n}\big(A_{\ell}(z)\big), \\
        Q_{m,n,s,\ell}(z) &= \varphi_{s}\left( \dfrac{P_{m,n,\ell}(z) - P_{m,n,\ell}(t)}{z - t} \right), \\
        \fR_{m,n,s,\ell}(z) &= P_{m,n,\ell}(z)F_s(z) - Q_{m,n,s,\ell}(z), \\
        \bp_{m,n,\ell}(z) &= {}^t\!\big(P_{m,n,\ell}, Q_{m,n,s_1,\ell}, \dots, Q_{m,n,s_M,\ell} \big),
    \end{align*}
    where
    \begin{align*}
        n!^{m-1} A_\ell(z) =
        \begin{cases}
            (z-n/2)_\ell (1-z-n/2)_\ell & \textrm{if $m$ is even},\\
            (z-\ell)_{2\ell} & \textrm{if $m$ is odd}.
        \end{cases}
    \end{align*}
\end{definition}

By Theorem~\ref{Pade} below, the main result of this section, the tuple of polynomials ${}^t \bp_{m,n,\ell}(z)$ defines a Pad\'{e}-type approximation of weight $2n$ of the system $(\Theta_2, \Theta_4, \ldots, \Theta_{m-1})$ (resp. $(R_3,R_5, \ldots, R_{m-1})$) when $m$ is odd (resp. even). We will prove the linear independence of $\big(\bp_{m,n,\ell}(z)\big)_{0\leq \ell \leq M}$ in Section~\ref{section: det non zero}. In Section~\ref{section: estimates bis} we will study the arithmetical properties and asymptotic behavior of the coefficients of $P_{m,n,\ell}$ written along the basis $((z)_k(1-z)_k)_{k\geq 0}$.\\
Recall that $V_1$ is the space of Laurent series $f$ satisfying $f(z)=f(1-z)$ (see Section~\ref{involution}).

\begin{theorem} \label{Pade}
    Let $m,n \in \Z_{\geq 0}$, with $m \geq 3$, and let $A(z) \in K[z]\cap V_1$.  For $s=2,\dots,m-1$, define
    \begin{align*}
        P(z) := D_{m,n}(A(z)),\quad Q_s(z) := \varphi_{s} \left( \frac{P(z) - P(t)}{z - t} \right) \AND \fR_s(z) := P(z)F_s(z) - Q_s(z).
    \end{align*}
    \begin{enumerate}
    \item \label{item:Thm Pade:1} Assume $m$ is odd. Then $\deg P = (m-1)n + \deg A$, and the vector of polynomials $\big(P(z), Q_2(z),Q_4(z),\dots,Q_{m-1}\big)$ forms a Pad\'{e}-type approximant of weight $2n$ to the $\frac{m-1}{2}$ tuple {$\big(\Theta_2(z), \Theta_4(z), \ldots, \Theta_{m-1}(z)\big)$}.
    \item \label{item:Thm Pade:2} Assume $m$ is even. Then $\deg P = (m - 2)n + \deg A$, and the vector of polynomials $\big(P(z), Q_3(z),Q_5(z),\dots,Q_{m-1}\big)$ forms a Pad\'{e}-type approximant of weight $2n$ to the $\frac{m-2}{2}$ tuple {$\big(R_3(z), R_5(z), \ldots, R_{m-1}(z)\big)$}.
    \item \label{item:Thm Pade:3} For each $s = 2, 3, \ldots, m - 1$, we have the explicit expansion
        \[
            \fR_s(z) = -\frac{1}{2} \sum_{k = n}^\infty n! \binom{k}{n} \frac{ \varphi_{s}(C_k(t)) }{(z)_{k+1}(1 - z)_{k+1}}.
        \]
        where $C_k(t) = (2t+n-1) (t+n)_{k-n}(1 - t)_{k-n} (t)^m_n A(t + n/2)$.
    \end{enumerate}
\end{theorem}

\begin{remark}
    \leavevmode
    \begin{itemize}
        \item In the case $m = 3$ and $A(z) = 1$, the polynomial $P(z)$ coincides with a case studied by Beukers~\cite[Section~6]{B}, where it appears in terms of a continuous dual Hahn polynomial (see also~\cite{R2}).
        \item In the case $m = 4$ and $A(z) = 1$, the polynomial $P(z)$ has been investigated by Pr\'{e}vost~\cite[Section~3.2]{P} (see also~\cite[Section~10]{B}), where it is identified as a Wilson polynomial.
    \end{itemize}
\end{remark}

\noindent{\bf Expansion of the Laurent series $F_s(z)$.} If $m$ is even (resp. odd) we denote by
\[
    S_m = \{s_1,\dots,s_M\}, \quad \textrm{where } M = \lfloor(m-1)/2\rfloor,
\]
the set of odd (resp. even) indices in $\{2,\dots,m-1\}$. By definition (see Section~\ref{subsection: theta_s and R_s}), for each integer $s\in S_m$, we have
\begin{align}\label{def: functions g_s}
    F_s(z) = \M(g_s(z)) \qquad \textrm{where } g_s(z) =
    \begin{cases}
        (2+z)^{-1}\log(1+z)^{s-1} & \textrm{if $s$ is even,} \\
        z^{-1}\log(1+z)^{s-1} & \textrm{if $s$ is odd}.
    \end{cases}
\end{align}
For each $s\in S_m$, let $(a_{s,k})_{k\geq 0}$ and $(b_{s,k})_{k\geq 0}$ be the sequences of rational coefficients satisfying
\begin{align}\label{expansion F_s} 
    g_s(z) = \sum_{k=0}^{\infty}a_{s,k}z^k \AND F_s(z) = \sum_{k=0}^{\infty}\dfrac{b_{s,k}(2k+1)!}{(z)_{k+1}(1-z)_{k+1}}.
\end{align}
The coefficients $b_{s,k}$ can be expressed thanks to Proposition~\ref{key expansion R_s and Theta_s}. Given integers $k,t\geq 0$, we have 
\begin{align}
    (-1)^{t+1}b_{3+2t,k} = \sum_{\substack{ j_0,\dots,j_t \geq 0 \\ j_0+\cdots + j_t = k-t}} \prod_{i=0}^{t} \frac{j_i!^2}{(j_i+1)(2j_i+1)!}, \label{eq: eq b_s,k, s odd} \\
    2(-1)^{t+1}b_{2+2t,k} = \sum_{\substack{ j_0,\dots,j_t \geq 0 \\ j_0+\cdots + j_t = k-t}}(j_t+1) \prod_{i=0}^{t} \frac{j_i!^2}{(j_i+1)(2j_i+1)!}. \label{eq: eq b_s,k, s even}
\end{align}
We do not need to work with an explicit expression of the coefficients $a_{s,k}$. Later, in Subsection~\ref{subsection: estimate approximations}, we will just need to control their $p$-adic absolute value. In view of~\eqref{eq: expansion g_s m even} and\eqref{eq: expansion g_s m odd} (and using the identity $1/(1+z)^{k+1} = \sum_{j\geq 0}\binom{k+j}{j}(-z)^j$), we have
\begin{align}\label{eq: a_s,k expressed as Z-linear combination of b_s,k}
    2a_{s,k} \in b_{s,0}\Z + \cdots + b_{s,\lfloor k/2 \rfloor}\Z \qquad (k\geq 0 \AND s\in S_m).
\end{align}
The following theorem is proved in Subsection~\ref{subsection: proof thm expansion}.

\begin{theorem} \label{thm:expansion}
    We keep the notation and hypotheses of Theorem~\ref{Pade}. Let $\ell\in\Z_{\geq 0}$ be such that $\deg(A) = 2\ell$. Then $P(z)\in V_1\cap K[z]$ and there exists $(p_k(n))_{0\leq k\leq N}\in K^{N+1}$, where $N=Mn+\ell$, such that
    \begin{align*}
        P(z) = \sum_{k=0}^{N} p_k(n) \frac{(z)_k(1-z)_k}{k!^2}.
    \end{align*}
    Furthermore, for each $s\in S_m$, we have
    \begin{align}
        Q_{s}(z) & = \sum_{k=1}^{N}p_{k}(n)\sum_{j=0}^{k-1} \frac{b_{s,j}}{j+1} \binom{2j+1}{j+1} \binom{k}{j+1}^{-2} \dfrac{(z+j+1)_{k-j-1}}{(k-j-1)!} \dfrac{(2+j-z)_{k-j-1}}{(k-j-1)!}, \label{eq thm: expansion Q_s}  \\
        \fR_{s}(z) & = \dfrac{n!}{2}\sum_{k=n}^{\infty} \binom{k}{n} \dfrac{r_{k,s}}{(z)_{k+1}(1-z)_{k+1}}, \label{eq thm: expansion remainder R_s}
    \end{align}
     where, writing $C_k(t) = (2t+n-1) (t+n)_{k-n}(1 - t)_{k-n} (t)^m_n A(t + n/2)$, we have
    \begin{align*} 
        r_{k,s}=\sum_{j=0}^{mn+2(k-n+\ell)+1}a_{s,j} \cdot \Eval_{t=0}\circ \bDelta_{-1}^j (C_{k}(t)).
    \end{align*}
\end{theorem}

\subsection{Properties of the difference operator} \label{subsection: Properties of the difference operator}

In this subsection, we prove properties of the difference operators that are needed for the proof of Theorem~$\ref{Pade}$.
Recall that $\bI:K[t]\longrightarrow  K[t]$ is the involution defined by $\bI(P(t)) = P(1-t)$ for each $P(t)\in K[t]$, and that the eigenspaces
\[
    \Wpm = {\big\{P(t)\in K[t] \,;\, \bI(P)= \pm P(t) \big\}}
\]
satisfy $\Wm = (2t-1)\Wp$. Given $\ee\in\{1,-1\}$, we define the difference operators
\[
    \bD_\ee = \btau_{1/2}+\ee\btau_{-1/2} \AND \bRvar{\ee} = \left[\dfrac{1}{2t-1}\right]\circ \bD_\ee.
\]
Note that $\deg \bD_{-1}(P(t)) \leq \deg(P(t)) -1$ for each $P(t)\in K[t]$. The identities $\bI\circ \btau_{\alpha} = \btau_{-\alpha}\circ \bI$ for $\alpha=\pm1/2$  implies that for $\ee=\pm 1$,
\begin{align}
    \label{eq: basic tau pm tau stabilized subsapces}
     \bD_1(\Wvar{\ee}) \subset \Wvar{\ee} \AND \bD_{-1}(\Wvar{\ee}) \subset \Wvar{-\ee},
\end{align}
from which we deduce $\bRp(\Wm) \subset \Wp$ and $\bRm(\Wp)\subset \Wp$.

\begin{lemma} \label{lemma key}
    Let $n,m \in\Z_{\ge 0}$ and set $T_n(t) = (t-n/2)_n\in K[t]$. Then $T_n(t)$ belongs to $\Wp$ (resp. $\Wm$) if $n$ is even (resp. odd), and
    \[
        {\bRvar{\ee}^n \big(T_n(t)^m\Wp \big)} \subset \Wp, \quad \textrm{where } \ee = (-1)^{m-1}.
    \]
    More generally,
    \begin{align*}
        \bRvar{\ee}^k\big(T_n(t)^m\Wp \big) \subset T_{n-k}(t)^m \Wp, \qquad \textrm{for $k=0,\dots,n$}.
    \end{align*}
\end{lemma}

\begin{proof}
    Since $\Wm = (2t-1)\Wp$, we just have to prove that given $A(t)\in \Wp$, the polynomial {$B(t) = \bD_\ee\big(T_n(t)^mA(t)\big)$} belongs to $T_{n-1}(t)^m\Wm$. Since $T_{n-1}(t)^m$ divides $\btau_{\pm 1/2}(T_n(t)^m)$, there exists $C(t)\in K[t]$ such that $B(t) = T_{n-1}(t)^m C(t)$. By \eqref{eq: basic tau pm tau stabilized subsapces}, we either have $B(t)\in \Wm$ and $T_{n-1}(t)^m\in \Wp$, or $B(t)\in \Wp$ and $T_{n-1}(t)^m\in \Wm$ (depending on the parity of $n$ and $m$). In both cases we conclude that $C(t)\in  \Wm$.
\end{proof}

The following proposition is a crucial ingredient in the proof of Theorem~$\ref{Pade}$.

\begin{proposition} \label{key}
        Let $P(t)\in \Wm$ and $\ee\in\{1,-1\}$. Define the operator $E = -\bD_{-1} \circ \left[\frac{1}{2t-1}\right]$. There exists a sequence of polynomials $(P_i(t))_{i\geq 0}$ in $K[t]$ such that for each integer $n\geq0$,
        \begin{align} \label{eq:prop key}
            [P(t)]\circ \bRvar{\ee}^n= \bD_\ee \circ \left(\sum_{i=0}^{n-1}\btau^i_{1/2}\circ [P_{i}(t)]\circ \bRvar{\ee}^{n-1-i}\right)+\btau^{n}_{1/2}\circ [ E^n(P(t))].
        \end{align}
\end{proposition}

\begin{proof}
    First note that by \eqref{eq: basic tau pm tau stabilized subsapces}, we have $E^k(P(t)) \in \Wm$ for each integer $k\geq 0$. We construct the polynomial $P_i(t)$ by induction on $i$. A quick computation shows that for any $Q(t)\in K[t]$, we have
    \begin{align*}
        [Q(t)] \circ \bD_\ee = \bD_\ee \circ[\btau_{1/2}(Q(t))]-\btau_{1/2}\circ[\bD_{-1}(Q(t))].
    \end{align*}
    Since $P(t)\in \Wm$, there exists $Q(t)\in \Wp$ such that $P(t) = (2t-1)Q(t)$, and setting $P_0(t) = \btau_{1/2}(Q(t))$, the above equality may be rewritten as
    \begin{align} \label{k=1}
        [P(t)]\circ \bRvar{\ee} = [Q(t)] \circ \bD_\ee = \bD_\ee \circ[P_0(t)]+\btau_{1/2}\circ[E(P(t))],
    \end{align}
    which is \eqref{eq:prop key} for $n=1$. Now fix an integer $n\geq 1$ and assume that $P_0(t),\dots,P_{n-1}(t)\in K[t]$ are such that \eqref{eq:prop key} holds. This implies that
    \begin{align*} 
        [P(t)]\circ \bRvar{\ee}^{n+1} =  \bD_\ee \circ \left(\sum_{i=0}^{n-1}\btau^i_{1/2}\circ [P_{i}(t)]\circ \bRvar{\ee}^{n-i}\right)+\btau^{n}_{1/2}\circ [ E^n(P(t))] \circ \bRvar{\ee}.
    \end{align*}
    Using \eqref{k=1} with $P(t)$ replaced by $E^n(P(t)) \in \Wm$, we find a polynomial $P_n(t)\in K[t]$ such that
    \begin{align*} 
        [ E^n(P(t))] \circ \bRvar{\ee} = \bD_\ee \circ[P_n(t)]+\btau_{1/2}\circ[E^{n+1}(P(t))].
    \end{align*}
    Since $\bD_\ee$ and $\btau_{1/2}$ commute, we obtain \eqref{eq:prop key} for $n+1$.
\end{proof}

\begin{corollary} \label{cor key}
        Let $n,m \geq 0$ be integers and $\ee\in\{1,-1\}$, and let $P(t)\in \Wm$ and $A(t)\in \Wp$. Then
        \begin{align*}
            P(t)\bRvar{\ee}^n {\big(T_n(t)^m A(t) \big)}  \in (1+\ee\btau_{-1}){\big( t^m K[t] \big)} + R(t),
        \end{align*}
        where $R(t) = \btau^{n}_{1/2}\Big( E^n(P(t))T_n(t)^m A(t) \Big)$ and $T_n(t) = (t-n/2)_n\in K[t]$ is defined as in Lemma~\ref{lemma key}. Moreover, $R(t)=0$ if $\deg P(t) < 2n$.
\end{corollary}

\begin{proof}
    By Proposition~\ref{key}, there exist $P_0(t),\dots,P_{n-1}(t)\in K[t]$ such that
    \begin{align*}
        P(t)\bRvar{\ee}^n {\big(T_n(t)^m A(t) \big)} & = \bD_\ee \left(\sum_{i=0}^{n-1}\btau^i_{1/2}\circ [P_{i}(t)]\circ \bRvar{\ee}^{n-1-i}\Big(T_n(t)^m A(t) \Big)\right) + R(t) \\
        & \subset  (1+\ee \btau_{-1}) \left(\sum_{i=0}^{n-1}\btau_{1/2}^{i+1} \Big(T_{i+1}(t)^m K[t] \Big)\right)  + R(t),
    \end{align*}
    the last inclusion coming from Lemma~\ref{lemma key} and the identity $\bD_\ee =  (1+\ee \btau_{-1})\circ\btau_{1/2}$. We conclude by noticing that $\btau^{i+1}_{1/2}(T_{i+1}(t)^m)\in t^m K[t]$ for $i=0,\dots,n-1$. Finally, suppose that $\deg P(t) < 2n$. Since $\deg(E(Q(t)) \leq \deg(Q(t)) -2$ for each $Q(t)\in \Wm$, we then have $E^n(P(t)) = 0$.
\end{proof}

\begin{lemma} \label{compute remainder}
    Define the operator $E = -\bD_{-1} \circ \left[\frac{1}{2t-1}\right]$ as in Proposition~\ref{key}. Let $k,n$ be integers with $k\ge n \geq 0$. Then
    \begin{align*}
        E^n\big((2t-1)(t)_k(1-t)_k \big)=n!\binom{k}{n}{{(2t-1)}}(t+n/2)_{k-n}(1+n/2-t)_{k-n}.
    \end{align*}
\end{lemma}

\begin{proof}
    Let $\alpha\in K$ and set $\beta = \alpha+1/2$. A direct computation shows that
    \begin{align*}
        E\big((2t-1)(t+\alpha)_k(1+\alpha-t)_k \big) = k(2t-1)(t+\beta)_{k-1}(1+\beta-t)_{k-1}.
    \end{align*}
    Applying this identity iteratively $n$ times with $\alpha=0,1/2,\ldots,(n-1)/2$, we obtain the expected result.
\end{proof}

\subsection{Proof of Theorem~$\ref{Pade}$} \label{subsection: Proof of Theorem Pade}

\begin{proof}[\textbf{Proof of Theorem~\ref{Pade} \ref{item:Thm Pade:1} and \ref{item:Thm Pade:2}}]
    Set $\ee= (-1)^{m-1}$. First, note that if $Q(t)\in \Wp$ has degree $\geq 1$, then $\deg(\bRp(Q(t))) = \deg(Q(t))-1$. If $\deg(Q(t)) \geq 2$, then $\deg(\bRm(Q(t))) = \deg(Q(t))-2$. This yields the expected value of $\deg(P(t))$. Let $1\le k\le 2n-1$ be an odd integer and define
    \[
        A_k(t)=(2t-1)(t)_{\tfrac{k-1}{2}}(1-t)_{\tfrac{k-1}{2}}\in \Wm.
    \]
    We have  $\deg(A_k(t)) = k < 2n$. Thanks to Corollary~\ref{cor key}, and since $P(t) = \bRvar{\ee}^n\big(T_n(t)^m A(t)\big)$, we obtain
    \begin{align*}
        A_k(t)P(t)\in  (\btau_{-1}+\ee)(t^m K[t]) \subset \ker\varphi_{s}  \qquad  \text{for } k=1,3,\ldots,2n-1,
    \end{align*}
    the last inclusion coming from Corollary~\ref{kernel}. So $A_k(t)P(t)$ belongs to $\ker \varphi_s$ for $k=1,3,\ldots,2n-1$. We conclude by Proposition~\ref{linear relations}~\ref{prop: pade cond: item 3}.
\end{proof}

\begin{proof}[\textbf{Proof of Theorem~\ref{Pade}~\ref{item:Thm Pade:3}}]
    Lemma~\ref{lemma key} ensures that $P(z)\in \Vp\cap K[z]$. Combined with \eqref{V^{+},V^{-}} we get  $P(z)F_s(z)\in K((1/z))$. Since $Q_s(z)$ is the polynomial part of this Laurent series (and since $I$ stabilizes $(1/z)\cdot K[[1/z]]$), we deduce that $Q(z)\in \Vp\cap K[z]$, and $\fR_s(z)\in \Vp\cap (1/z)\cdot K[[1/z]]$. Corollary~\ref{sub ker}~\ref{item: cor: sub ker:2} yields
    \begin{align*}
        \fR_s(z)=-\dfrac{1}{2}\sum_{k=0}^{\infty}\dfrac{\varphi_{\fR_s}\big((2t-1)(t)_k(1-t)_k\big)}{(z)_{k+1}(1-z)_{k+1}},
    \end{align*}
    Fix $k \geq 0$. It only remains to prove that
        \begin{align*}
            \varphi_{\fR_s}\big((2t-1)(t)_k(1-t)_k)\big)=n!\binom{k}{n}\varphi_{s}\Big({{(2t+n-1)}}(t+n)_{k-n}(1-t)_{k-n} (t)^m_nA(t+n/2)\Big).
        \end{align*}
    Recall that $\pi : K((1/z))\rightarrow  (1/z)\cdot K[[1/z]]$ is the projection defined in Section~\ref{section: f transform}. Set $D= [P(z)]$, so that $D^{*} = [P(t)]$. Then $\fR_s(z)=\pi(P(z)F_{s}(z)) = \pi(D(F_s))$, and Proposition~$\ref{key prop}$ leads us to
    \begin{align*}
        \varphi_{\fR_s}\big((2t-1)(t)_k(1-t)_k)\big)=\varphi_{s}\big((2t-1)(t)_k(1-t)_kP(t)\big).
    \end{align*}
    Since $P(t)=\bRvar{\ee}^{n}((t-n/2)^m_nA(t))$, where $\ee= (-1)^{m-1}$, we have
    \begin{align*}
        (2t-1)(t)_k(1-t)_kP(t)=[(2t-1)(t)_k(1-t)_k]\circ \bRvar{\ee}^{n}\big((t-n/2)^m_nA(t)\big).
    \end{align*}
    Combining Corollaries~\ref{kernel} and~\ref{cor key}, we obtain
    \begin{align*}
        \varphi_{s}\big((2t-1)(t)_k(1-t)_kP(t)\big) = \varphi_s\big( \btau^{n}_{1/2}\big( B(t) (t-n/2)^m_nA(t) \big) \big),
    \end{align*}
    where $B(t) = E^n\big((2t-1)(t)_k(1-t)_k\big)$. Using Lemma~\ref{compute remainder}, we finally get
    \begin{align*}
        \btau^{n}_{1/2}\big( B(t)P(t) \big) = n!\binom{k}{n}{{(2t+n-1)}}(t+n)_{k-n}(1-t)_{k-n} (t)^m_nA(t+n/2).
    \end{align*}
\end{proof}

\subsection{Proof of Theorem~\ref{thm:expansion}}\label{subsection: proof thm expansion}

Recall that the structure of $V_1$ is studied in Subsection~\ref{subsection: structure eigenspace}. First, note that the degree of $A(z)$ is even since it belongs to $V_1\cap K[z]$ (see Lemma~\ref{poly eigen1}), so $\ell = \deg(A)/2$ is an integer. Lemma~\ref{lemma key} implies that $P(z)\in V_1\cap K[z]$. The existence of the coefficients $(p_k(n))_k$ follows from Lemma~\ref{poly eigen1}.\\
Eq.~\eqref{eq thm: expansion Q_s} is a direct consequence of \eqref{expansion F_s} and the following result applied with $f = F_s$.

\begin{lemma}
    Let $f(z)\in \Vp\cap (1/z)\cdot K[[1/z]]$ and $P(z)\in \Vp\cap K[z]$ of degree $2N$, and write
    \begin{align*}
        f(z) = \sum_{k=0}^{\infty}\dfrac{a_k(2k+1)!}{(z)_{k+1}(1-z)_{k+1}} \AND P(z)=\sum_{k=0}^N p_k\dfrac{(z)_{k}}{k!}\dfrac{(1-z)_k}{k!},
    \end{align*}
    with  $a_k,p_j\in K$ for $k\in\Z_{\geq 0}$ and $j=0,\dots,N$. Then one has
    \begin{align*}
    \varphi_f\left(\dfrac{P(z)-P(t)}{z-t}\right)=\sum_{k=1}^{N}p_k\sum_{j=0}^{k-1}\frac{a_j}{j+1}\binom{2j+1}{j+1}\binom{k}{j+1}^{-2}\dfrac{(z+j+1)_{k-j-1}}{(k-j-1)!}\dfrac{(2+j-z)_{k-j-1}}{(k-j-1)!}.
    \end{align*}
\end{lemma}

\begin{proof}
    Given $k\ge 0$, define the polynomials
    \[
        t_{2k}(z)=\frac{(z)_k}{k!}\frac{(1-z)_k}{k!},  \qquad t_{2k+1}(z)=-\frac{(z)_{k+1}}{(k+1)!}\frac{(1-z)_k}{k!}.
    \]
    By linearity, it suffices to prove the formula when $P(z) = t_{2k}(z)$. Using \cite[(55)]{P}, we get
    \[
        \frac{t_{2k}(z)-t_{2k}(t)}{z-t} = \sum_{i=1}^{2k} (-1)^i \frac{t_{i-1}(t)}{\lfloor(i+1)/2 \rfloor} \frac{t_{2k}(z)}{t_i(z)}
    \]
    (with the notation of Prévost \cite{P}, we have $t_{i}(z) = (-1)^{\lfloor{(i+1)/2\rfloor}}N_{i}(z-1)$). By Corollary~\ref{sub ker} $(i)$, the polynomial $t_{2k}(t)$ belongs to $\ker\varphi_f$ for each $k\ge 0$, thus
    \begin{align*}
        \varphi_f\!\left(\frac{t_{2k}(z)-t_{2k}(t)}{z-t}\right) = \sum_{j=0}^{k-1} \frac{\varphi_f(t_{2j+1}(t))}{j+1} \frac{t_{2k}(z)}{t_{2(j+1)}(z)}.
    \end{align*}
    According to Lemma~\ref{value phif} $(ii)$, we have $\displaystyle \varphi_f(t_{2j+1}(t)) = a_j \binom{2j+1}{j+1}$. To conclude, we use the identity
    \[
        \frac{t_{2k}(z)}{t_{2(j+1)}(z)} = \binom{k}{j+1}^{-2}  \frac{(z+j+1)_{k-j-1}}{(k-j-1)!} \frac{(2+j-z)_{k-j-1}}{(k-j-1)!}.
    \]
\end{proof}

\noindent{\bf Proof of \eqref{eq thm: expansion remainder R_s}} Let $s\in S_m$. By Theorem~\ref{Pade}~\ref{item:Thm Pade:3}, it remains to prove that $\varphi_{F_s}(C_{k}(t))=-r_{k,s}$.  Write
\[
    C_{k}(t) = \sum_{j=0}^{d_k} (-1)^j \beta_j \frac{(t)_j}{j!}.
\]
where $d_k = \deg C_k = 2(k-n+\ell)+mn$ and $\beta_j\in K$. Using the identity $\bDelta_{-1}((t)_j) = -j(t)_{j-1}$, (where $\bDelta_{-1} = \btau_{-1}-1$), we obtain
\begin{align*}
    \beta_j = \Eval_{t=0} \circ \bDelta^j_{-1}(C_{k}(t)).
\end{align*}
On the other hand, \cite[Corollary~4.9]{KP2} (applied with $g=g_s$ and $f = F_s$), yields
\[
    \varphi_{s}\!\left(\frac{(t)_j}{j!}\right) = (-1)^{j+1} a_{s,j}.
\]
We conclude that $\varphi_{s}(C_{k}(t)) = -\sum_{j=0}^{d_k} a_{s,j}\cdot\Eval_{t=0} \circ \bDelta^{j}_{-1}(C_{k}(t))$.

\section{Linear independence of the Pad\'{e} approximants} \label{section: det non zero}

Let $m, n$ be positive integers with $m \geq 3$. Following the notation of Section~\ref{section: Construction of Pade} we set
\begin{align*}
    F_s = \begin{cases}
    \Theta_s & \textrm{if $s$ is even},\\
     R_s & \textrm{otherwise,}
     \end{cases}
\end{align*}
for each $s\geq 2$. As usual, for simplicity we write $\phi_s = \phi_{F_s}$. If $m$ is even (resp. odd) we denote by
\[
    S_m = \{s_1,\dots,s_M\}, \quad \textrm{where } M = \lfloor(m-1)/2\rfloor,
\]
the set of odd (resp. even) indices in $\{2,\dots,m-1\}$.

\medskip

Recall that according to Theorem~\ref{Pade}, the tuple of polynomials ${}^t \bp_{m,n,\ell}(z)$ from Definition~\ref{def: construction Pade approximants, choice of A_ell} provides a Pad\'{e}-type approximation of weight $2n$ of the system $(\Theta_2, \Theta_4, \ldots, \Theta_{m-1})$ (resp. $(R_3,R_5, \ldots, R_{m-1})$) when $m$ is odd (resp. even). Define the $(M+1)\times (M+1)$ matrix
\[
    \cM_{m,n}(z) = \Big( \bp_{m,n,0}(z) \ldots \bp_{m,n,M}(z) \Big),
\]
whose columns are $\bp_{m,n,\ell}(z)$ for $\ell=0,\dots, M$. The main result of this section is the following one, which guarantees the linear independence of the above Padé approximants.

\begin{theorem} \label{non zero det}
    We have $\det \cM_{m,n}(z)\in \Q^\times$.
\end{theorem}

We briefly outline the strategy of the proof. By definition, $\cM_{m,n}(z)$ is a polynomial in $z$. In the next subsection we prove the following lemma, which is essentially an application of \cite[Lemma~4.2(ii)]{DHK3}, and implies that $\det \cM_{m,n}(z) \in K[z]$ is a constant polynomial. Recall that the polynomials $A_k(z)$ are as in Definition~\ref{def: construction Pade approximants, choice of A_ell}.

\begin{lemma}\label{det is const}
    Define the $M \times M$ determinant $\Omega_{m,n}$ by
    \[
        \Omega_{m,n} = \det\Big(\varphi_{s_j}\big((2t+n-1)(t)_n^m A_{k-1}(t+n/2) \big)\Big)_{1 \leq j,k \leq M}.
    \]
    There exists $c\in \Q^\times$ such that $\det \cM_{m,n}(z)= c\cdot \Omega_{m,n}\in \Q$.
\end{lemma}

We then show that $\Omega_{m,n}\neq 0$ in Subsection~\ref{subsection: proof of non zero}. This is based on the property of the intersection of the kernels of the formal {umbral} transforms $\varphi_{R_s}$ and $\varphi_{\Theta_s}$ (when $s$ runs throughout $S_m$) established in Corollary~\ref{cor: expression restricted intersection ker phi} of Subsection~\ref{subsection: theta_s and R_s}.

\begin{remark}\label{remark: linear indpce, residual term}
    Essentially, we use a method similar to the one introduced in \cite[Section~8.2]{KP2}, and whose main advantage is that it does not require to get an ``closed and explicit expression'' of the determinant. However, in the present article our arguments are more complicated because the crucial \cite[Lemma~8.4]{KP2} is now replaced by Corollary~\ref{cor: expression restricted intersection ker phi}. The main difficulty comes from the residual term $\bpp\big(K[t]_{\leq m-1-\delta_m}\big)$ appearing in the corollary, and which did not exist in \cite[Lemma~8.4]{KP2}. We solve this issue thanks to the technical Proposition~\ref{cor: key for non-vanishing det} of Subsection~\ref{subsection: study of kernel}.
\end{remark}

\subsection{Proof of Lemma~\ref{det is const}} \label{subsection: the determinants are constant}

Form the $(M+1) \times (M+1) $ matrix
\begin{align*}
    \cM_{m,n,\fR}(z) = \big(\br_{m,n,0}(z) \ldots \br_{m,n,M}(z)\big),
\end{align*}
where $\br_{m,n,\ell}(z)$ is the vector column ${}^t\!\big(P_{m,n,\ell}(z),\fR_{m,n,s_1,\ell}(z),\dots,\fR_{m,n,s_M,\ell}(z)\big)$ for  $\ell = 0, \dotsc, M = |S_m|$. By definition of the Pad\'{e} approximants $\fR_{m,n,s,\ell}(z)$, we have
\begin{align}
    \label{eq: matrix relation Pade approximations}
    U(z)\mathcal{M}_{m,n}(z) = \mathcal{M}_{m,n,\mathfrak{R}}(z), \quad \text{where }
    U(z) = \begin{pmatrix}
        1 & 0 &  \cdots & 0 \\
        F_{s_1}(z) & -1 &  \cdots & 0 \\
        \vdots & \vdots &  \ddots & \vdots \\
        F_{s_M}(z) & 0  & \cdots & -1
    \end{pmatrix}.
\end{align}

The entries in the first row of $\cM_{m,n,\fR}(z)$ are the polynomials $P_{m,n,\ell}(z)$, for $\ell=0,\dots,M$. Note that $P_{m,n,\ell}(z)$ has degree $2Mn+2\ell$, hence
\begin{align}
    \label{eq proof: lemme det Matp and Theta: eq 1}
    P_{m,n,0}(z), \dotsc, P_{m,n,M-1}(z) \in \Q[z]_{\leq2M(n+1)-2} \AND  P_{m,n,M}(z) \in \tilde{c} z^{2Mn} + \Q[z]_{\leq 2M(n+1)-1},
\end{align}
where $\tilde{c} \in \Q^\times$ denotes the leading coefficient of $P_{m,n,M}(z)$.
On the other hand, for each $s\in S_m$ and each $\ell$ with $0 \leq \ell \leq M$, Theorem~\ref{Pade}~\ref{item:Thm Pade:3} ensures that
\begin{align}
    \label{eq proof: lemme det Matp and Theta: eq 2}
    \fR_{m,n,s,\ell}(z) \in \frac{n! \varphi_{s}\big({(2t+n-1)}(t)^m_n A_\ell(t+n/2)\big)}{z^{2(n+1)}} + \frac{1}{z^{2n+3}} \Q[[1/z]].
\end{align}
For completeness, we briefly recall the main argument of \cite[Lemma~4.2~(ii)]{DHK3}. Expanding $\det \mathcal{M}_{m,n,\mathfrak{R}}(z)$ along its first row and using \eqref{eq proof: lemme det Matp and Theta: eq 1} and \eqref{eq proof: lemme det Matp and Theta: eq 2}, we obtain
\begin{align*}
    \det \cM_{m,n,\fR}(z)
    = (n!)^{M} \tilde{c} \cdot \det\Big(\varphi_{s_j}\big((2t+n-1)(t)^m_n A_{\ell-1}(t+n/2)\big)\Big)_{1 \leq j,\ell \leq M} + f(z),
\end{align*}
where $f(z)\in \frac{1}{z} \Q[[1/z]]$. Finally, by \eqref{eq: matrix relation Pade approximations} we have $\det \mathcal{M}_{m,n,\mathfrak{R}}(z) = (-1)^{M-1} \det \mathcal{M}_{m,n}(z) \in \Q[z]$. So $f(z) = 0$ and the conclusion follows easily.

\subsection{A key proposition} \label{subsection: study of kernel}

Given $\ee\in\{1,-1\}$, the projection $\bpvar{\ee} :  P(t) \mapsto \frac{1}{2}(P(t)+ \ee P(1-t))$ is as in Section~\ref{section: f transform}. For any $k\in\Z_{\geq 0}$, define the operators $\bq_{\ee,k} : K[t]\rightarrow K[t]$ by setting
\begin{align*}
    \bq_{\ee,k}(P(t)) = P(t) +\ee P(1-t-k) \qquad (P(t)\in K[t])).
\end{align*}
Let $P(t),Q(t)\in K[t]$ be such that $Q(t) = P(t+1/2)$. Then, $\bq_{\ee,k+1}(Q) = 0$ if and only if $\bq_{\ee,k}(P) = 0$, hence
\begin{align}
    \label{eq: inclusion ker q_k via shift}
    \btau_{1/2}(\ker \bq_{\ee,k}) = \ker \bq_{\ee,k+1}.
\end{align}

\begin{proposition}\label{cor: key for non-vanishing det}
    Let $m,n\geq 1$ be integers and let $P(t),Q(t),R(t)\in K[t]$ be such that
    \begin{enumerate}
      \item \label{enum: prop clef for non_vanishing det: item 1} $(t)_n^m P(t) = Q(t) + (\btau_{-1}+(-1)^{m+1})(t^m R(t))$;
      \item \label{enum: prop clef for non_vanishing det: item 2}  $\deg Q(t) <  nm + \deg P(t)$;
      \item \label{enum: prop clef for non_vanishing det: item 3} $Q(t)\in \ker \bpm = \ker \bq_{-1,0}$.
    \end{enumerate}
    If $P(t)\in \ker \bq_{1,n} \setminus\{0\}$, then $\deg P \geq m-\delta_m$, where $\delta_m = \dim \ker (\btau_{-1}+(-1)^{m+1})\in \{0,1\}$.
\end{proposition}

Proposition~\ref{cor: key for non-vanishing det} will be a consequence of the following auxiliary result.

\begin{lemma}\label{prop: key for non-vanishing det}
    Let $m,n\geq 1$ be integers. Set $\ee = (-1)^{m+1}$ and $\delta_m = \dim \ker (\btau_{-1}+\ee)\in \{0,1\}$. Suppose that there exist $P(t),Q(t),R(t)\in K[t]$, with $P(t)\neq 0$, satisfying conditions~\ref{enum: prop clef for non_vanishing det: item 1}--\ref{enum: prop clef for non_vanishing det: item 3} of Proposition~\ref{cor: key for non-vanishing det}. Then there exist polynomials $S_{-1},S_0,\dots,S_n\in K[t]$ such that for $k=0,\dots,n$
    \begin{enumerate}[label=(\alph*)]
      \item \label{enum: prop clef for non_vanishing det: item got 1} $S_k(t)\in \ker \bq_{-1,k-1}$;
      \item \label{enum: prop clef for non_vanishing det: item got 2} $(t+k)_{n-k}^m P(t) + (1-t)^m S_k(t) = S_{k-1}(t+1) \in \ker \bq_{-1,k}$;
      \item \label{enum: prop clef for non_vanishing det: item got 3} $\deg S_{k}(t) \leq (n-k-1)m + \deg P + \delta_m$.
    \end{enumerate}
\end{lemma}

\begin{proof}
    For simplicity, write $\bq_{-1,k} = \bq_k$. By condition~\ref{enum: prop clef for non_vanishing det: item 2} the degree of $(\btau_{-1}+\ee)(t^m R(t))$ is equal to $\deg \big((t)_n^m P(t)\big) = nm + \deg P$. So, $R$ has degree $(n-1)m + \deg P + \delta_m$. The polynomial
    \begin{align*}
        S(t) = \frac{1}{2}\big(R(t)+ R(-t)\big) \in \ker \bq_{1}
    \end{align*}
    has degree at most $(n-1)m + \deg P + \delta_m$. We proceed by induction on $k$ starting with $k=0$. Applying $\bpm$ to both sides of Eq.~\ref{enum: prop clef for non_vanishing det: item 1}, and using \eqref{eq inter: identity op easy}, we find
    \begin{align*}
        \bpm\big( (t)_n^m P(t) ) = \bpm\circ (\btau_{-1}+ \ee) \big(t^m R(t)\big) = (\btau_{-1}+ \ee)(t^m S(t)) = -2\ee \bpm( (1-t)^m S(1-t)),
    \end{align*}
    where the last equality comes from $S(t-1) = S(1-t)$. Furthermore, $S_0(t) := -2\ee S(t-1)$ belongs to $\ker \bq_{-1}$ by \eqref{eq: inclusion ker q_k via shift}.  We also have $\deg S_0 = \deg S \leq (n-1)m+\deg P + \delta_m$. By the above, $A(t):=(t)_n^mP(t) + (1-t)^m S_0(t)$ belongs to $\ker \bpm = \ker \bq_0$. Finally, \ref{enum: prop clef for non_vanishing det: item got 2} holds with $S_{-1}(t)=A(t-1)$.

    \medskip

    Suppose now that  for some integer $k$ with $0\leq k < n$ there exists $S_k(t)\in K[t]$ such that \ref{enum: prop clef for non_vanishing det: item got 1}--\ref{enum: prop clef for non_vanishing det: item got 3} hold. Applying $\bq_k$ to \ref{enum: prop clef for non_vanishing det: item got 2}, we find
    \begin{align*}
        (t+k)_{n-k}^m P(t) - (t+k)^m S_k(1-t-k) = (1-t)_{n-k}^mP(1-t-k)-(1-t)^m S_k(t) \in \ker\bq_{k}.
    \end{align*}
    Since $(1-t)^m$ divides the right-hand side polynomial (and since $(1-t)$ and $(t+k)$ are coprime), we deduce that there exists $S_{k+1}(t)\in K[t]$ such that
    \begin{align}
        \label{eq inter: relation S_k and S_k+1}
        (t+k+1)_{n-k-1}^m P(t) - S_k(1-t-k) = -(1-t)^m S_{k+1}(t).
    \end{align}
    As $(t+k)^m(1-t)^m S_{k+1}(t) \in \ker \bq_k$, we also have $S_{k+1}(t)\in \ker \bq_k$. We infer from \ref{enum: prop clef for non_vanishing det: item got 1} combined with \eqref{eq: inclusion ker q_k via shift} that $S_k(1-t-k) = S_k(t+1) \in \ker \bq_{k+1}$. So \ref{enum: prop clef for non_vanishing det: item got 1} and \ref{enum: prop clef for non_vanishing det: item got 2} hold with $k$ replaced by $k+1$. Finally, considering the degree of the polynomials involved in \eqref{eq inter: relation S_k and S_k+1} together with \ref{enum: prop clef for non_vanishing det: item got 3}, we obtain
    \begin{align*}
        m + \deg S_{k+1} \leq \max\big\{(n-k-1)m + \deg P, \deg S_k\big\} \leq (n-k-1)m + \deg P + \delta_m.
    \end{align*}
    This concludes the induction step.
\end{proof}

\begin{proof}[{\bf Proof of Proposition~\ref{cor: key for non-vanishing det}}]
    Let $S_0,\dots,S_n\in K[t]$ be the polynomials provided by Lemma~\ref{prop: key for non-vanishing det}. Suppose that $\deg P < m-\delta_m$. Then Lemma~\ref{prop: key for non-vanishing det}~\ref{enum: prop clef for non_vanishing det: item got 3} with $k=n$ implies $S_n = 0$. Together with Lemma~\ref{prop: key for non-vanishing det}~\ref{enum: prop clef for non_vanishing det: item got 2} and the hypothesis $P(t)\in \ker \bq_{1,n}\setminus\{0\}$, this yields $P(t) \in \ker \bq_{1,n} \cap \ker \bq_{-1,n} = \{0\}$, a contradiction.
\end{proof}

\subsection{Proof of Theorem~\ref{non zero det}} \label{subsection: proof of non zero}

Set $\delta_m = \dim \ker (\btau_{-1}+(-1)^{m+1})\in \{0,1\}$. We have $\delta_m = 1$ if and only if $m$ is even. By Lemma~\ref{det is const}, it suffices to show that
\begin{align}\label{eq: det Omega non nul}
    \Omega_{m,n} \ne 0.
\end{align}
By contradiction, suppose that $\Omega_{m,n} = 0$. Then there exists a non-zero vector $\ba = {}^t(a_1, \ldots, a_{M}) \in K^{M}$ in the kernel of the matrix
\[
    \Big(\varphi_{s_j}\big((2t+n-1)(t)_n^m A_{k-1}(t+n/2) \big)\Big)_{1 \leq j,k \leq M},
\]
where $S_m = \{s_1,\dots,s_M\}$ is defined at the beginning of Section~\ref{section: det non zero}. By linearity of the $\varphi_{s}$, and using Corollary~\ref{cor: expression restricted intersection ker phi}, we deduce that
\[
    (t)^m_n P(t) \in \bigcap_{s \in S_m} \ker \varphi_{s} = \bpp\big( K[t]_{\leq m-1-\delta_m}\big) +  (\btau_{-1}+(-1)^{m+1})(t^m K[t]),
\]
where $P(t) = \btau_{1/2}^n \Big((2t-1)\sum_{k=1}^{M} a_k A_{k-1}(t)\Big) \neq 0$. So, there exist $Q(t), R(t)\in K[t]$, with $Q(t)\in \textrm{Im }\bpp = \ker \bpm$ and $\deg Q(t) \leq m-1-\delta_m < nm+\deg P$, such that
\begin{align*}
    (t)^m_n P(t) = Q(t) + (\btau_{-1}+(-1)^{m+1})(t^m R(t)).
\end{align*}
The polynomials $P(t), Q(t), R(t)$ satisfy the hypotheses \ref{enum: prop clef for non_vanishing det: item 1}--\ref{enum: prop clef for non_vanishing det: item 3} of Proposition~\ref{cor: key for non-vanishing det}. Furthermore,  $(2t-1)\sum_{k=1}^{M} a_k A_{k-1}(t)$ belongs to $\ker \bp_{1} = \ker \bq_{1,0}$, so that $P(t)\in \ker \bq_{1,n}$ by \eqref{eq: inclusion ker q_k via shift}. By Proposition~\ref{cor: key for non-vanishing det}, we must have $\deg P(t)+\delta_m \geq m$. Yet, since $M=\lfloor(m-1)/2\rfloor$, this leads us to the contradiction
\begin{align*}
    \deg P(t)+\delta_m \leq 1+2(M-1) + \delta_m = m-2 < m.
\end{align*}

\subsection{An interesting corollary}

The following result will be useful in the Appendix~\ref{section: appendix rec} (the reader can skip it on the first reading).\\
We keep the notation introduced in Definition~\ref{def: construction Pade approximants, choice of A_ell}. Recall that $T_k(z) = (z-k/2)_k$ for each $k\geq 0$.
\begin{corollary}\label{cor: non-sing matrix, useful for rec}
    Let $n,m,M\geq 0$ be integers with $m\geq 3$ and $M=\lfloor (m-1)/2 \rfloor$. The $M\times M$ matrix
    \[
        \Big(\phi_s\big(T_{2n+1}(t)P_{m,n,\ell}(t) \big) \Big)_{\substack{s\in S_m \\ 0\leq \ell < M}}
    \]
    is non-singular.
\end{corollary}
\begin{proof}
    Define the operation $ E = (\btau_{-1/2} - \btau_{1/2}) \circ \left[\frac{1}{2t-1}\right]$ as in Proposition~\ref{key}. Recall that $P_{m,n,\ell}(z) = \bRvar{\ee}^n \Big(T_n(t)^m A_\ell(t) \Big)$, where $\ee=(-1)^{m-1}$ (see Definition~\ref{def: construction Pade approximants, choice of A_ell}). By Corollary~\ref{cor key} combined with Corollary~\ref{kernel}, we have, for each $s\in S_m$,
    \begin{align*}
        \phi_s\big( T_{2n+1}(t)P_{m,n,\ell}(t) \big) = \phi_s\big(R(t)\big),
    \end{align*}
    where
    \begin{align*}
        R(t) = \btau^{n}_{1/2}\Big( E^n\big(T_{2n+1}(t)\big)T_n(t)^m A_\ell(t) \Big).
    \end{align*}
    Now, notice that, since $T_{2n+1}(t)\in W_{-1}$, the polynomial $E^n(T_{2n+1}(t))$ belongs to $W_{-1}$ and has degree $1$. So, there exists $\lambda\neq 0$ such that  $E^n(T_{2n+1}(t)) = \lambda(2t-1)$. We deduce that
    \begin{align*}
        R(t) = \lambda(2t+n-1)(t)_n^m A_\ell(t+n/2).
    \end{align*}
    We conclude by \eqref{eq: det Omega non nul}.
\end{proof}

\section{Properties of the polynomials $P_{m,n,\ell}(z)$} \label{section: estimates bis}

Here, we study the polynomials $P_{m,n,\ell}(z)$ introduced in Definition~\ref{def: construction Pade approximants, choice of A_ell}. The eigenspaces $\Vpm$ are defined as in Section~\ref{involution}. Since $P_{m,n,\ell}(z)\in \Vp\cap K[z]$ (in other words $P_{m,n,\ell}(1-z) = P_{m,n,\ell}(z)$), we may use Lemma~\ref{poly eigen1} with $\alpha = 0$. The goal of this section is to study the arithmetic properties and asymptotic behavior of the coefficients $p_{0,\ell}^{(m)}(n),\ldots, p_{N_\ell,\ell}^{(m)}(n)\in \Q$ such that
\begin{align*}
    P_{m,n,\ell}(z)=\sum_{k=0}^{N_\ell}p_{k,\ell}^{(m)}(n)\dfrac{(z)_{k}}{k!}\dfrac{(1-z)_{k}}{k!},
\end{align*}
where $N_\ell = \lfloor(m-1)/2\rfloor n$ satisfies $\deg P_{m,n,\ell}(z) = 2N_\ell$. With that in mind, we obtain the following result.

\begin{theorem}\label{thm: property of the coeff of P_n}
    For each integer $n,m,\ell\geq 0$ with $n$ even, $m\geq 3$ and $0\leq \ell \leq M:=\lfloor (m-1)/2 \rfloor$, we have
    \begin{align*}
        (-1)^k p_{k,\ell}^{(m)}(n) \in \Z_{>0} \qquad (0\leq k \leq N_\ell).
    \end{align*}
    Furthermore, with the convention $0^0 = 1$, we have
    \begin{align*}
        \limsup_{\substack{n\to\infty \\ \textrm{$n$ even}}} \max_{0 \leq k\leq N_\ell} |p_{k,\ell}^{(m)}(n)|^{1/n} \leq C(m) \qquad \textrm{where } C(m) =
        \begin{cases}
          \displaystyle \left(\frac{(m-1)^2}{2}\right)^{(m-1)/2}  & \textrm{if $m \geq 3$ is odd,} \\
          \displaystyle \frac{m}{2}\left(\frac{(m-2)m}{2}\right)^{(m-2)/2} & \textrm{if $m \geq 2$ is even.}
        \end{cases}
    \end{align*}
\end{theorem}

\begin{remark}
    For even $m\geq 4$, we actually do not have to restrict ourselves to even indices $n$ in Theorem~\ref{thm: property of the coeff of P_n}.
\end{remark}

We keep the notation of Section~\ref{section: Construction of Pade}.  Recall that for $\ee = \pm1$, we have
\begin{align*}
    D_\ee = \tau_{1/2} + \ee \tau_{-1/2} \AND R_\ee = \left[ \frac{1}{2z-1}\right] \circ D_\ee.
\end{align*}
Let $m, n$ be positive integers with $m \geq 1$ and set $\ee = (-1)^{m-1}$ and $M = \lfloor (m-1)/2 \rfloor$. By definition, we have
\begin{align*}
    P_{m,n,\ell}(z) = R_\ee^n\left(H_{n,\ell}^{(m)}(z)\right) \qquad (\ell = 0,\dots, M),
\end{align*}
where $H_{n,\ell}^{(m)}(z)$ is the polynomial
\begin{align}\label{eq: def H_ell}
    H_{n,\ell}^{(m)}(z) = \frac{(z-n/2)_{n}^m}{n!^{m-1}} A_\ell^{(m)}(z) \AND A_\ell^{(m)}(z) =
    \begin{cases}
        (z-n/2)_\ell (1-z-n/2)_\ell & \textrm{if $m$ is even},\\
        (z-\ell)_{2\ell} & \textrm{if $m$ is odd}.
    \end{cases}
\end{align}
Note that $A_\ell^{(m)}(z) = n!^{m-1} A_\ell(z)$, where $A_\ell(z)$ is as in Definition~\ref{def: construction Pade approximants, choice of A_ell}. For each $\ell\in\Z_{\geq 0}$, the degree of $P_{m,n,\ell}$ is $\deg P_{m,n,\ell} = 2N_\ell$, where
\begin{align*}
    N_\ell := Mn+\ell = \left\lfloor \frac{m-1}{2} \right\rfloor n + \ell.
\end{align*}

\noindent{\bf Strategy and structure of the proof.} Theorem~\ref{thm: property of the coeff of P_n} is proved in Subsection~\ref{subsection : proof thm coeff P_n}.\\ 
Our idea is to write
\begin{align*}
    H_{n,\ell}^{(m)}(z) = \sum_{k} a_k^{(m)} S_k(z),
\end{align*}
along some ``nice'' basis $(S_k(z))_{k\geq 0}$ of $V_{\pm 1} \cap K[z]$, such that $R_{\ee}^n(S_k(z))$ is proportional to some $(z)_{j}(1-z)_j$. In subsection~\ref{subsection: reduction ell =0}, we explain how to reduce the problem to the case where $\ell=0$.\\
When $m$ is odd, we will consider in Subsection~\ref{subsection: odd case} a subfamily of $(T_{k}(z))_{k\geq 0}$, where for any integer $k\geq 0$
\[
    T_k(z) = (z-k/2)_k.
\]
This is motivated by the identities
\begin{align*}
    R_1(T_{k+1}(z))  = T_{k}(z) \qquad \textrm{for each $k\geq 0$}
\end{align*}
and $(z)_{k}(1-z)_k = (-1)^k T_{2k}(z)$.\\
When $m$ is even, we will work in Subsection~\ref{subsection: even case} with the basis $(U_k)_{k\geq 0}$ of $V_1\cap K[z]$, defined by
\begin{align*}
    U_k(z) = (-1)^k(z-n/2)_k(1-z-n/2)_k \qquad (k\geq 0).
\end{align*}
Indeed, in that case, we will see that
\begin{align*}
    R_{-1}^n(U_k(z))  = \frac{k!}{n!} (-1)^{k-n}(z)_{k-n}(1-z)_{k-n} \qquad \textrm{for each $k\geq n$}.
\end{align*}

\subsection{Case of odd $m$ and $\ell=0$} \label{subsection: odd case}

The polynomial $H_{n,0}^{(m)}(z)$ is defined as in \eqref{eq: def H_ell}. Recall that for each $k\geq 0$, the polynomial $T_k(z)$ is
\begin{align*}
    T_k(z) = (z-k/2)_k = k!\binom{z+k/2-1}{k}.
\end{align*}
Given non-negative integers $n,m\geq 0$, with $m$ odd, we define by induction a sequence of coefficients $(a_j^{(m)}(n))_{0\leq j\leq N^{(m)}}$, where
\begin{align} \label{eq def: N^(m) odd case}
    N^{(m)} = \frac{n(m-1)}{2},
\end{align}
by setting $a_0^{(1)}(n) = 1$ and by using the relation
\begin{align*}
    a_j^{(m+2)}(n) = \sum_{k=\max\{j-n,0\}}^{\min\{j,N^{(m)}\}} a_k^{(m)}(n) \binom{n+k}{k}^2\binom{n}{j-k} \qquad (0\leq j \leq N^{(m+2)}).
\end{align*}
Clearly, we have
\begin{align*}
    a_j^{(m)}(n) \in \Z_{>0} \qquad (0 \leq j \leq N^{(m)}).
\end{align*}

\begin{proposition}\label{prop: rec coeff H^m odd case}
    Let $n,m \geq 0$, with $m$ odd. Then
    \begin{align}\label{eq prop: rec coeff H^m odd case}
        H_{n,0}^{(m)}(z) = \sum_{j=0}^{N^{(m)}} a_j^{(m)}(n)\frac{ T_{n+2j}(z)}{(j)!^2}.
    \end{align}
\end{proposition}

\begin{proposition}\label{prop: asymptotic coeff odd case}
    Given an odd integer $m \geq 1$, then setting $M=(m-1)/2$, we have
    \begin{align}\label{eq: asymptotic coeff odd case}
        \limsup_{n\to\infty} \max_{0\leq j\leq N^{(m)}} \big\{|a_j^{(m)}(n)|^{1/n} \big\} \leq (2M)^M M^{M},
    \end{align}
    with the convention $0^0 = 1$.
\end{proposition}

We will prove Proposition~\ref{prop: rec coeff H^m odd case} by induction on $m$. The following technical result will be needed in the induction step.

\begin{lemma}\label{lem: rec coeff H^m odd case}
    For any integer $n,j \geq 0$, we have
    \begin{align*}
        \left(\frac{T_n(z)}{n!} \right)^2 \cdot \frac{T_{n+2j}(z)}{j!^2} = \binom{n+j}{j}^2 \cdot \sum_{k=j}^{n+j} \binom{n}{k-j} \frac{T_{n+2k}(z)}{k!^2}.
    \end{align*}
\end{lemma}

\begin{proof}
    Fix an integer $n\geq 0$. Define $F(j,z) = \sum_{k = j}^{n+j} F(j,k,z)$, where
    \begin{align*}
        F(j,k,z)= \binom{n}{k-j}\frac{T_{n+2k}(z)}{k!^2} = \binom{n}{k-j}\binom{z+n/2+k-1}{n+2k}\frac{(n+2k)!}{k!^2}
    \end{align*}
    is hypergeometric with respect to $j$ and $k$ (with the convention that $\binom{x}{i} = 0$ if $i\in\Z_{<0}$). A quick computation ensures that the hypergeometric term
    \begin{align*}
        G(j,k) = \frac{(j-k)k^2}{n+j+1-k} F(j,k,z)
    \end{align*}
    satisfies the relation
    \[
        (n+j+1)^2F(j+1,k,z) +\left(j-z+\frac{n}{2}+1\right)\left(j+z+\frac{n}{2}\right)F(j,k,z) = G(j,k+1) - G(j,k)
    \]
    (this function was found by using Zeilberger's algorithm with Maple, see \cite{ZeilbergerAeqB}). Taking the sum over $k\in\Z$, we get for any integer $j\geq 0$
    \begin{align*}
        F(j+1,z) = \frac{(z-j-n/2-1)(j+z+n/2)}{(n+j+1)^2}F(j,z).
    \end{align*}
    The sequence $(F'(j,z))_{j\geq 0}$ defined by
    \begin{align*}
        F'(j,z) = \left(\frac{T_n(z)}{n!} \right)^2 \cdot \frac{T_{n+2j}(z)}{j!^2}\binom{n+j}{j}^{-2} = \left(\frac{T_n(z)}{n!} \right)^2 \cdot \binom{z+n/2+j-1}{n+2j}\frac{(n+2j)!n!^2}{(n+j)!^2}
    \end{align*}
    satisfies the same recurrence, so there exists $\lambda(z)$ such that $F'(j,z) = \lambda(z) F(j,z)$ for each $j\geq 0$. Moreover, since both $F'(j,z)$ and $F(j,z)$ have degree $3n+2j$ in the variable $z$, the constant $\lambda$ does not depend on $z$. Evaluating at $j=0$ and $z = z_0 :=n/2+1$, we find that $F'(0,z_0) = n! = F(0,z_0)$ (the only non-zero term in the sum defining $F(0,z_0)$ is for $k=0$). We deduce that $F'(j,z) = F(j,z)$ for any integer $j\geq 0$, as expected.
\end{proof}

\begin{proof}[{\bf Proof of Proposition~\ref{prop: rec coeff H^m odd case}}]
    We prove \eqref{eq prop: rec coeff H^m odd case} by induction on $m\geq 1$. For $m=1$, we have $H^{(1)}_{n,0}(z) = T_n(z)$, so it is automatic. Suppose that \eqref{eq prop: rec coeff H^m odd case} holds for a given $m\geq 1$. Then, using Lemma~\ref{lem: rec coeff H^m odd case}, we obtain
    \begin{align*}
        H^{(m+2)}_{n,0} = \left(\frac{T_{n}(z)}{n!}\right)^2H_{n,0}^{(m)} & = \sum_{k=0}^{N^{(m)}} a_k^{(m)}(n) \binom{n+k}{k}^2 \cdot \sum_{j=k}^{n+k} \binom{n}{j-k} \frac{T_{n+2k}(z)}{k!^2} \\
        & = \sum_{j=0}^{N^{(m+1)}} \left( \sum_{k=\max\{0,j-n\}}^{\min\{N^{(m)},j\}} a_k^{(m)}(n) \binom{n+k}{k}^2 \cdot \binom{n}{j-k} \right) \frac{T_{n+2j}(z)}{j!^2},
    \end{align*}
    which, by definition of the coefficients $(a_j^{(m+2)(n)})_j$, is exactly \eqref{eq prop: rec coeff H^m odd case} with $m$ replaced by $m+2$.
\end{proof}

\begin{proof}[{\bf Proof of Proposition~\ref{prop: asymptotic coeff odd case}}]
    Define
    \begin{align*}
        A^{(m)}(n) = \max_{0\leq j \leq N^{(m)}}\left\{a_j^{(m)}(n)\right\} \AND A_\infty^{(m)} = \limsup_{n\to\infty} A^{(m)}(n)^{1/n}.
    \end{align*}
    Given $n\geq 0$, we have
    \begin{align*}
        A^{(m+2)}(n) \leq \sum_{j=0}^{N^{(m+2)}} a_j^{(m+2)}(n) & = \sum_{j=0}^{N^{(m+2)}}\sum_{k=\max\{j-n,0\}}^{\min\{j,N^{(m)}\}} a_k^{(m)}(n) \binom{n+k}{k}^2\binom{n}{j-k} \\
        & = \sum_{k=0}^{N^{(m)}} a_k^{(m)}(n) \binom{n+k}{k}^2 \cdot \sum_{j=k}^{n+k} \binom{n}{j-k} \\
        & \leq A^{(m)}(n) 2^n (N^{(m)}+1)\binom{n+N^{(m)}}{n}^2,
    \end{align*}
    since $\sum_{\ell}\binom{n}{\ell} = 2^\ell$ and since $k\mapsto \binom{n+k}{k}$ is increasing. Using Stirling's formula, we deduce that
    \begin{align*}
        A_\infty^{(m+2)} \leq  2\left(\frac{(1+M_m)^{1+M_m}}{M_m^{M_m}}\right)^2 A_\infty^{(m)},
    \end{align*}
    where $M_m = N^{(m)}/n = (m-1)/2$. By induction, we get
    \begin{align*}
        A_\infty^{(m)} \leq  2^{M_m} \prod_{b=0}^{M_m-1}\left(\frac{(1+b)^{1+b}}{b^{b}}\right)^2 =  2^{M_m} M_m^{2M_m}.
    \end{align*}
\end{proof}

\noindent{{\bf Remark.}} The estimate \eqref{eq: asymptotic coeff odd case} is not optimal as soon as $m\geq 5$. To get the precise growth of $M^{(6)}(n)$, we have to deal with the sum
\begin{align*}
    F(n) = \sum_{k=0}^{n} a_k^{(3)}(n) \binom{n+k}{k}^2 = \sum_{k=0}^{n} \binom{n}{k} \binom{n+k}{k}^2.
\end{align*}
By using Zeilberger's algorithm, we can show that $F(n)$ follows a linear recurrence of order $3$ whose characteristic polynomial is $2x^3-39x^2-5x-1$, and Poincaré-Perron theorem gives
\begin{align*}
    F(n)^{1/n} \leq 19.628\cdots \quad (\textrm{modulus of the largest root}).
\end{align*}
Finally, we can deduce that we actually have $M_\infty^{(6)} \leq 39,257\cdots$ instead of the upper bound $M_\infty^{(6)} \leq 64$ given by our proposition.

\subsection{Case of even $m$ and $\ell=0$} \label{subsection: even case}

The polynomial $H_{n,0}^{(m)}(z)$ is as in \eqref{eq: def H_ell}. Fix $n\geq 0$ and define for each $k\geq 0$ the polynomial
\begin{align*}
    U_k(z,n) = U_k(z) = (-1)^k(z-n/2)_k(1-z-n/2)_k = k!^2 \binom{z-n/2+k-1}{k}\binom{z+n/2-1}{k}.
\end{align*}
Given an even integer $m\geq 2$, we define by induction a sequence of coefficients $(a_j^{(m)}(n))_{0\leq j\leq N^{(m)}}$, where
\begin{align}\label{eq def: N^(m) even case}
    N^{(m)} = \frac{n(m-2)}{2},
\end{align}
by setting $a_0^{(2)}(n) = 1$ and by using the relation
\begin{align*}
        a_j^{(m+2)}(n) = \sum_{k=\max\{0,j-n\}}^{\min\{j,N^{(m)}\}} a_k^{(m)}(n) \binom{n+k}{n} \binom{n+k}{j}  \binom{n+j}{n+k} \qquad (0\leq j \leq N^{(m+1)}).
    \end{align*}
Similarly to the odd case, we have
\begin{align*}
    a_j^{(m)}(n) \in \Z_{>0} \qquad (0 \leq j \leq N^{(m)}).
\end{align*}

The next two results are analogs of Propositions~\ref{prop: rec coeff H^m odd case} and~\ref{prop: asymptotic coeff odd case} dealing with odd $m$.

\begin{proposition}\label{prop: rec coeff H^m even case}
    Let $n,m \geq 0$, with $m\geq 2$ even. Then
    \begin{align}\label{eq prop: rec coeff H^m even case}
        H_{n,0}^{(m)}(z) = \sum_{j=0}^{N^{(m)}}  a_j^{(m)}(n)\frac{U_{n+j}(z)}{j!(n+j)!}.
    \end{align}
\end{proposition}

\begin{proposition}\label{prop: asymptotic coeff even case}
    For any even positive integer $m$, writing $M=(m-2)/2$, we have
    \begin{align}\label{eq: asymptotic coeff even case}
        \limsup_{n\to\infty} \max_{0\leq j\leq N^{(m)}}|a_j^{(m)}(n)|^{1/n} \leq (2M)^M(M+1)^{M+1},
    \end{align}
    with the convention $0^0 = 1$.
\end{proposition}

Similarly to the odd case, we will prove Proposition~\ref{prop: rec coeff H^m even case} by induction on $m$. In the induction step, Lemma~\ref{lem: rec coeff H^m odd case} is now replaced by the following one.

\begin{lemma}\label{lem: rec coeff H^m even case}
    Let $\alpha$ be an indeterminate. For an integer $k\geq 0$, we define
    \begin{align*}
        S^{(\alpha)}_k(z) = S(k) = (-1)^k\frac{(z)_k(1-z+\alpha)_k}{k!^2} = \binom{-z}{k}\binom{k-z+\alpha}{k}.
    \end{align*}
    Let $k,\ell\geq 0$ be integers. Then
    \begin{align*}
        S_k(z)S_\ell(z) = \sum_{j=\max\{k,\ell\}}^{k+\ell} \binom{k+\ell+\alpha}{k+\ell-j} \binom{j}{k} \binom{j}{\ell}S_j(z).
    \end{align*}
    In particular
    \begin{align}\label{eq lem: product of U_k}
        \frac{U_k(z,n)}{k!^2}\frac{U_\ell(z,n)}{\ell!^2} = \sum_{j=\max\{k,\ell\}}^{k+\ell} \binom{k+\ell-n}{k+\ell-j} \binom{j}{k} \binom{j}{\ell} \frac{U_j(z,n)}{j!^2}.
    \end{align}
\end{lemma}

\begin{proof}
    We have to prove that $A_k = B_k$, where
    \begin{align*}
        A(k,\ell) &:= \binom{-z}{k}\binom{k-z+\alpha}{k}\binom{-z}{\ell}\binom{\ell-z+\alpha}{\ell}, \\
        B(k,\ell) &:= \sum_{j\in\Z} F(k,\ell,j),\qquad \textrm{where } F(k,\ell,j) = \binom{k+\ell+\alpha}{k+\ell-j} \binom{j}{k} \binom{j}{\ell}\binom{-z}{j}\binom{j-z+\alpha}{j}
    \end{align*}
    (with the convention that $\binom{z}{i}=0$ for each $i\in\Z_{<0}$). Note that the function $F(k,\ell,j)$ is hypergeometric with respect to $k$, $\ell$ and $j$, and $F(k,\ell,j) = F(\ell,k,j)$. Using Zeilberger's algorithm (\cite{ZeilbergerAeqB}), we find that
    \begin{align*}
        (k+1)^2F(k+1,\ell,j) + (z+k)(k+1+\alpha-z)F(k,\ell,j) = G(k,\ell,j+1)-G(k,\ell,j),
    \end{align*}
    where
    \[
        G(k,\ell,j):= -\frac{(j+\alpha)(j-k)(j-\ell)}{k+\ell+1-j}F(k,\ell,j).
    \]
    Taking the sum over $j\in\Z$, this shows that $B(k,\ell,z)$ satisfies the order one recurrence relation
    \begin{align*}
        (k+1)^2B(k+1,\ell) + (z+k)(k+1+\alpha-z)B(k,\ell) = 0.
    \end{align*}
    By symmetry, we also have
    \begin{align}\label{eq proof: rec symmetry in l}
        (\ell+1)^2B(k,\ell+1) + (z+\ell)(\ell+1+\alpha-z)B(k,\ell) = 0.
    \end{align}
    A very short computation shows that $(A(k,\ell))_k$ satisfies the same recurrence, so $A(k,\ell) = B(k,\ell)$ if and only if $A(0,\ell) = B(0,\ell)$. Now, taking $k=0$ in \eqref{eq proof: rec symmetry in l}, we obtain
    \begin{align*}
        (\ell+1)^2B(0,\ell+1) + (z+\ell)(\ell+1+\alpha-z)B(0,\ell) = 0,
    \end{align*}
    a relation that the sequence $(A(0,\ell))_\ell$ also satisfies. So $A(0,\ell) = B(0,\ell)$ is and only if $A(0,0) = B(0,0)$, which is trivially true.

    \medskip

    Finally, \eqref{eq lem: product of U_k} is a direct consequence of the identity
    \begin{align*}
       k!^2 S^{(-n)}_k\left(z -\frac{n}{2} \right) = U_k(z,n).
    \end{align*}
    obtained by taking $\alpha = -n$ and by shifting the variable $z$ by $-n/2$.
\end{proof}

\begin{proof}[{\bf Proof of Proposition~\ref{prop: rec coeff H^m even case}.}]
    We proceed by induction on $m$. For $m=2$, this is trivial since $U_n(z,n) = T_n(z)^2$. Suppose that \eqref{eq prop: rec coeff H^m even case} holds for an even integer $m\geq 2$. Then, by \eqref{eq lem: product of U_k} of Lemma~\ref{lem: rec coeff H^m even case}, we have
    \begin{align*}
        H_{n,0}^{(m+2)} = \frac{U_n(z)}{n!^2} H_{n,0}^{(m)} & = \sum_{k=0}^{N^{(m)}} a_k^{(m)}(n)\frac{U_n(z)}{n!^2}\frac{ U_{n+k}(z)}{k!(n+k)!} \\
        & = \sum_{k=0}^{N^{(m)}} \frac{(n+k)!}{k!} a_k^{(m)}(n) \sum_{j=n+k}^{2n+k} \binom{n+k}{2n+k-j} \binom{j}{n} \binom{j}{n+k}\frac{U_j(z)}{j!^2}.
    \end{align*}
    Now, notice that
    \begin{align*}
        \frac{(n+k)!}{k!}\binom{j}{n}\frac{U_j(z)}{j!^2} = \binom{n+k}{n}\frac{U_j(z)}{(j-n)!j!},
    \end{align*}
    hence
        \begin{align*}
        H_{n,0}^{(m+2)}  & = \sum_{j=n}^{N^{(m+1)}+n}  \left(\sum_{k=\max\{0,j-2n\}}^{\min\{j-n,N\}}a_k^{(m)}(n)  \binom{n+k}{n} \binom{n+k}{j-n} \binom{j}{n+k}\right) \frac{U_j(z)}{(j-n)!j!},
    \end{align*}
    and we conclude by re-indexing from $j=0$ to $N^{(m)}$.
\end{proof}

\begin{proof}[{\bf Proof of Proposition~\ref{prop: asymptotic coeff even case}}]
     Define
    \begin{align*}
        A^{(m)}(n) = \max_{0\leq j \leq N^{(m)}}\left\{a_j^{(m)}(n)\right\} \AND A_\infty^{(m)} = \limsup_{n\to\infty} A^{(m)}(n)^{1/n}.
    \end{align*}
    Given $n\geq 0$, we have
    \begin{align*}
         A^{(m+2)}(n) \leq \sum_{j=0}^{N^{(m+2)}} a_j^{(m+2)}(n) & = \sum_{j=0}^{N^{(m+2)}} \sum_{k=\max\{0,j-n\}}^{\min\{j,N^{(m)}\}} a_k^{(m)}(n) \binom{n+k}{n} \binom{n+k}{j}  \binom{n+j}{n+k} \\
         & =  \sum_{k=0}^{N^{(m)}} a_k^{(m)}(n) \binom{n+k}{n} \sum_{j=0}^{k} \binom{n+k}{j}  \binom{n+j}{n+k}.
    \end{align*}
    Now, note that
    \begin{align*}
        \binom{n+k}{j}  \binom{n+j}{n+k} = \binom{n+j}{n}  \binom{n}{j-k} \leq 2^n \binom{n+j}{n}.
    \end{align*}
    Since $t\mapsto \binom{n+t}{n}$ is increasing, we deduce that
    \begin{align*}
        A^{(m+2)}(n) & \leq   2^n A^{(m)}(n)  \sum_{k=0}^{N^{(m)}} \binom{n+k}{n} \sum_{j=k}^{n+k}\binom{n+j}{n} \\
         & \leq 2^n (N^{(m)}+1) A^{(m)}(n) \binom{N^{(m)}+n}{n} \binom{N^{(m)}+2n}{n}.
    \end{align*}
    Using Stirling's formula and writing $M_m = N^{(m)}/n = (m-2)/2$, this yields
    \begin{align*}
        A_\infty^{(m+2)} & \leq   2 \frac{(M_{m+2}+1)^{M_{m+2}+1}}{(M_{m+2}-1)^{M_{m+2}-1}}A_\infty^{(m)}
    \end{align*}
    (since $M_{m+2} = M_m+1$). By induction, we get
    \begin{align*}
        M_\infty^{(m)} & \leq   2^{M_m}\prod_{b=1}^{M_m} \frac{(b+1)^{b+1}}{(b-1)^{b-1}} = 2^{M_m}M_m^{M_m}(M_m+1)^{M_m+1}.
    \end{align*}
\end{proof}

\noindent{\bf Remark.} Once again, \eqref{eq: asymptotic coeff even case} is not optimal starting $m = 4$. It is possible to get more precise results using Zeilberger's algorithm. For example, for $m=4$, since $N^{(2)} = 0$, we have the explicit formula
\begin{align*}
    a_j^{(4)}(n) = \binom{n}{j}  \binom{n+j}{n} \qquad (0\leq j \leq n),
\end{align*}
and Zeilberger's algorithm ensures that $F(j) = \sum_{j=0}^n  a_j^{(4)}(n)$ satisfies a linear recurrence of order $2$, whose characteristic polynomial is $(x+1)(x^2-6x+1)$. By Poincaré-Perron theorem, we deduce that
\begin{align*}
    M_\infty^{(4)} = 3+2\sqrt 2 = 5.828\cdots
\end{align*}
while Proposition~\ref{prop: asymptotic coeff even case} yields $M_\infty^{(4)} \leq 8$.

\subsection{Reduction to the case $\ell=0$} \label{subsection: reduction ell =0}

In this section, we explain how the study of the polynomials $H_{n,\ell}^{(m)}(z)$ defined in \eqref{eq: def H_ell} is reduced to that of $H_{n,0}^{(m)}(z)$. Of course, we have to distinguish between odd and even $m$. We keep the notation of Subsections~\ref{subsection: odd case} and~\ref{subsection: even case} for the sequence $(a_{j}^{(m)}(n))_j$ (and the polynomials $U_k(z,n)$ introduced in Section~\ref{subsection: even case}). Recall that the integer $N^{(m)}$ is defined by \eqref{eq def: N^(m) even case} (resp. \eqref{eq def: N^(m) odd case}) if $m$ is even (resp. odd).

\begin{proposition}\label{prop: rec coeff H^m even case gen case}
    Let $m, M$ be integers with $m\geq 2$ even and $M=(m-2)/2$. Then, for any $\ell$ with $0 \leq \ell \leq M$, and any integer $n\geq M$, there exist coefficients $(a_{j,\ell}^{(m)}(n))_{0\leq j \leq N_\ell}$, where $N_\ell = Mn+\ell$, such that
    \begin{align*}
        H_{n,\ell}^{(m)}(z) = \sum_{j=0}^{N_\ell}  a_{j,\ell}^{(m)}(n)\frac{U_{n+j}(z)}{j!(n+j)!}.
    \end{align*}
    Furthermore, for $j=0,\dots,N_\ell$, we have $(-1)^\ell a_{j,\ell}^{(m)}(n) \in \Z_{>0}$ and
    \begin{align*}
        \limsup_{n\to\infty} \max_{0\leq j\leq N_\ell}|a_{j,\ell}^{(m)}(n)|^{1/n} \leq \limsup_{n\to\infty} \max_{0\leq j\leq N_\ell} \big\{|a_{j}^{(m)}(n)|^{1/n}\big\} \leq (2M)^M(M+1)^{M+1}. 
    \end{align*}
\end{proposition}

\begin{proof}
    By Proposition~\ref{prop: rec coeff H^m even case} combined with Lemma~\ref{lem: rec coeff H^m even case}, we find
    \begin{align*}
        H_{n,\ell}^{(m)}(z) = (-1)^\ell H_{n,0}^{(m)}(z) U_\ell(z,n) &=  (-1)^\ell \ell!^2 \sum_{j=0}^{N^{(m)}} a_j^{(m)}(n)\frac{U_{n+j}(z)}{j!(n+j)!}\frac{U_\ell(z,n)}{\ell!^2} \\
        & = (-1)^\ell\sum_{j=0}^{N^{(m)}} a_j^{(m)}(n) \sum_{i = 0}^{\ell} b_{i}(n,j) \frac{U_{n+i+j}(z)}{(i+j)!(n+i+j)!},
    \end{align*}
    where $N^{(m)} = Mn = N_0$ and
    \begin{align*}
        b_{i}(n,j)  &= \ell!^2\frac{(n+j)!}{j!} \binom{j+\ell}{\ell-i} \binom{n+i+j}{i}\binom{n+i+j}{\ell} \frac{(i+j)!}{(n+i+j)!} \\
        &= \ell!^2\binom{j+\ell-i}{j} \binom{j+\ell}{i}\binom{n+i+j}{\ell} \in \Z_{>0}.
    \end{align*}
    Furthermore
    \[
        \limsup_{n\to\infty} \max_{\substack{0 \leq i \leq \ell \\ 0\leq j\leq N^{(m)}}} \big\{ b_{i}(n,j)^{1/n} \big\} = 1.
    \]
    We conclude with Propositions~\ref{prop: rec coeff H^m even case} and~\ref{prop: asymptotic coeff even case}.
\end{proof}

\begin{proposition}\label{prop: rec coeff H^m odd case gen case}
    Let $m, M$ be integers with $m\geq 3$ odd and $M=(m-1)/2$. Then, for any $\ell$ with $0 \leq \ell \leq M$, and any even integer $n\geq M$, there exist coefficients $(a_{j,\ell}^{(m)}(n))_{0\leq j \leq N_\ell}$, where $N_\ell = Mn+\ell$, such that
    \begin{align*}
        H_{n,\ell}^{(m)}(z) = \sum_{j=0}^{N_\ell}  a_{j,\ell}^{(m)}(n)\frac{T_{n+2j}(z)}{j!^2}.
    \end{align*}
    Furthermore, for $j=0,\dots,N_\ell$, we have $a_{j,\ell}^{(m)}(n) \in \Z_{>0}$ and
    \begin{align*}
        \limsup_{\substack{n\to\infty \\ \textrm{$n$ even}}} \max_{0\leq j\leq N_\ell}\big\{|a_{j,\ell}^{(m)}(n)|^{1/n}\big\} \leq \limsup_{n\to\infty} \max_{0\leq j\leq N_\ell}|a_{j}^{(m)}(n)|^{1/n} \leq   (2M)^M M^{M}. 
    \end{align*}
\end{proposition}

\begin{proof}
    Applying Lemma~\ref{lem: rec coeff H^m even case} with $\alpha=0$ (and noting that $k!^2S^{(0)}_k(z) = T_{2k}(z)$ for each integer $k$), we obtain, for each $N,n,j,\ell \geq $ with $n=2N$ even and $\ell \leq M \leq n$,
    \begin{align*}
        \frac{T_{n+2j}(z)}{(N+j)!^2}\frac{T_{2\ell}(z)}{\ell!^2} = \sum_{i=0}^{\ell} \binom{N+j+\ell}{\ell-i} \binom{N+i+j}{i} \binom{N+i+j}{\ell}\frac{T_{n+2i+2j}(z)}{(N+i+j)!^2}.
    \end{align*}
    By Proposition~\ref{prop: rec coeff H^m odd case} combined with the above, we get
    \begin{align*}
        H_{n,\ell}^{(m)}(z) = H_{n,0}^{(m)}(z) T_{2\ell}(z)  = \sum_{j=0}^{N^{(m)}} a_j^{(m)}(n)\frac{ T_{n+2j}(z)}{j!^2}T_{2\ell}(z)= \sum_{j=0}^{N^{(m)}} a_j^{(m)}(n) \sum_{i=0}^{\ell} b_i(n,j)  \frac{T_{n+2i+2j}(z)}{(i+j)!^2},
    \end{align*}
    where
    \begin{align*}
        b_i(n,j) & = \ell!\frac{(i+j)!^2(N+j)!^2}{(N+i+j)!^2j!^2} \binom{N+j+\ell}{\ell-i} \binom{N+i+j}{i} \binom{N+i+j}{\ell} \\
        & = \binom{N+j}{\ell-i}\binom{N+j+\ell}{\ell-i} \binom{i+j}{i}^2(\ell-i)!i! \in \Z_{>0}.
    \end{align*}
    We deduce that
    \[
        \limsup_{n\to\infty} \max_{\substack{0 \leq i \leq \ell \\ 0\leq j\leq N^{(m)}}} \big\{ b_{i}(n,j)^{1/n} \big\} = 1,
    \]
    and we conclude with Proposition~\ref{prop: rec coeff H^m odd case} and~\ref{prop: asymptotic coeff odd case}.
\end{proof}

\noindent{\bf Remark.} The condition $n$ even in Proposition~\ref{prop: rec coeff H^m even case gen case} is sufficient for our purpose. In order to treat the case of odd indices $n$, we would need to be able to express a product $T_{n+2j}(z)T_{2\ell}(z)$ in terms of $T_{n+2i+2j}$. We did not do it in order to keep arguments as short as possible. However, curious readers may try to prove that
\begin{align*}
    \frac{T_{2n+1}(z)}{(2n+1)!}\frac{T_{2\ell}(z)}{\ell!^2} = \sum_{k=0}^{\ell} \binom{n+1/2+\ell}{2\ell-k} \binom{n+k}{k} \binom{2\ell-k}{\ell} 4^k \frac{T_{2n+1+2k}(z)}{(2n+1+2k)!}.
\end{align*}

\subsection{Proof of Theorem~\ref{thm: property of the coeff of P_n}} \label{subsection : proof thm coeff P_n}

\noindent{\bf Even case.} We first suppose that $m\geq 2$ is even and fix an integer $n\geq 0$. By Proposition~\ref{prop: rec coeff H^m even case gen case}, we can write
\begin{align*}
    H_{n,\ell}^{(m)}(z) = \sum_{j=0}^{N_\ell}  a_{j,\ell}^{(m)}(n)\frac{U_{n+j}(z,n)}{j!(n+j)!},
\end{align*}
where $N_\ell = Mn+\ell$. The key-point is to notice that for any integer $k\geq 1$ and any $\alpha, \beta\in K$ with $\beta = \alpha+1/2$, we have the identity
\begin{align*}
    R_{-1} \Big((z+\alpha)_k(1+\alpha-z)_k \Big) = -k  \left(z+\beta\right)_{k-1}  \left(1+\beta-z\right)_{k-1}.
\end{align*}
By induction, if follows that for any $k\geq n$,
\begin{align*}
    R_{-1}^n(U_k(z,n)) = (-1)^{k+n} \frac{k!}{(k-n)!} (z)_{k-n}(1-z)_{k-n}.
\end{align*}
We deduce that
\begin{align*}
     P_{m,n,\ell}(z) = R_{-1}^n\left(H_{n,\ell}^{(m)}(z)\right) 
     = \sum_{j=0}^{N_\ell}  (-1)^{j} a_{j,\ell}^{(m)}(n)\frac{(z)_{j}(1-z)_{j}}{j!^2}.
\end{align*}
For the asymptotic estimates, we conclude with the last part of Proposition~\ref{prop: rec coeff H^m even case gen case}.

\medskip

\noindent{\bf Odd case.} Suppose that $m \geq 3$ is odd and fix an even integer $n\geq 0$. By Proposition~\ref{prop: rec coeff H^m odd case gen case}, we can write
\begin{align*}
    H_{n,\ell}^{(m)}(z) = \sum_{j=0}^{N_\ell}  a_{j,\ell}^{(m)}(n)\frac{T_{n+2j}(z)}{j!^2},
\end{align*}
where $N_\ell = Mn+\ell$. A short computation provides the identity $R_1(T_k(z)) = T_{k-1}(z)$ for each integer $k\geq 1$. We deduce that
\begin{align*}
     P_{m,n,\ell}(z) = R_{1}^n\left(H_{n,\ell}^{(m)}(z)\right) = \sum_{j=0}^{N_\ell}  a_{j,\ell}^{(m)}(n)\frac{T_{2j}(z)}{j!^2} = \sum_{j=0}^{N_\ell} (-1)^j a_{j,\ell}^{(m)}(n)\frac{(z)_j(1-z)_j}{j!^2},
\end{align*}
and we conclude with the part of Proposition~\ref{prop: rec coeff H^m odd case gen case} for the asymptotic estimates.


\section{Estimates} \label{section: estimates}

We keep the notation of Section~\ref{section: Construction of Pade} and Definition~\ref{def: construction Pade approximants, choice of A_ell}, where we construct Padé approximants for the systems $(\Theta_2, \Theta_4, \ldots, \Theta_{m-1})$ and $(R_3,R_5, \ldots, R_{m-1})$) when $m$ is odd (resp. even). As usual, if $m$ is even (resp. odd), then the set
\begin{align*}
    S_m \subset \{2,\dots,m-1\}
\end{align*}
is the set of odd (resp. even) indices in $\{2,\dots,m-1\}$.\\
This section is organized as follows. Given a fixed rational number $x$, we will estimate, as $n$ tends to infinity,
\begin{itemize}
  \item the archimedean growth of $P_{m,n,\ell}(x)$ and $Q_{m,n,s,\ell}(x)$ (see Subsection~\ref{subsection: estimate approximants});
  \item the archimedean growth of the common denominator of $P_{m,n,\ell}(x)$ and $Q_{m,n,s,\ell}(x)$ (see Subsection~\ref{subsection: estimate denominator}).
  \item the $p$-adic growth of the remainder $\fR_{m,n,s,\ell}(x) = P_{m,n,\ell}(x)F_{s}(x)-Q_{m,n,s,\ell}(x)$ (see Subsection~\ref{subsection: estimate approximations}).
\end{itemize}
All those estimates will be crucial in the proof of our main theorem in the last Section~\ref{section: Proof of theorem}.

\subsection{Preliminaries}

In this section, we establish upper bounds for the archimedean and $p$-adic norm of the coefficients $b_{s,k}$ defined by~ \eqref{expansion F_s}. We will use the archimedean estimates in the next subsection in order to bound from above the quantities $|Q_{m,n,s,\ell}(x)|$ as $n$ tends to infinity. The $p$-adic estimates will be useful in Subsection~\ref{subsection: estimate approximations}, where we study the $p$-adic convergence of the remainder $\fR_{m,n,s,\ell}(x)$. We just need estimates of the form $|b_{s,k}|_p = p^{o(k)}$ as $k$ tends to infinity, that is why we are not trying to be subtle in the proof of Proposition~\ref{prop: first estimates of b_s,k} below. In Subsection~\ref{subsection: estimate denominator} we will determine a much better denominator of $b_{s,k}$ (see Lemma~\ref{den bsk}). Given an integer $n\geq 1$, we set
\begin{align*}
    d_n = \lcm(1,\dots,n).
\end{align*}
According to \cite[Lemma $3.1$]{P}, we have
\begin{align}\label{eq: denom Prevost inverse binomial}
    \frac{d_n}{j}\binom{n}{j}^{-1} &\in \Z \qquad ( 1\le j \le n).
\end{align}

\begin{proposition}\label{prop: first estimates of b_s,k}
    Let $s\geq 2$ be an integer and $p$ be a prime number. For any integer $k\geq 0$, we have
    \begin{align*}
        \frac{(2k+1)!}{k!^2}|b_{s,k}| &\le (4(k+1))^{s},\\
        (d_{k+1}d_{2k+1})^{s/2} b_{s,k}&\in\Z \AND |b_{s,k}|_p \leq p^{(s+1)\log_p(2k+1)}.
    \end{align*}
\end{proposition}

\begin{proof}
    Write $s = 2+i+2t$, where $i\in\{0,1\}$ and $t=\lfloor (s-2)/2\rfloor$. Let $k\geq t$ be integer. The set $J_t(k)$ of $(t+1)$--tuples $(j_0,\dots,j_t)\in\Z_{\geq 0}^{t+1}$ such that $j_0+\cdots+ j_t = k-t$ is bounded from above by $(k+1)^{t+1}$ (since for each such tuple satisfies $0\leq j_0,\dots,j_t \leq k$). Given $(j_0,\dots,j_t)\in J_t(k)$, we have (using the estimate $\binom{2n+1}{n+1} \geq 2^{2n+1}/(2n+2) = 4^n /(n+1)$)
    \begin{align}\label{eq proof: estimate b_s,j}
        \frac{j_i!^2}{(j_i+1)(2j_i+1)!} = \frac{1}{(j_i+1)^2} \binom{2j_i+1}{j_i+1}^{-1} \leq \frac{4}{(j_i+1)4^{j_i+1}}
    \end{align}
    for $i=0,\dots,t$. Hence, since $(j_0+1)+\cdots + (j_t+1)= k+1$,
    \begin{align*}
        (j_t+1)\prod_{i=0}^{t}\frac{j_i!^2}{(j_i+1)(2j_i+1)!} \leq \frac{4^{t+1}}{4^{k+1}},
    \end{align*}
    and we deduce, in view of \eqref{eq: eq b_s,k, s odd} and \eqref{eq: eq b_s,k, s even}, that $ |b_{s,k}| \leq (4(k+1))^{t+1} / 4^{k+1}$. We conclude by using the crude upper bound $(2k+1)!/k!^2 \leq (k+1)\binom{2k+1}{k} \leq (k+1)4^{k+1}$ (and by noting that $t+2\leq s$).

    \medskip

    Let $(j_0,\dots,j_t)\in J_t(k)$. Eq.~\eqref{eq: denom Prevost inverse binomial} combined with \eqref{eq proof: estimate b_s,j} ensures that $(d_{k+1}d_{2k+1})^{t+1}b_{s,k}$ is an integer. Since $2(t+1) \leq s$ and $d_{k+1}$ divides $d_{2k+1}$, we find the crude upper bound
    \begin{align*}
        |b_{s,k}|_p \leq |d_{2k+1}|_p^{-s} \leq p^{s\log_p (2k+1)}.
    \end{align*}
    (we used the formula $|d_N|_p = p^{-\lfloor \log_p(N) \rfloor}$ for the last inequality).
\end{proof}

To obtain a better common denominator of the coefficients $b_{s,k}$ in Subsection~\ref{subsection: estimate denominator}, we will need the following beautiful estimate, established in \cite[Theorem 8.1]{gaudron2013minima}. Given positive integers $r,K$, set
\begin{align*}
    d(r,K) = \lcm\left( \binom{K}{i_0,\dots,i_{r-1}} = \frac{K!}{i_0!\cdots i_{r-1}!}\ \;;\; i_0,\dots,i_{r-1} \in \Z_{\geq 0} \AND i_0+\cdots + i_{r-1} = K \right).
\end{align*}

\begin{theorem}\label{thm: lcm multinom coeff}
    Let $K,r$ be positive integers. Then
    \begin{align*}
        d(r,K) = \prod_{j=1}^{r-1} \frac{d_{\lfloor(K+r-1)/j\rfloor}}{K+j}.
    \end{align*}
\end{theorem}

We will use the following form of the above theorem.

\begin{corollary}\label{cor: lcm multinom coeff}
    Let $k,t$ and $j_0,\dots,j_t\geq 0$ be non-negative integers with $k\geq t$ and $j_0+\cdots + j_t = k-t$. Then
    \begin{align*}
        \left(\prod_{j=1}^{t} d_{\lfloor k/j\rfloor}\right) \frac{\prod_{i=0}^{t} j_i!}{k!}  \in \Z.
    \end{align*}
\end{corollary}

\begin{proof}
    According to Theorem~\ref{thm: lcm multinom coeff} with $K=k-t$ and $r=t+1$, the integer
    \begin{align*}
        \frac{k!}{\prod_{i=0}^{t} j_i!} = \binom{k-t}{j_0,\dots,j_t} \prod_{j=1}^{t}(k-t+j),
    \end{align*}
    divides $\prod_{j=1}^{t} d_{\lfloor k/j\rfloor}$.
\end{proof}

\subsection{Absolute values of the Pad\'{e} approximants} \label{subsection: estimate approximants}


\begin{proposition} \label{prop: estimate Pade via auxiliary seq}
    Let $x \in \Q$  and $m\in\Z_{\geq 3}$.  For each integer $\ell= 0,\dots, M:= \lfloor (m-1)/2 \rfloor$, we have
    \[
        \limsup_{\substack{n \to \infty \\ \textrm{$n$ even}}} \max_{s\in S_m}
        \bigl\{ |P_{n,\ell}(x)|,\, |Q_{m,n,s,\ell}(x)| \bigr\}^{1/n} \le C(m),
    \]
    where $C(m)$ is defined in Theorem~\ref{thm: property of the coeff of P_n}.
\end{proposition}

\begin{remark}
    In Section~\ref{section: appendix rec}, we will explain how we can improve the above proposition. For $m=4$ and $m=6$, we may take $C(4) = C(6) = 1$, so that $C(m)$ can replace the constant $C_0(m)$ defined in \eqref{eq: def C(m)}.
\end{remark}


\begin{proof}[\textbf{Proof of Proposition~\ref{prop: estimate Pade via auxiliary seq}}]
    Let $k\ge 0$ be an integer and $\alpha \in \C$. Then \cite[Lemma~7.8]{KP2} implies that
    \begin{align}\label{estimate binomial}
        \frac{|(\alpha)_k|}{k!} \le (ek)^{|\alpha|},
    \end{align}
    (where $e = \exp(1)$). Let $x \in \Q \setminus \Z_{\le 0}$. By \eqref{estimate binomial} we have
    \begin{align*}
        \max_{0\leq k \leq N} \left|\dfrac{(x)_{k}}{k!}\dfrac{(1-x)_{k}}{k!}\right| \leq (eN)^{2|x|+1}.
    \end{align*}
    Combined with Theorem~\ref{thm: property of the coeff of P_n}, we deduce that $\limsup_{\substack{n \to \infty \\ \textrm{$n$ even}}}|P_{m,n,\ell}(x)|^{1/n} \leq C(m)$. Now, note that for each $j,k\geq 0$ with $j\leq k-1$, we have
    \begin{align*}
        \left|\binom{k}{j+1}^{-2} \frac{(x+j+1)_{k-j-1}}{(k-j-1)!} \frac{(2+j-x)_{k-j-1}}{(k-j-1)!} \right|
        & \leq \left(\frac{(j+1)!(|x|+j+2)_{k-j-1}}{k!} \right)^2 \\
        & \leq \left(\frac{(|x|+1)_k}{k!} \right)^2 \leq (ek)^{2|x|+2},
    \end{align*}
    the last inequality coming from~\eqref{estimate binomial}. Recall that $N_\ell = Mn+\ell$. Using the expansion of $Q_{m,n,s,\ell}(z)$ provided by Theorem~\ref{thm:expansion} with the above, we get
    \begin{align*}
        |Q_{m,n,s,\ell}(x)|  & \leq \max_{1\leq k\leq N_\ell}  |p_{k,\ell}^{(m)}(n)| \sum_{k=1}^{N_\ell} \sum_{j=0}^{k-1} |b_{s,j}| \frac{(2j+1)!}{j!^2} (ek)^{2|x|+2}.
    \end{align*}
    Using Proposition~\ref{prop: first estimates of b_s,k} together with Theorem~\ref{thm: property of the coeff of P_n}, this yields $\limsup_{\substack{n \to \infty \\ \textrm{$n$ even}}}|Q_{m,n,s,\ell}(x)|^{1/n} \leq  C(m)$.
\end{proof}

\subsection{Denominators of the Pad\'{e} approximants} \label{subsection: estimate denominator}

We keep the notation introduced at the beginning of Section~\ref{section: estimates}. In this section, $n,m,\ell\geq 0$ are fixed integers, with $n$ even, $m\geq 3$ and $0\le \ell \le M:=\lfloor (m-1)/2 \rfloor$. Our goal here is to bound from above the common denominator of $P_{n,\ell}(x)$ and $Q_{m,n,s,\ell}(x)$ as a function of $x\in \Q^\times$, see Proposition~\ref{den ask} below. For that, we need the following notation.
Recall that
\begin{align*}
    N_\ell= Mn+\ell.
\end{align*}
Given a positive integer $N$ and a rational number $x$, we set $d_N = \lcm(1,\dots,N)$,
\begin{align*}
    \den(x) &=\min\{ n\in \Z_{\geq 1} \,;\,  nx\in \Z \}, \notag \\
    \mu_{N}(x) &=\den(x)^N\prod_{\substack{q:\text{prime} \\ q|\den(x)}}q^{\lfloor N/(q-1)\rfloor} \AND \mu(x)  =  \den(x)\prod_{\substack{q:{\rm{prime}} \\ q|{\den}(x)}}q^{\tfrac{1}{q-1}}. 
\end{align*}
Note that for any integer $\ell$ and any rational number $x$, we have $\den(x+\ell) = \den(x)$. Consequently, we also have $\mu_n(x+\ell) = \mu_n(x)$.

\begin{proposition} \label{den ask}
    Let $x\in \Q^\times$ and let $n,m,\ell, M = \lfloor(m-1)/2\rfloor$ be as at the beginning of this subsection. Define $\fD_m(n) = 2$ (resp. $\fD_m(n) = d_{n+1}$) if $m$ is odd (resp. even), and set
    \begin{align} \label{Dn(x)}
        D_n(x) &= d^2_{N} \mu_{N}(x)^2 \fD_m(n) \left(\prod_{j=1}^{M-1}d_{ N/j}\right)^2,
    \end{align}
    where $N=M(n+1)$. Then
    \begin{align*}
        D_n(x) P_{m,n,\ell}(x)\in \Z \AND D_n(x)Q_{m,n,s,\ell}(x)\in \Z \qquad (s\in S_m).
    \end{align*}
    Furthermore, we have
    \begin{align} \label{eq: asympt D_n(x)}
        \limsup_{\substack{ n \to\infty \\ \textrm{$n$ even}}} D_n(x)^{1/n} &= \mu(x)^{2M}\exp\Big(m-1 + \sum_{j=1}^{M-1} \frac{2M}{j}\Big).
    \end{align}
    If $|x|_p > 1$, we also have
    \begin{align} \label{eq: asympt D_n(x) p-adic}
        \limsup_{\substack{ n \to\infty \\ \textrm{$n$ even}}} |D_n(x)|_p^{1/n} &= |\mu(x)|_p^{2M} = |x|_p^{-2M} p^{-\frac{2M}{p-1}}.
    \end{align}
\end{proposition}

\begin{remark}
    Note that the denominator $D_n(x)$ does not depend on $\ell$.
\end{remark}

We prove Proposition~\ref{den ask} at the end of this subsection. In order to control the denominators, we will frequently use the following result, which is a direct consequence of \cite[Lemma~2.2]{B}. For any non-negative integer $N$ and any $x\in\Q$, we have
\begin{align}
    \mu_N(x) \dfrac{(x)_k}{k!} &\in \Z \qquad (0\leq k\leq N) \label{well-known}.
\end{align}


\begin{lemma} \label{den bsk}
    Given integers $t,k\geq 0$, define the integers
    \begin{align*}
        \fD_{0,t}(k) = 2\left(\prod_{j=1}^{t}d_{\lfloor (k+1)/ j\rfloor}\right)^2 \AND \fD_{1,t}(k) = 2^{-1}d_{\lfloor (k+1)/(t+1) \rfloor} \fD_{0,t}(k).
    \end{align*}
    Then $\fD_{i,t}(k)$ divides $\fD_{i,t+1}(k)$ and $\fD_{i,t}(k+1)$ for $i=0,1$. Furthermore, we have
    \begin{align*}
        \fD_{i,t}(k)\frac{(2k+1)!}{k!^2}b_{2+i+2t,k} \in \Z \qquad (i=0,1).
    \end{align*}
\end{lemma}

\begin{proof}
    Given $j_0, \cdots, j_t \geq 0$ be integers such that $j_0+\cdots + j_t = k-t$ and using the identity  $(j+1)(2j+1)! = (2j+2)!/2$, we find
    \begin{align*}
        \frac{(2k+1)!}{k!^2}(j_t+1)\prod_{i=0}^{t} \frac{j_i!^2}{(j_i+1)(2j_i+1)!} = K(j_0,\dots,j_t) \left(\frac{\prod_{i=0}^{t} j_i!}{k!}\right)^2,
    \end{align*}
    where $K(j_0,\dots,j_t) = 2^t \binom{2k+1}{2j_0+2,\dots,2j_{t-1}+2,2j_t+1} \in \Z$. Combined with Corollary~\ref{cor: lcm multinom coeff} and in view of~\eqref{eq: eq b_s,k, s even}, we obtain, as expected,
    \begin{align*}
        \fD_{0,t}(k)\frac{(2k+1)!}{k!^2}b_{2+2t,k}\in \Z.
    \end{align*}
    Let $j_0, \cdots, j_t \geq 0$ be integers satisfying $j_0+\cdots + j_t = k+1-t$. In order to estimate the denominator of $b_{3+2t,k}$, we need to study the term
    \[
        \frac{(2k+1)!}{k!^2}\prod_{i=0}^{t} \frac{j_i!^2}{(j_i+1)(2j_i+1)!}.
    \]
    By symmetry, we may assume without loss of generality that $j_0\geq \cdots \geq j_t$, so that $(t+1)(j_t+1) \leq \sum_{i=0}^{t}(j_i+1) = k+1$. By the above, we have
    \begin{align*}
        \frac{(2k+1)!}{k!^2}\prod_{i=0}^{t} \frac{j_i!^2}{(j_i+1)(2j_i+1)!}  = \frac{K(j_0,\dots,j_t)}{j_t+1}\left(\frac{\prod_{i=0}^{t} j_i!}{k!}\right)^2. 
    \end{align*}
     and using once again Corollary~\ref{cor: lcm multinom coeff}, we find $\fD_{1,t}(k)\frac{(2k+1)!}{k!^2} b_{3+2t,k}\in \Z$.
\end{proof}

\begin{proof}[\textbf{Proof of Proposition~\ref{den ask}}]
    By Theorem~\ref{thm: property of the coeff of P_n} combined with \eqref{well-known}, we obtain
    \[
        \mu_{N_\ell}(x)^2P_{m,n,\ell}(x)\in \Z, \qquad \textrm{where } N_\ell = nM +\ell.
    \]
    For $Q_{m,n,s,\ell}(x)$, we will use the expansion~\eqref{eq thm: expansion Q_s} established in Theorem~\ref{thm:expansion}, namely
    \begin{align*}
        Q_{m,n,s,\ell}(z) = \sum_{k=1}^{N_\ell}p_{k,\ell}^{(m)}(n)\sum_{j=0}^{k-1} \frac{b_{s,j}}{(j+1)^2} \frac{(2j+1)!}{j!^2} \binom{k}{j+1}^{-2} \dfrac{(z+j+1)_{k-j-1}}{(k-j-1)!} \dfrac{(2+j-z)_{k-j-1}}{(k-j-1)!},
    \end{align*}
    (where the coefficients $p_{k,\ell}^{(m)}(n)$ are all integers according to  Theorem~\ref{thm: property of the coeff of P_n}). Eq.~\eqref{eq: denom Prevost inverse binomial} implies that
    \begin{align*} 
        d^2_{N_\ell}\frac{1}{(j+1)^2}\binom{k}{j+1}^{-2} \in \Z \qquad \text{for} \ 1\le k \le N_\ell \AND \ 0\le j \leq k-1.
    \end{align*}
    Combined with Lemma~\ref{den bsk} and~\eqref{well-known}, we finally find, writing $s=2+i+2t$ with $i\in\{0,1\}$ and $t\geq 0$,
    \begin{align*}
        d^2_{N_\ell} \mu_{N_\ell}(x)^2 \fD_{i,t}(N_\ell-1) Q_{m,n,s,\ell}(x) \in \Z.
    \end{align*}
    The upper bound $s\leq m-1$ yields $t\leq \lfloor (m-3)/2 \rfloor = M-1$. On the other hand $N_\ell \leq nM+M = (n+1)M$. We deduce that $\fD_{i,t}(N_\ell-1)$ divides
    \begin{align*}
         \fD_{i,M-1}((n+1)M-1) = \fD_m(n) \left(\prod_{j=1}^{M-1}d_{ M(n+1)/j}\right)^2.
    \end{align*}
    Estimate \eqref{eq: asympt D_n(x)} follows from $\mu_N(x)^{1/N}\sim \mu(x)$ together with $d_N^{1/N} \sim e$ as $N$ goes to infinity (this is a consequence of the prime number theorem).\\
    To get \eqref{eq: asympt D_n(x) p-adic}, it suffices to notice that $|d_k|_p = p^{-\lfloor \log_p(k) \rfloor} = p^{o(k)}$ as $k$ tends to infinity, and that $|\den(x)|_p = |x|_p^{-1}$ under the assumption $|x|_p > 1$.
\end{proof}

\subsection{$p$-adic absolute value of the Pad\'{e} approximations} \label{subsection: estimate approximations}

The main result of this subsection is the following one.

\begin{proposition} \label{remainder est}
    Let $p$ be a prime number and let $x\in \Q$ be such that $|x|_p > 1$. Let $m \geq 3$ be an integer. For each $s\in S_m$ and  $\ell= 0,\dots,M:= \lfloor (m-1)/2 \rfloor$, the series $\fR_{m,n,s,\ell}(x)$ converges in $\Q_p$, and
    \begin{align*}
        \limsup_{\substack{n\to\infty\\ \textrm{$n$ even}}}\big|\fR_{m,n,s,\ell}(x)\big|_p^{1/n} \le p^{-\tfrac{2}{p-1}}  |x|^{-2}_p.
    \end{align*}
\end{proposition}

\begin{proof} [\textbf{Proof of Proposition~\ref{remainder est}}]
    Fix integers $s,\ell,n \geq 0$ with $n$ even, $s\in S_m$ and $0 \le \ell \le M$. For $k\geq n$, we set
    \[
        C_k(t) = (2t+n-1) (t+n)_{k-n}(1 - t)_{k-n} \frac{(t)^m_n}{n!^m} \tA_\ell(t + n/2),
    \]
    where
    \begin{align*}
        \tA_\ell(z) =
        \begin{cases}
            (z-n/2)_\ell (1-z-n/2)_\ell & \textrm{if $m$ is even},\\
            (z-\ell)_{2\ell} & \textrm{if $m$ is odd},
        \end{cases}
    \end{align*}
    and define $c(k) = \deg(C_k) = mn+2(k-n+\ell)+1$. Note that $\tA_\ell(z) = n!^{m-1}A_\ell(z)$, where the polynomial $A_\ell(z)$ is as in Definition~\ref{def: construction Pade approximants, choice of A_ell}). Since $n$ is even, we have $\tA_\ell(t+n/2)\in \Z[t]$. According to~\eqref{eq thm: expansion remainder R_s} of Theorem~\ref{thm:expansion}
    \begin{align*}
        \fR_{m,n,s,\ell}(z) & = \dfrac{n!^2}{2}\sum_{k=n}^{\infty} \binom{k}{n} \dfrac{r_{k,s}}{(z)_{k+1}(1-z)_{k+1}}, \qquad\textrm{where } r_{k,s}=\sum_{j=0}^{c(k)} a_{s,j}\cdot \Eval_{t=0}\circ \bDelta_{-1}^j (C_{k}(t)).
    \end{align*}
    Let $k \ge n$ be an integer and $j\in\{0,\dots,c(k)\}$. Since the polynomial $T(t) := (t)_n^m /n!^m$ satisfies $T(\Z)\subset \Z$, we clearly have $\Eval_{t=0}\circ \bDelta_{-1}^j (C_{k}(t)) \in \Z$. Eq.~\eqref{eq: a_s,k expressed as Z-linear combination of b_s,k} combined with Proposition~\ref{prop: first estimates of b_s,k} implies that
    \begin{align*}
        |2r_{k,s}|_p \leq \max_{0\leq i \leq \lfloor c(k)/2 \rfloor}\{ |b_{s,i}|_p \} \leq p^{(s+1)\log_p(2c(k)+1)} = p^{\GrO(\log k)}
    \end{align*}
    where the implicit constant in $\GrO(\log k)$ depends on $n$ and $m$. Together with
    \begin{equation*}
        \left|\frac{n!^2}{(x)_{k+1}(1-x)_{k+1}} \right|_p \left|\binom{k}{n}\right|_p  \le |n!|_p^2 |x|_p^{-2(k+1)} \leq p^{-2k},
    \end{equation*}
    this yields the convergence of the series $\fR_{m,n,s,\ell}(x)$ in $\Q_p$. Moreover
    \begin{align*}
        \bigl|2\fR_{m,n,s,\ell}(x)\bigr|_p   \le
           |n!|_p^2 \, |x|_p^{-2n} \cdot \max_{k \ge n} \left\{p^{(s+1)\log_p(c(k))}|x|_p^{-2(k-n+1)} \right\}
         = e^{o(n)}\, |n!|_p^2 \, |x|_p^{-2n},
    \end{align*}
    as $n$ tends to $\infty$. We conclude by raising both sides to the power $1/n$ and by using the classical estimate
    \[
    \lim_{n \to \infty} |n!|_p^{\,1/n}
       \le p^{-\tfrac{1}{p-1}}.
    \]
\end{proof}

\section{Proof of our main theorem} \label{section: Proof of theorem}

We keep the notation of Section~\ref{section: Construction of Pade}. Let $n,m\geq 0$ be integers with $n$ even and $m\geq 3$. Recall that if $m$ is even (resp. odd), then $S_m$ is the subset of $\{2,\dots,m-1\}$ formed by the odd (resp. even) indices. Given $s\in S_m$ and $0\leq \ell \leq M:= \lfloor (m-1)/2 \rfloor$, the polynomials $P_{m,n,\ell}(z), Q_{m,n,s,\ell}(z)$ and the remainder $\fR_{m,n,s,\ell}(z)=P_{m,n,\ell}(z)F_s(z)-Q_{m,n,s,\ell}(z)$ are as in Definition~\ref{def: construction Pade approximants, choice of A_ell}.\\
Theorem~\ref{main theorem} will appear as a corollary of the more general Theorem~\ref{main theorem: general version} below. Given a rational number $x$, the quantity
\begin{align*}
      \mu(x) & = \displaystyle \den(x)\prod_{\substack{q:{\rm{prime}} \\ q|{\den}(x)}}q^{\tfrac{1}{q-1}} 
\end{align*}
is defined as in Subsection~\ref{subsection: estimate denominator}. Recall that given an integer $m\geq 3$, the quantity $f(m)$ is defined by \eqref{eq: def f(m)}.

\begin{theorem} \label{main theorem: general version}
    Let $p$ be a prime number and $x\in \Q$ with $|x|_p>1$. Let $m\ge 3$ be an integer such that
    \begin{align} \label{hyp p,x}
         f(m) < 2(M+1)\Big(\log|x|_p+\dfrac{\log p}{p-1}\Big) - 2M\log \mu(x).
    \end{align}
    Then, writing $S_m = \{s_1,\dots,s_M\}$, the $M+1$ elements $1$, $F_{s_1}(x)$, $\dots$, $F_{s_M}(x)\in \Q_p$ are linearly independent over $\Q$.
\end{theorem}

We will deduce Theorem~\ref{main theorem: general version} from the next one, which is an analog of \cite[Theorem 11.4]{KP2}.

\begin{theorem}\label{thm: linear indpce generic values}
    Let $p$ be a prime number and $x\in \Q$ with $|x|_p>1$, and fix an integer $m\ge 3$. Let $I$ be an infinite subset of $\Z_{\geq 0}$, and let $\beta,\rho_\infty,\rho_p,\delta$ be real numbers and $(D_n)_{n\geq 0}$ be a sequence of positive integers such that, for each $s\in S_m$ and each $\ell=0,\dots,M$,
    \begin{align*}
        & D_n P_{m,n,\ell}(x), \; D_n Q_{m,n,s,\ell}(x) \in \Z \AND  \limsup_{\substack{n\to\infty \\ n\in I}} \max\big\{|P_{m,n,\ell}(x)|, |Q_{m,n,s,\ell}(x)|\big\}^{1/n} \leq e^\beta, \\
        & \limsup_{\substack{n\to\infty \\ n\in I}} D_n^{1/n} \leq e^{\rho_\infty} \AND \limsup_{\substack{n\to\infty \\ n\in I}} |D_n|_p^{1/n} \leq e^{-\rho_p}, \\
        & \limsup_{\substack{n\to\infty \\ n\in I}} |\fR_{m,n,s,\ell}(x)|_p^{1/n} \leq e^{-\delta}.
    \end{align*}
    Suppose that $\beta+\rho_\infty < \delta + \rho_p$. Then the $M+1$ elements $1$, $F_{s_1}(x)$, $\dots$, $F_{s_M}(x)\in \Q_p$ are linearly independent over $\Q$.
\end{theorem}

\begin{remark}
    The series $\fR_{m,n,s,\ell}(x)$ converges in $\Q_p$ by Proposition~\ref{remainder est}.
\end{remark}

\begin{proof}
    The proof is classical, see for example \cite[Proposition $5.7$]{DHK2} (in the case $K=\Q$, $v_0=p$), which also provides an effective irrationality measure result, or \cite[Theorem 11.4]{KP2}. For the sake of completion we recall the arguments.\\
    By contradiction, suppose that the coordinates of the $(M+1)$--tuple $(1,F_{s}(x))_{s\in S_m}$ are linearly dependent over $\Q$. Then, there exists $(b,b_{s})_{s\in S_m}\in\Z^{M+1}\setminus\{\boldsymbol{0}\}$ such that
    \begin{align}
        \label{eq proof: main thm: eq1}
        b+\sum_{s \in S_m} b_{s} F_{s}(x) = 0.
    \end{align}
    Given $n\in I$, define the integers $\hp_{m,n,\ell} := D_nP_{m,n,\ell}(x)$ and $\hq_{m,n,s,\ell} := D_n Q_{m,n,s,\ell}(x)$ for each $\ell=0,\dots,M$. Theorem~\ref{non zero det} implies that the $(M+1)\times (M+1)$ matrix whose columns are $(\hp_{m,n,\ell}, \hq_{m,n,s,\ell})_{s\in S_m}$ for $\ell=0,\dots,M$,
    is non-singular. Consequently, there exists an integer $\ell$ with $0\leq \ell \leq M$ such that
    \begin{align*}
        K_n = K_n(\ell) := b \hp_{m,n,\ell} + \sum_{s\in S_m} b_{s} \hq_{m,n,s,\ell} \textrm{ is a non-zero integer.}
    \end{align*}
    By hypothesis, we have $\limsup_{\substack{n\to \infty \\ n\in I}} |K_n|^{1/n} \leq e^{\beta + \rho_\infty}$. On the other hand, writing
    \begin{align*}
        K_n = K_n - \hp_{m,n,\ell}\Big(b+\sum_{s \in S_m} b_{s} F_{s}(x)\Big) 
         = \sum_{s \in S_m} b_{s} D_n \fR_{m,n,s,\ell}(x),
    \end{align*}
    we find $\limsup_{\substack{n\to \infty \\ \textrm{$n\in I$}}}|K_n|_p^{1/n} \leq e^{-\rho_p-\delta}$. Since $|K_n||K_n|_p\geq 1$, it follows that $\beta + \rho_\infty \geq \delta + \rho_p$, a contradiction.
\end{proof}

\begin{proof}[\textbf{Proof of Theorem~\ref{main theorem: general version}}]
    Denote by $I$ the set of positive even integers. Let $n$ be an integer and denote by $D_n(x)$ the integer defined as in Proposition~\ref{den ask}. According to the aforementioned proposition, $D_n(x)P_{m,n,\ell}(x)$ and $D_n(x) Q_{m,n,s,\ell}(x)$ are integers, and
    \begin{align*}
        \limsup_{\substack{n\to\infty \\ n\in I}} D_n^{1/n} \leq e^{\rho_\infty} \AND \limsup_{\substack{n\to\infty \\ n\in I}} |D_n|_p^{1/n} \leq e^{-\rho_p},
    \end{align*}
    where,
    \begin{align*}
        \rho_\infty &= 2M\log\mu(x) + (m-1) + \sum_{j=1}^{M-1}\frac{2M}{j}, \\
        \rho_p &= 2M\left(\log |x|_p+\frac{\log p}{p-1}\right).
    \end{align*}
    On the other hand Proposition~\ref{prop: estimate Pade via auxiliary seq} yields
    \[
        \limsup_{\substack{n \to \infty \\ n\in I}}  \max_{s\in S_m}
        \bigl\{|P_{n,\ell}(x)|,\, |Q_{n,s,\ell}(x)| \bigr\}^{1/n} \le e^\beta,\qquad \textrm{where } \beta = \log(C(m)).
    \]
    where $C(m)$ is as in the statement of Theorem~\ref{thm: property of the coeff of P_n}. In the case $m=4$ and $m=6$, we may take $\beta=0$, see Proposition~\ref{prop: estimate Pade via m=4 and m=6, Perron rec}. Finally, Proposition~\ref{remainder est} yields
    \begin{align*}
        \limsup_{\substack{n\to \infty \\ n \in I}} |\fR_{m,n,s,\ell}(x)|_p^{1/n} \leq e^{-\delta}, \qquad \textrm{where } \delta =
        \ 2\Big(\log|x|_p+\dfrac{\log p}{p-1}\Big).
    \end{align*}
    We conclude by Theorem~\ref{thm: linear indpce generic values}.
\end{proof}

\begin{proof}[\textbf{Proof of Theorem~\ref{main theorem}}]
    Let $p$ be a prime number and $s\geq 2$ be an integer. Combining \eqref{eq: link zeta_p and derivative of G_p} of Lemma~\ref{expansion zeta_p} with Lemma~\ref{lem: series expansion Theta_s and R_s} and, we find for any $x\in \Q_p$ with $|x|_p > 1$:
    \begin{align*}
        \frac{ (-1)^{s-1}\omega(x)^{s-1}}{(s-1)!} R_s(x) &=\zeta_p(s,x),
    \end{align*}
    where 
    $\omega$ denotes the Teichmüller character (see Section~\ref{section: notation}). Furthermore, writing $x = a/p^r$ with $|a|_p = 1$ and $r\in\Z_{\geq 1}$, we have
    \begin{align*}
        \omega\left(\frac{x+1}{2} \right) = p^{-r}\omega\left(\frac{a+p^r}{2} \right) =  p^{-r}\omega\left(\frac{a}{2} \right) =  \omega\left(\frac{x}{2} \right).
    \end{align*}
    So, combining the above with Lemma~\ref{identities}~\ref{assertion identities 5}, we also deduce that
    \begin{align*}
        \frac{ (-2)^{s-1}\omega(x/2)^{s-1}}{(s-1)!}\Theta_s(x) = \zeta_p\left(s,\frac{x+1}{2}\right) - \zeta_p\left(s,\frac{x}{2} \right),
    \end{align*}
    Taking $x = p^{-r}$ (resp. $x = 2p^{-r}$) with $r\geq \Z_{\geq 1}$, we have $\omega(x) = p^{-r}\in\Q$ (resp. $\omega(x/2) = p^{-r} \in \Q$). Consequently, there exists $\lambda_s,\mu_s\in\Q^\times$ such that
    \begin{align*}
        \zeta_p(s,p^{-r}) = \lambda_s R_s(p^{-r})   \AND \zeta_p(s,p^{-r} + 1/2) - \zeta_p(s,p^{-r}) = \mu_s\Theta_s(2p^{-r}).
    \end{align*}
    We conclude by applying Theorem~\ref{main theorem: general version} with $x=p^{-r}$ for even $m$ (resp. $x=2p^{-r}$ for odd $m$), and by using the identities
    \[
        |p^{-r}|_p = |2p^{-r}|_p =p^r \AND \mu(p^{-r}) = \mu(2p^{-r}) = p^{r+\tfrac{1}{p-1}}.
    \]
\end{proof}

\begin{proof}[{\bf Proof of Corollary~\ref{cor: irrationality zeta_p(3)}}]
    According to \cite[Proposition 5.2]{B}, we have
    \begin{align*}
        R_3(1/4) &= 4^3\zeta_2(3) \qquad (\textrm{for $p=2$}),\\
        R_3(1/3) &= 3^3\zeta_3(3) \qquad (\textrm{for $p=3$})
    \end{align*}
(note that in \cite{B}, we have $T(z) = R_3(z)$ and $T_p(x)$ denotes the $p$-adic series $R_3(x)$). Theorem~\ref{main theorem: general version} with $m=4$, $p=2$ and $x=1/p^2$ (resp. $p=3$ and $x=1/p$) yields $R_3(1/4)\in \Q_2\setminus \Q$ (resp. $R_3(1/3)\in\Q_3\setminus\Q$).
\end{proof}

\section{Appendix: recurrence relation} \label{section: appendix rec}

We keep the notation of Section~\ref{section: Construction of Pade}. 
The polynomials $P_{m,n,\ell}(z)$ and $Q_{m,n,s,\ell}(z)$ are as in Definition~\ref{def: construction Pade approximants, choice of A_ell}. As usual, if $m$ is even (resp. odd), then the set $S_m \subset \{2,\dots,m-1\}$ is the set of odd (resp. even) indices in $\{2,\dots,m-1\}$. The main application of this section is Proposition~\ref{prop: estimate Pade via m=4 and m=6, Perron rec} below, and is proved in Subsection~\ref{subsection: case m=6}. This is an improvement of Proposition~\ref{prop: estimate Pade via auxiliary seq} in the special cases $m=4$ and $m=6$. 
In Subsection~\ref{subsection: existence of matrix rec Perron}, we prove that for general $m$, the vector
\[
    \big(P_{m,n,M-1}(z),\dots,P_{m,n,0}(z), P_{m,n-1,M-1}(z), \dots,P_{m,n-1,0}(z)\big)
\]
satisfies a certain Poincaré-Perron difference system.
In Section~\ref{section: estimates}, we just need to control the $P_{m,n,\ell}(x)$ and $Q_{m,n,s,\ell}(x)$ for $\ell < M$. However, in this section we will also need to control these quantities for $\ell=M$. We deal with this issue in Subsection~\ref{subsection: control of the pol for ell = M}.

\begin{proposition} \label{prop: estimate Pade via m=4 and m=6, Perron rec}
    Let $x \in \Q$ and $m\in\{4,6\}$.  For each integer $\ell= 0,\dots, M:= (m-2)/2$, we have
    \[
        \limsup_{n \to \infty } \max_{s\in S_m}
        \bigl\{ |P_{n,\ell}(x)|,\, |Q_{m,n,s,\ell}(x)| \bigr\}^{1/n}
        \le  1.
    \]
\end{proposition}

The proof is based on (the matrix version of) Poincaré-Perron theorem, see Subsection~\ref{subsection: case m=6}.

\subsection{Existence of a recurrence (general case)}\label{subsection: existence of matrix rec Perron}

Recall that for each integer $k\geq 0$, we have $T_k(z) = (z-k/2)_k$. Fix an integer $m\geq 3$, and for simplicity, write $P_{n,\ell} = P_{m,n,\ell}$ and $Q_{n,s,\ell} = Q_{m,n,s,\ell}$ for each $n,\ell\geq 0$ and $s\in S_m$. We also set
\begin{align*}
    Q_{n,\infty,\ell}(z) = P_{n,\ell}(z).
\end{align*}
Recall that $P_{n,\ell}(z)$ has degree $2Mn+2\ell$, where $M = \left\lfloor (m-1)/2 \right\rfloor$, and that Proposition~\ref{linear relations}\ref{prop: pade cond: item 2} and Theorem~\ref{Pade} implies that for each $n\geq 1$,
\begin{align}\label{eq: simplified kernel useful for proof}
    \phi_s\big(t^k P_{n,\ell}(t)\big) = 0 \qquad \textrm{for $s\in S_m$ and $0 \leq k \leq 2n-1$.}
\end{align}

\begin{theorem}\label{thm: general rec matrix Perron}
    For each $n\geq 2$, there exists a $(2M+1)\times (2M+1)$ upper triangular matrix $A(n)$, with coefficients in $\Q$, whose first $M$ diagonal entries are non-zero, such that, for each $s\in S_m\cup \{\infty\}$, we have
    \begin{align}\label{eq: matrix rec relation 1}
        \left(\begin{matrix}
            T_2(z)Q_{n,s,M-1}(z) \\
            \vdots \\
            T_2(z) Q_{n,s,0}(z) \\
            0 \\
            \vdots \\
            0
        \end{matrix}
        \right)
        = A(n)\left(\begin{matrix}
            Q_{n+1,s,0}(z) \\
            Q_{n,s,M-1}(z) \\
            \vdots \\
            Q_{n,s,0}(z) \\
            Q_{n-1,s,M-1}(z) \\
            \vdots \\
            Q_{n-1,s,0}(z)
        \end{matrix} \right).
    \end{align}
\end{theorem}

\begin{remark}
    Once we know that a recurrence \eqref{eq: matrix rec relation 1} exists, finding the coefficients of $A(n)$ becomes a pure problem of linear algebra. We may also deduce from the recurrence \eqref{eq: matrix rec relation 1} a Poincaré-Perron recurrence of the form
    \begin{align*}
        Z_{n+1,s} = A(n,z)Z_{n,s} \qquad (n\geq 2).
    \end{align*}
    where $ Z_{n,s} = {}^t \Big(Q_{n,s,M-1}(z),\dots, Q_{n,s,0}(z), Q_{n-1,s,M-1}(z),\dots, Q_{n-1,s,0}(z)\Big)$, and $A(n,z)$ is a $2M \times 2M$ matrix whose coefficients are rational polynomials in $z$.
\end{remark}

\begin{proof}[{\bf Proof of Theorem~\ref{thm: general rec matrix Perron}}]
    It suffices to treat the case $s=\infty$, since for each $s\in S_m$
    \begin{align*}
        Q_{n,s,\ell}(z) = \phi_s\left(\frac{P_{n,\ell}(z)-P_{n,\ell}(t)}{z-t} \right) \AND  T_2(z)Q_{n,s,\ell}(z) = \phi_s\left(\frac{T_2(z)P_{n,\ell}(z)- T_2(t)P_{n,\ell}(t)}{z-t} \right),
    \end{align*}
    (the last equality coming from the fact that $(T_2(z)-T_2(t))P_{n,\ell}(t)/(z-t) \in \ker \phi_s$ by \eqref{eq: simplified kernel useful for proof}).\\
    The polynomials $\{P_{k,0}(z),\dots,P_{k,M-1}(z)\}_{0 \leq k < n}$ form a basis of $V_1 \cap \Q[z]_{< 2Mn}$, where the eigenspace $V_1 = \{ f \;;\; f(z) = f(1-z)\}$ is defined as in Section~\ref{involution}. In particular, given $\ell\in\{0,\dots,M-1\}$, there exist coefficients $\lambda_{i,j}\in \Q$, such that,
    \begin{align}
        T_{2}(z) P_{n,\ell}(z) & = \lambda_{n+1,0} P_{n+1,0}(z) + \sum_{i=0}^{n} \sum_{j=0}^{M-1} \lambda_{i,j} P_{i,j}(z).  \label{eq: expansion for rec 1}
    \end{align}
    Considering the degree of $T_{2}(z) P_{n,\ell}(z)$, we must have $\lambda_{n+1,0}\neq 0$ if $\ell=M-1$. Similarly, if $\ell < M-1$, then $\lambda_{n+1,0} = \lambda_{n,M-1} = \cdots = \lambda_{n,\ell+2} = 0$ and $\lambda_{n,\ell+1}\neq 0$. So, it just remains to prove that
    \begin{align} \label{eq proof: rec matrix inter}
        \lambda_{i,j} = 0 \qquad \textrm{for $i = 0,\dots,n-2$ and $0\leq j < M$}.
    \end{align}
    According to \eqref{eq: simplified kernel useful for proof}, for each integer $i$ with $i\geq k+1$
    \[
        T_{2k+1}(t)P_{i,j}(t) \in \bigcap_{s\in S_m} \ker \phi_s \qquad (0\leq j < M).
    \]
    Let $k\in\{0,\dots,n-2\}$ and suppose that for each $i$ in $\{0,\dots,k-1\}$, we have $\lambda_{i,0} = \cdots =\lambda_{i,M-1} = 0$ (if $k=0$, this is automatic). Multiplying  each side of \eqref{eq: expansion for rec 1} by $T_{2k+1}(t)$, and since for $i=k+1,\dots,n+1$ and $j=0,\dots,M-1$, Eq.~\eqref{eq: simplified kernel useful for proof} yields
    \[
        T_{2k+1}(t)T_2(t)P_{n,\ell}(t) \in \bigcap_{s\in S_m} \ker \phi_s \AND T_{2k+1}(t)P_{i,j}(t) \in \bigcap_{s\in S_m} \ker \phi_s,
    \]
    we deduce that,
    \begin{align*}
        0 = & \sum_{j=0}^{M-1} \lambda_{k,j}\phi_s\Big(T_{2k+1}(t) P_{k,j}(z)\Big).
    \end{align*}
    By Corollary~\ref{cor: non-sing matrix, useful for rec}, this yields $\lambda_{k,j} =0$ for $j=0,\dots,M-1$. Hence \eqref{eq proof: rec matrix inter}.
\end{proof}

\subsection{Formula for the polynomials $Q_{n,s,M}(z)$} \label{subsection: control of the pol for ell = M}

We keep the notation of Section~\ref{section: Construction of Pade} (for the operators $R_{\pm 1}$ etc.) and subsection~\ref{subsection: existence of matrix rec Perron}. Theorem~\ref{thm: general rec matrix Perron} allows to control the asymptotic behavior of $Q_{n,s,\ell}(x)$ for $s\in S_m\cup\{\infty\}$ and $\ell=0,\dots, M-1$. However, in the proof of our main theorem we also need to control $Q_{n,s,M}(x)$ . Proposition~\ref{prop: case ell = M} ensures that this will be automatic.

\begin{proposition}\label{prop: case ell = M}
    Let $m\geq 3$ be an integer. There exist an integer $d\geq 0$ and $(q_i(n))_{n\geq 2M}$ satisfying $q_i(n) = \GrO(n^d)$ for $i=0,\dots,M$, such that for each $s\in S_m \cup \{\infty\}$ and each $n\geq 2M$,
    \begin{align*}
        Q_{n,s,M}(z) = q_M(n) Q_{n+1,s,0}(z) + \sum_{i=0}^{M-1} q_{i}(n) Q_{n,s,i}(z).
    \end{align*}
    In particular, for each $x\in\Q$, we have
    \begin{align*}
        \limsup_{n\to\infty} |Q_{n,s,M}(x)|^{1/n} \leq \max_{0\leq \ell < M} |Q_{n,s,\ell}(x)|^{1/n}.
    \end{align*}
\end{proposition}

\begin{proof}
    In view of the definition of $Q_{m,n,s,\ell}(z)$ (for $s\in S_m$), it suffices treat the case of $s=\infty$. We denote by $A_{n,\ell}(z)$ the polynomial $n!^{m-1}A_\ell(z)$ of Definition~\ref{def: construction Pade approximants, choice of A_ell}. Writing $\ee = (-1)^{m-1}$, we have
    \begin{align}\label{eq proof: inter technique 0}
        R_\ee(T_{n+1}(z)^m) = T_n(z)^m Q_n(z), \qquad \textrm{where }  Q_n(z) = \frac{\left(z+\frac{n}{2} \right)^m - \left(1-z+\frac{n}{2}\right)^m}{2z-1} \in V_1 \cap \Q[z].
    \end{align}
    The polynomial $Q_n(z)\in V_1$ has degree $2M$, and $A_{n,0}(z),\dots, A_{n,M}(z)$ form a basis of $V_1 \cap \Q[z]_{\leq 2M}$, so there exist coefficients  $\lambda_\ell(n)\in \Q$ for $\ell=0,\dots,M$, with $\lambda_M(n)\neq 0$, such that
    \begin{align}\label{eq proof: inter technique}
        Q_n(z)  = \sum_{\ell=0}^{M} \lambda_\ell(n) A_{n,\ell}(z).
    \end{align}
    Note that we can easily compute $\lambda_M(n)$ by looking at the leading coefficients: we find $\lambda_M(n) = 1$ is $m$ if odd, and $\lambda_M(n) = m(n+1)/2$ when $m$ is even. More generally, we can find the coefficients $\lambda_\ell(n)$ by induction. Suppose first that $m$ is even and $n\geq 2M$. Then $0,-1,\dots,-(\ell-1)$ are simple roots of $A_\ell(z+n/2)$ and $A_\ell(-\ell+n/2)\neq 0$. Therefore,
    \begin{align*}
        \lambda_\ell(n) = A_{n,\ell}\left(-\ell + \frac{n}{2}\right)^{-1}\left( Q_n\left(-\ell + \frac{n}{2}\right) - \sum_{j=0}^{\ell-1} \lambda_\ell(n) A_{n,j}\left(-\ell + \frac{n}{2}\right) \right) \qquad (\textrm{for } \ell=0,\dots,M).
    \end{align*}
    The above relation implies that there exist absolute constants $d,C>0$ such that
    \begin{align*}
        |\lambda_\ell(n)| \leq C n^d.
    \end{align*}
    When $m$ is odd, the reasoning is similar (we evaluate at $z = \ell$ instead of $z = -\ell+n/2$). By definition of the polynomials $P_{m,n,\ell}(z)$, Eqs.~\eqref{eq proof: inter technique 0} and~\eqref{eq proof: inter technique} implies that
    \begin{align*}
        P_{m,n+1,0}(z) = R_\ee^{n+1}\left((n+1)!^{-m+1}T_{n+1}(z)^m \right) & =  R_\ee^{n}\left((n+1)!^{-m+1}\sum_{\ell=0}^{M} \lambda_\ell(n) A_{n,\ell}(z)T_{n}(z)^m \right) \\
        & = \sum_{\ell=0}^{M} \lambda_\ell(n) \left(\frac{n}{n+1}\right)^{m-1} P_{m,n,\ell}(z),
    \end{align*}
    and the conclusion follows easily.
\end{proof}

\subsection{Proof of Proposition~\ref{prop: estimate Pade via m=4 and m=6, Perron rec}} \label{subsection: case m=6}

\noindent{\bf Case $m=4$.} In that case $M=1$, and given $n\geq 0$, we have
\begin{align*}
    P_{4,n,0}(z) = \sum_{k=0}^{n} (-1)^k a_k^{(4)}(n) \dfrac{(z)_{k}}{k!}\dfrac{(1-z)_{k}}{k!} \qquad \textrm{where } a_k^{(4)}(n) = \binom{n}{k}  \binom{n+k}{n},
\end{align*}
see the remark at the end of Subsection~\ref{subsection: even case}. Zeilberger's algorithm (see \cite{ZeilbergerAeqB}), ensures that $( P_{4,n,0}(z))_n$ satisfies the Poincaré-Perron recurrence
\begin{align*}
    u_2(n) P_{n+2}(z) +u_1(n) P_{n+1}(z) + u_0(n) P_n(z) = 0 \qquad (n\geq 0),
\end{align*}
where
\begin{align*}
    u_2(n) &= n^3+6n^2+12n+8, \\
    u_1(n) &= -(2n^3+9n^2+(-4z+15+4z^2)n-6z+9+6z^2), \\
    u_0(n) &= n^3+3n^2+3n+1.
\end{align*}
The sequence $(Q_{4,n,s,0}(z))_{n\geq 1}$ follows the same recurrence for $s\in S_4 = \{3\}$. The characteristic polynomial of the above recurrence is $X^2-2X+1 = (X-1)^2$. By Poincaré-Perron theorem (see \cite{Perron} and \cite[Theorem C]{Pituk}, we deduce that, for $\ell=0$ and for each $x\in\Q$,
\begin{align*}
    \limsup_{n\to\infty} \max\{|P_{4,n,\ell}(x)|, |Q_{4,n,3,\ell}(x)|\}^{1/n} \leq 1.
\end{align*}
By Proposition~\ref{prop: case ell = M}, the same estimate holds for $\ell=1$.

\bigskip

\noindent{\bf Case $m=6$.} For simplicity, we write $P_{n,\ell}(z) = P_{6,n,\ell}(z)$. According to Section~\ref{section: estimates bis}, for each integer $n\geq 0$ and $\ell=0,1$, we have
\begin{align*}
    P_{n,\ell}(z) = \sum_{j=0}^{2n+\ell} a_{j,\ell}^{(6)}(n)\dfrac{T_{2j}(z)}{j!^2},
\end{align*}
where as always $T_{2j}(z) = (z-j)_{2j}$. The coefficients $a_{j,\ell}^{(6)}$ are given by the following formulas. For $j=0,\dots,2n$
\begin{align*}
    a_{2n-j,0}^{(6)} & = \sum_{k=\max\{0,n-j\}}^{\min\{2n-j,n\}} \binom{n}{k} \binom{n+k}{n}^2 \binom{n+k}{k+j-n}  \binom{3n-j}{n+k},
\end{align*}
and for $i=0,\dots,2n+1$, we have
 \begin{align*}
    a_{2n+1-i,1}^{(6)}(n)
    & = -(3n+1-i)\sum_{j=\max\{0,i-1\}}^{\min\{2n,i\}} a_{2n-j,0}^{(6)}(n) \binom{2n-2j+i}{2n-j}\binom{2n+1-j}{j+1-i}
\end{align*}
(see the proof of Proposition~\ref{prop: rec coeff H^m even case gen case}). Additionally, using Lemma~\ref{lem: rec coeff H^m even case} (with $\alpha=0$) in order to expand the terms of the form $T_{2}(z)T_{2j}(z)$, we obtain the following lemma (details are left to the reader).

\begin{lemma}
    We have
    \begin{align*}
        T_{2}(z)P_{n,1}(z) = \sum_{i=1}^{2n+2} b_i^{(0)}(n) \frac{T_{2i}(z)}{i!^2} \AND
        T_{2}(z)P_{n+1,0}(z) = \sum_{i=1}^{2n+3} b_i^{(1)}(n) \frac{T_{2i}(z)}{i!^2}
    \end{align*}
    where for $i=0,\dots,2n+1$, resp. $i=0,\dots,2n+2$,
    \begin{align*}
        b_{2n+2-i}^{(0)}(n) &= \sum_{j=\max\{0,i-1\}}^{i} (2n+2-i) a_{2n+1-j,1}^{(6)}(n) \binom{2n+2-j}{i-j}\binom{2n+2-i}{j+1-i},\\
        \textrm{resp. } b_{2n+3-i}^{(1)}(n) &= \sum_{j=\max\{0,i-1\}}^{i} (2n+3-i) a_{2n+2-j,0}^{(6)}(n+1) \binom{2n+3-j}{i-j}\binom{2n+3-i}{j+1-i}.
    \end{align*}
\end{lemma}

\noindent{\bf Computation of the coefficients of the recurrence.} Define the vector
\begin{align*}
    \bP(n) = \left(P_{n+1,1}(z),(n+1)^2 P_{n+1,0}(z), P_{n,1}(z), n^2P_{n,0}(z), P_{n-1,1}(z), (n-1)^2P_{n-1,0}(z) \right).
\end{align*}
If $\bx = (x_i)_i$ and $\by= (y_i)_i$ are two points of the same finite length, we write $\bx\cdot\by = \sum_i x_iy_i$ . According to Theorem~\ref{thm: general rec matrix Perron}, there exist $\lambda(n),\mu(n) \in \Q^6$ such that
\begin{align*}
    T_{2}(z)(n+1)^2P_{n+1,0}(z) = \lambda(n)\cdot \bP(n)  \AND   T_{2}(z)P_{n,1}(z) = \mu(n) \cdot \bP(n).
\end{align*}
Finding the coefficients $\lambda(n)$ and $\mu(n)$ is a linear algebra problem that we solve by identification (writing all of these polynomials along the basis $(T_{2j}(z))_{j\geq 0}$, and then looking at the last four or five terms of highest degree). Using Maple, we find that
\begin{align*}
    \lambda(n) = \left(\begin{matrix}
      \displaystyle -\frac{(2n+3)(n+1)^2}{3n+4} \\
      \displaystyle -\frac{(5n^2+16n+14)(n+1)}{3(3n+4)} \\
      \displaystyle -\frac{14(n+1)^3}{3(3n+4)} \\
      \displaystyle \frac{(5n^2-4n+5)(n+1)^3}{3n^2(3n+4)} \\
      \displaystyle 0 \\
      \displaystyle 0
    \end{matrix}\right),
    \quad
    \mu(n) = \left(\begin{matrix}
      \displaystyle 0 \\
      \displaystyle -\frac{2(n+1)^3}{3(3n+2)} \\
      \displaystyle -\frac{13n^3+18n^2+15n+4}{3(3n+1)} \\
      \displaystyle -\frac{2(2n+1)(n+1)(n^2+n+1)(2n^2-1)}{3n^2(3n+2)(3n+1)} \\
      \displaystyle -\frac{n(14+28n+41n^3+37n^2)}{3(3n+2)(3n+1)} \\
      \displaystyle \frac{14n(n+n^2+n^5-2n^4+n^3+1)}{3(3n+2)(n-1)^2(3n+1)}
    \end{matrix}\right).
\end{align*}
After some rearrangements, we finally obtain
\begin{align*}
    X(n+1) = A(n,z)X(n),
\end{align*}
where
\begin{align*}
    X(n) = \left(\begin{matrix}
      P_{n,1}(z) \\
      n^2P_{n,0}(z) \\
      P_{n-1,1}(z) \\
      (n-1)^2P_{n-1,0}(z)
    \end{matrix}\right)
    \AND
    A(n,z) =
    \left(\begin{matrix}
      \frac{37}{12} & \frac{35}{18} & \frac{205}{36} & -\frac{35}{18}  \\
      -\frac{13}{2} & -\frac{4}{3} & -\frac{41}{6} & \frac{7}{3}  \\
      1 & 0 & 0 & 0 \\
      0 & 1 & 0 & 0
    \end{matrix}\right)+o(1)
\end{align*}
as $n$ tends to infinity. The characteristic polynomial of the limit matrix is
\begin{align*}
    P(X) = \frac{1}{4}X(4X+1)(X-1)^2.
\end{align*}
According to (the matrix version of) Poincaré-Perron theorem (see \cite[Theorem 1]{Pituk}), we have
\begin{align*}
    \limsup_{n\to\infty} \| X(n)\|^{1/n} \leq 1,
\end{align*}
where $\|\cdot \|$ denotes the supremum norm on $\Q^4$. Combined with Proposition~\ref{prop: case ell = M}, we deduce that for $\ell=0,1,2$,
\begin{align*}
    \limsup_{n\to\infty} \max\{|P_{6,n,\ell}(x)|, |Q_{6,n,3,\ell}(x)|, |Q_{6,n,5,\ell}(x)|\}^{1/n} \leq 1.
\end{align*}

\bigskip
\bigskip

\noindent
{\bf Acknowledgements.} Parts of this work were written while the first named author was visiting Universit\'{e} Claude Bernard Lyon $1$ in November 2025. He would like to thank them for their hospitality.
This work is partially supported by the Research Institute for Mathematical Sciences, an international joint usage and research center located at Kyoto University. It is also partly funded by the National Research Agency (ANR) under project ``ANR-25-CE40-1961-01''.
The first named author is supported by JSPS KAKENHI Grant Number JP24K16905.

\medskip

\noindent{\bf Generative-AI declaration}. This project started in 2024 and is a natural extension of our previous work \cite{KP2}. The authors declare that no dedicated generative AI tools\footnote{we do not take into account AI-powered features embedded in general-purpose services, such as search engines or translation tools.} were used in the research, analysis, or writing of this article, except for this statement and to translate certain parts of the introduction into English.


\bibliography{}

\footnotesize {

\begin{scriptsize}
    \begin{minipage}[t]{0.38\textwidth}
        Makoto Kawashima
        \\kawasima@mi.meijigakuin.ac.jp
        \\Institute for Mathematical Informatics,
        \\Meiji Gakuin University,
        \\Totsuka, Yokohama, Kanagawa
        \\224-8539, Japan\\\\
    \end{minipage}
    \begin{minipage}[t]{0.32\textwidth}
        Anthony Po\"{e}ls
        \\ poels@math.univ-lyon1.fr
        \\Universite Claude Bernard Lyon 1,
        \\Institut Camille Jordan UMR 5208,
        \\69622 Villeurbanne, France\\\\
    \end{minipage}
\end{scriptsize}

\end{document}